%% file: main-v2-dg.tex
\documentclass[10pt]{article}
\usepackage[margin=1in]{geometry}

\usepackage{graphicx} 
\usepackage{amssymb}  
\usepackage{amsthm}

\usepackage{amsmath,mathtools}

\usepackage{color}
\allowdisplaybreaks[1]

\usepackage{float}
\usepackage{mwe}
\usepackage[latin1]{inputenc}
\usepackage{enumerate}
\usepackage[shortlabels]{enumitem}
\usepackage{algorithm}
\usepackage{algpseudocode}
\usepackage{caption,subcaption}%

\usepackage{tikz}
\usetikzlibrary{shapes.geometric, arrows}
\usetikzlibrary{positioning}
\usetikzlibrary{arrows,calc}
\usepackage{relsize}
\usepackage[colorlinks=true,linkcolor=blue]{hyperref}%

\usepackage{biblatex}

\newcommand\numberthis{\addtocounter{equation}{1}\tag{\theequation}}

\theoremstyle{plain}  
\newtheorem{theorem}{Theorem}[section]

\newtheorem{defn}[theorem]{Definition}
\newtheorem{lemma}[theorem]{Lemma}
\newtheorem{corollary}[theorem]{Corollary}
\newtheorem{proposition}[theorem]{Proposition}
\newtheorem{conjecture}[theorem]{Conjecture}

\theoremstyle{definition}

\theoremstyle{remark} 

\newcommand{\p}{\mathbb{P}}
\newcommand{\E}{\mathbb{E}}

\newcommand{\ER}{Erd{\"o}s-R\'{e}nyi }

\newcommand{\Bern}{\ensuremath{\operatorname{Bern}}\xspace}

\def\0{{\bf 0}}
\def\1{{\bf 1}}

\def \R{ \mathbb{R}}

\def \G{ \mathbb{G}}
\def \A{ \mathcal{A}}

\def \Nc{ \mathcal{N}}
\def \Rc{ \mathcal{R}}

\def \bmat{\left[\begin{matrix}}
	\def \emat{\end{matrix}\right]}

\def \xy1vec{\left[\begin{matrix}x\\y\\1\end{matrix}\right]}
\def \QED{\begin{flushright}\Halmos\end{flushright}\end{proof}}

\long\def\old#1{}

\definecolor{DarkerGreen}{RGB}{0,170,0}

\definecolor{orange}{rgb}{1,0.5,0}

\renewcommand\and{\end{tabular}\kern-\tabcolsep\ and\ \kern-\tabcolsep\begin{tabular}[t]{c}}
\let\origthanks\thanks
\renewcommand\thanks[1]{\begingroup\let\rlap\relax\origthanks{#1}\endgroup}

\title{\LARGE \bf Densest Subgraphs of a  Dense Erd\"os-R\'enyi Graph. Asymptotics, Landscape and Universality}
\author{Houssam El Cheairi\thanks{MIT. Email: \texttt{houssamc@mit.edu}}, David Gamarnik\thanks{MIT. Email : \texttt{gamarnik@mit.edu}}}

\begin{document}
\date{\today}
\maketitle


\begin{abstract}
\input{Abstract}

\end{abstract}


{
  \hypersetup{hidelinks}
  \tableofcontents
}

\section{Introduction.}\label{Introduction}
\input{Introduction}

\section{Problem Formulation and the Main Results.} \label{Problem Statement and Notations}
\input{Problem_Formulation_and_Results}


\input{Main_Results}\label{Main Results}

\section{Preliminary Results.}\label{Preliminary Results}
We begin by establishing a series of preliminary technical results. Subsection \ref{Combinatorial Results} is devoted to a series of combinatorial statements,
and Subsection \ref{Tail Bounds} is devoted to concentration inequalities and related tail bounds.

\subsection{Combinatorial Results.} \label{Combinatorial Results}

\input{Combinatorial_Results}

\subsection{Tail Bounds.} \label{Tail Bounds}

\input{Tail_Bounds}

\section{Upper Bounds for the Rademacher Disorder.} \label{Proof of Upper Bounds in Theorems 3.1, 3.2, 3.4}
In this section we derive upper bounds appearing in Theorems \ref{Theorem 3.1}, \ref{Theorem 3.2}, \ref{Theorem 3.4} using the first moment method. Introduce the following variable for $\gamma_n >0$:
$$ U_{\gamma_n}  \triangleq \sum_{S \subset V(G), |S|=K} \mathbf{1}\left(Z_S \geq \gamma_n \binom{K}{2}\right), $$
where $Z_S\triangleq\sum_{\substack{1\le i<j\le n\\ i,j \in S}} Z_{ij}$ under the Rademacher distribution of the disorder $Z_{ij}, 1\leq i<j \leq n$. Let $\gamma_n \triangleq (1 + \delta_n) 2\sqrt{\frac{\log \left( \frac{n}{K}\right)}{K}}$ where $\delta_n$ is some sequence satisfying $\delta_n=o(1)$. We will make a more concrete choice of the sequence $\delta_n$ later. We will progressively make assumptions on $\gamma_n $ (and therefore implicitly on $\delta_n$) allowing us to describe a set of possible values of $\gamma_n$ leading to desired asymptotic bounds on the first moment, then we will check that said values are consistent with the assumption imposed on $\delta_n$.

\input{Proof_of_Upper_Bounds_in_Theorems_3.1,3.2,3.3.tex}

\section{Lower Bounds for the Rademacher Disorder. Preliminary Estimates.} \label{Preliminary Proposition}
\input{Main_Proposition_Proof}

\section{Rademacher Disorder. Proof of Theorem \ref{Theorem 3.1}.} \label{Proof of Lower Bound in Theorem 3.1}

\input{Proof_of_Lower_Bound_in_Theorem_3.1.tex}

\section{Rademacher Disorder. Proof of Theorem \ref{Theorem 3.2}.} \label{Proof of Lower Bound in Theorem 3.2}
\input{Proof_of_Lower_Bound_in_Theorem_3.2.tex}

\section{Rademacher Disorder. Proof of Theorem \ref{Theorem 3.4}.} \label{Proof of Lower Bound in Theorem 3.4}
\input{Proof_of_Lower_Bound_in_Theorem_3.3.tex}

\section{Upper Bounds for the Gaussian Disorder.} \label{Proof of Upper Bounds in Theorems 3.6, 3.7, 3.8}
\input{Proof_of_Upper_Bounds_in_Theorems_3.5,3.6,3.7.tex}

\section{Gaussian Disorder. Proof of Theorem \ref{Theorem 3.6}.}\label{Proof of Lower Bound in Theorem 3.6}
\input{Proof_of_Lower_Bound_in_Theorem_3.5.tex}

\section{Gaussian Disorder. Proof of Theorem \ref{Theorem 3.7}.} \label{Proof of Lower Bound in Theorem 3.7}
\input{Proof_of_Lower_Bound_in_Theorem_3.6}

\section{Gaussian Disorder. Proof of Theorem \ref{Theorem 3.8}.} \label{Proof of Lower Bound in Theorem 3.8}

\input{Proof_of_Lower_Bound_in_Theorem_3.7.tex}

 \section{Proof of Corollaries \ref{Corollary 3.5}, \ref{Corollary 3.9}.} \label{Proof of Corollary 3.5}
\input{Proof_of_Corollary_3.4.tex}

\section{Overlap Gap Property.}

\input{Overlap_Gap_Property.tex}

\section{General Disorder.} \label{Proof of Theorem 3.11}

\input{Lindeberg.tex}

\section*{Acknowledgement.}
The authors would like to thank Noga Alon for several useful suggestions regarding Talagrand's concentration inequality.


\printbibliography
\end{document}

%% file: Abstract.tex

We consider the problem of estimating  the edge density of densest $K$-node subgraphs  of an \ER graph $\G(n,1/2)$. The problem is well-understood
in the regime $K=\Theta(\log n)$ and in the regime $K=\Theta(n)$. In the former case it can be reduced to the problem of estimating the size of largest cliques,
and its extensions, \cite{BBBV}. 
In the latter case the full answer is known up to the order $n^{3\over 2}$ using  sophisticated methods from the theory of spin glasses. The intermediate
case $K=n^\alpha, \alpha\in (0,1)$ however is not well studied and this is our focus. 

We establish that that in this regime the density (that is the maximum 
number of edges supported by any $K$-node subgraph)  is
${1\over 4}K^2+{1+o(1)\over 2}K^{3\over 2}\sqrt{\log (n/K)}$, w.h.p. as $n\to\infty$,
and provide more refined asymptotics under the $o(\cdot)$, for various ranges of $\alpha$. 
This extends earlier similar results in \cite{GamarnikZadik} where this asymptotics was confirmed only when $\alpha$ is a small constant.

We extend our results to the case of ''weighted'' graphs, when the weights have either Gaussian or arbitrary sub-Gaussian distributions. 
The proofs are based on 
the second moment method combined with concentration bounds, the Borell-TIS inequality for the Gaussian case and the Talagrand's inequality
for the  case of distributions with bounded support (including the $\G(n,1/2)$ case). 
The case of general distribution is treated using  a novel symmetrized version 
 of the   Lindeberg argument, which reduces the general case to the Gaussian
case. 
Finally, using the results above we conduct the landscape analysis of the related Hidden Clique Problem, and establish that it exhibits an overlap gap property
when the size of the clique is $O(n^{2\over 3})$, confirming a hypothesis stated in \cite{GamarnikZadik}.

%% file: Introduction.tex
We consider the problem of estimating the density of densest subgraphs of a dense \ER graph $\G(n,1/2)$ with a fixed cardinality of the supporting
node set. Given any simple undirected graph $G=(V,E)$ ($V$ and $E$ standing for the node and edge sets respectively),    and any subset $S\subset V$
the density of $S$ is defined simply as  the number of edges in $S$, that is the cardinality of the set $\{(i,j): (i,j)\in E\}$. While typically the density
is defined as the number of edges normalized by the number of nodes, we focus on the case when sets $S$ have a fixed cardinality, so this 
distinction is non-essential. We are interested in estimating the values of the maximum density of subsets $S$ of $\G(n,1/2)$ with a fixed cardinality $|S|=K$.
When $K\le 2(1+o(1))\log_2 n$ the optimum value is ${K\choose 2}$ since a clique of this size exists with high probability (w.h.p.) 
\cite{Matula}, see \cite{AlonProbabilistic}
for  the textbook reference. The extension to the case $K=O(\log n)$ is also obtained in~\cite{BBBV}. At the other end
of the spectrum, the answer for the case $K=\Theta(n)$ can be derived using the sophisticated methods of spin glass theory, involving evaluation
of the so-called Parisi measure~\cite{GamarnikSubhabrata}. While this reference is devoted to a Gaussian model with a planted sparse rank one matrix, the answer
for the graph can be inferred by assuming that the matrix is trivially zero and connecting the Gaussian model to the random graph model using
standard Lindeberg type argument (more on this below). 

The intermediate case $K=n^\alpha, \alpha\in (0,1)$, however received relatively little attention. The first reference which directly addresses 
 this regime is~\cite{GamarnikZadik}. It is not hard to ''guess'' the right asymptotic answer for this regime. The asymptotic answer ''should be'' 
\begin{align}\label{eq:main-asymptotics}
{1\over 2}{K\choose 2}+{1\over 2}K^{3\over 2}\sqrt{ \log (n/K)}+E_n
\end{align}
where $E_n=o(K^{3\over 2} \sqrt{\log(n/K)})$, for the following reasons. For any fixed set $S, |S|=K$, the number of edges has a Binomial distribution with ${K\choose 2}$ trials, and as such is 
approximated by $ \frac{1}{2}{K\choose 2}+(K/{2\sqrt{2}})Z$, where $Z$ is a standard normal r.v. The extra $\sqrt{2 K\log (n/K)}$ factor arises from extremizing
over all $K$-subsets $S$. It is not hard to turn this intuition into an upper bound using standard first moment argument, and this was done
in~\cite{GamarnikZadik}. Using a much more involved second moment argument augmented further by a certain flattening procedure, inspired
by~\cite{BBBV}, the asymptotics (\ref{eq:main-asymptotics}) was confirmed in~\cite{GamarnikZadik}  by a matching lower bound, but only when 
$\alpha<1/2$. The gap between the upper and lower bound, namely the error hidden in $E_n$ was upper bounded by $K^\beta$ for some $1.449<\beta<3/2$,
see Theorem~4 in the above-mentioned reference  and a follow up remark. 
Extension to the Gaussian case, namely the case when the edges of the graph are equipped with a Gaussian rather than Bernoulli 
weighted distribution, was considered later in~\cite{ZadikArousWein}. While the model considered in \cite{ZadikArousWein} included a planted signal, one could informally interpret the obtained asymptotics in the context of a vanishing signal. There, the asymptotics (\ref{eq:main-asymptotics}) (which is naturally the same) was confirmed
for all $\alpha\in (0,1)$ with an error bound $O\left( \max\{ \left(K^{5/2}/n\right)\sqrt{\log(n)}, K \sqrt{\log(n)} \}\right)$ on $E_n$. The Gaussian extension as well as an
extension to the case of arbitrary weight distribution is considered in our paper as well.

We now summarize our results. In short we confirm the  asymptotics (\ref{eq:main-asymptotics}) in the entire range $\alpha\in (0,1)$ 
for the unweighted random graph case $\G(n,1/2)$, w.h.p. as $n\to\infty$. 
The error bounds we obtain depend on the  value of $\alpha$ and are progressively weaker
as $\alpha$ increases. Specifically, we show that $E_n=O\left(\sqrt{K\log (n)^3}\right)$ when $\alpha\in (0,1/2)$, though the second
order term $\left(K^{3\over 2}/2\right) \sqrt{\log (n/K) }$ has to be slightly modified, see Theorem~\ref{Theorem 3.1}. Similarly, 
$E_n=O\left(K\log n \right)$ when $\alpha\in [1/2,2/3)$ and $O\left( \left(K^{5\over 2}/n\right)\sqrt{\log n}\right)$ when $\alpha\in [2/3,1)$, with second order 
terms slightly modified. It is not hard to see
that the  bound for the case $\alpha\in [2/3,1)$ is the weakest of three (though still clearly order of magnitude smaller than the second leading asymptotics $K^{3\over 2} \sqrt{\log(n/K)}$. 

Turning to the Gaussian case, our results are slightly stronger in this case. Specifically, we obtain an error bound $a_n\sqrt{K\over \log n}$  on $E_n$ w.h.p. 
when $\alpha\in (0,1/2)\cup (1/2,2/3)$, where  $a_n\ge 0$ is an an arbitrarily slowly growing function. For the ''boundary'' case $\alpha=1/2$ we show that
 $|E_n|\le a_n K$, where $a_n\ge 0$ is again an arbitrarily slowly growing  function. Finally, our bound on $E_n$ for the case $\alpha\in [2/3,1)$
 is $O\left( \left(K^{5\over 2}/n\right)\sqrt{\log n}\right)$, namely it is the same as for the case of a random \ER graph $\G(n,1/2)$. 
 
 For the case when the weight distribution is general, we show that the answer is the same as for the Gaussian case within an
 error bound $O\left(K^{\frac{4}{3}} \log(n)^{\frac{7}{6}}\right)$ 
 when the underlying distribution is sub-Gaussian. Namely, we confirm that the optimal value of the densest subgraph problem is universal
 (depends on the underlying distribution only via its first and second moment).  
 Unfortunately, we managed to derive this bound only on the expectation of the optimal value.
 We do show that the bound holds w.h.p. as well but for a more restrictive case when the underlying distribution has a bounded support. 
 When $K=\omega(n^{6\over 7})$ the bound  $O\left(K^{\frac{4}{3}} \log(n)^{\frac{7}{6}}\right)$ that we obtain is smaller than the error bound $O\left( \left(K^{5\over 2}/n\right)\sqrt{\log n}\right)$ 
 arising in the Gaussian case for $\alpha\ge 2/3$. This does not mean (bizzarely) that  our bound for the sub-Gaussian case is stronger than for the 
 Gaussian case, as the error bound $O\left(K^{\frac{4}{3}} \log(n)^{\frac{7}{6}}\right)$  is \emph{on top} of the error bound incurred for the Gaussian case itself.
 
 Next we discuss our proof ideas. The core of the argument is the usage of the second moment method providing a matching lower bound
 to the upper bound obtained by the application of the first moment estimation. Specifically, letting $Z_\theta$ denote the number
 of subgraphs with edge density $\theta$ we compute the asymptotic values for $\E[Z_\theta]$ and $\E[Z_\theta^2]$. If $\theta$ is such
 that $\E[Z_\theta]\to 0$, then by Markov inequality the optimal value is at most $\theta$ w.h.p. Choosing the ''largest'' value  $\theta^*$
 such that the expectation $\E[Z_{\theta^*}]$ \emph{does not} converge to zero (and in fact diverges to infinity by adjusting $\theta^*$ to a slightly smaller value $\theta$)
 we next check whether $\E[Z_\theta^2]$ is asymptotically close to $\E^2[Z_\theta]$. If this is the case the application of a standard Paley-Zygmund inequality
 asserts that the graphs with density $\theta$ exist w.h.p. This is a general strategy we employ in this paper.
  While the use of the second moment method is very standard
 and was employed widely, including \cite{GamarnikZadik} and \cite{ZadikArousWein}, the bounds associated with this method have
 to be extremely delicate and thus our derivations are very involved. 
 An important ingredient is a hard (non-asymptotic)
 bound on a tail of a multi-variate normal distribution in terms of the underlying covariance matrix. This is known as 
 Savage's Theorem (Theorem~\ref{Theorem 4.14}). As we are not aware of any similar bound for multi-nomial distributions, we have
 derived it ourselves (Lemma~\ref{Lemma 4.11}) using  the Hoeffding bound. 
 
 The asymptotic estimates for the \ER graph are tight enough in the
 case $\alpha<1/2$, in the sense that indeed the moments match: $\E[Z_\theta^2]=(1+o(1))\E^2[Z_\theta]$. However, the estimations 
 for the case $\alpha\in [1/2,2/3)$ only allow us to conclude that $\E[Z_\theta^2]=O\left(e^{\Theta(\log(n)^2)}\E^2[Z_\theta]\right)$. This in turn
 implies that graphs with density $\theta$ exist with probability at least $O\left(e^{-\Theta(\log(n)^2)}\right)$, which is unfortunately vanishing. 
 To complete the proof we use Talagrand's  concentration inequality to show that the value of the optimum is concentrated around its median, as well as its mean.
 The error value $t$ occurs with probability at most $4\exp\left(-t^2/(2K^2)\right)$. For the Gaussian case this error bound follows
 directly from Gaussian concentration inequality known as Borell-TIS inequality.
 Using this strong concentration bounds we boost the weak lower bound of order $e^{-\Theta(\log(n)^2)}$ to a high probability bound again by slightly adjusting the value
 of $\theta$ to a slightly smaller value.
  The case $\alpha\in [2/3,1)$ leads to a weaker estimate $\E[Z_\theta^2]=\exp\left(O\left({K^3\over n^2}\log n\right)\right)\E^2[Z_\theta]$.
  Similarly to the previous case, we use Talagrand's concentration inequality 
 to boost the required success probability from $\exp\left(-O\left({K^3\over n^2}\log n\right)\right)$ to $1-o(1)$. 

Next we discuss our approach for handling the case of a general distribution. Arguably, the method of choice for establishing universality type results
is the Lindeberg method. See~\cite{Lindeberg} for a general description of  the method and \cite{Subhabrata}
for application of the method for studying extremal problems in random graphs, similar to our setting.
The rough idea behind the method is as follows. Suppose we have a test function $f(X)$ applied to a vector of i.i.d. random
variables $X\in \R^T$. Our goal is to show that the expected value $\E[f(X)]$ 
does not change significantly if $X$ is replaced by a Gaussian vector with matching first and second moments. The approach is to consider an 
interpolation scheme $Y_t, 0\le t\le T$, were the coordinates of $X$ are switched to the one of $Z$ one by one, so that $Y_0=X$
and $Y_T=Z$. Then, at every step $t$ of the interpolation, the change of the expected values is estimated by one-dimensional Taylor expansion
around the nominal value $0$. As the first two moments match, the error term can be bounded in terms of the third moments of coordinates of $X$ and its
gaussian counterparts. The universality argument then holds if this error term is $o(1/T)$ so that the accumulated error is $o(1)$. 
As the maximization operator is not a smooth function allowing for Taylor expansion, the maximization operator if often replaced by its soft-max
version, such as partition function with large enough inverse temperature, see~\cite{Subhabrata}. As we show in the body of the paper
this approach indeed works in our case, but it incurs and error bound $O\left(K^{\frac{2}{3}}n^{\frac{2}{3}} \log(n)^{\frac{2}{3}}\right)$. This error bound is only 
meaningful when $\alpha\ge 4/5$, since below this value the error bound is larger than $K^{3\over 2}$ -- the second order term in our asymptotics
of the optimum. In order to derive a tighter estimate, which is in fact  $O\left(K^{\frac{4}{3}} \log(n)^{\frac{7}{6}}\right)$, we introduce a novel and more ''economic''
version of the Lindeberg's argument, exploiting symmetry in a fundamental way. Specifically, we show that when the interpolation order is chosen
uniformly at random ''most'' of the third order Taylor expansion terms disappear since they are not associated with subraphs supporting the extremal value.
This symmetry holds in expectation for fairly trivial reasons at the beginning and the end of the interpolation, since in these cases the distribution of weights
on all edges is identical. The intermediate case $1\le t\le T-1$ however is a different story, as in this case there is a mix of distributions. This is where
a random uniform choice comes to the rescue providing us with averaged out symmetry. We believe that our proof technique is of separate interest.

Finally, we turn to the description of our results regarding the Overlap Gap Property associated with the Planted Clique model. 
We begin with the description of the Planted Clique model (also often called Hidden Clique model). 
Starting with the \ER graph $\G(n,1/2)$ we select an arbitrary subset of $K$ nodes (say nodes $1,\ldots,K$) and place a clique 
on this node set. That is, for every pair $(i,j), 1\le i<j\le K$  if the original graph did not have $(i,j)$ as an edge, it turns this pair into an edge in the modified
graph, and if the pair was already an edge, it is left untouched. The resulting graph is denoted by $\G(n,K,\Bern(1/2))$. A key question of algorithmic interest is to recover
the underlying clique when its precise location is not known to the algorithm designer.  
A long stream of research starting with~\cite{Jerrum},\cite{AlonKrivelevitchSudakov} studied this problem. Specifically, the model
was introduced in~\cite{Jerrum} and the first non-trivial polynomial time algorithm leading to 
the recovery of the clique  when $K\ge cn^{1\over 2}$ for any constant $c>0$ was constructed in~\cite{AlonKrivelevitchSudakov}. 
Notably, without any bound on the computation time, the recovery is possible as soon as $K\ge (2+c)\log_2 n$ for any constant $c$, since,
as mentioned above, the largest clique naturally occurring in $\G(n,1/2)$ is only $2(1+o(1))\log_2 n$, and thus  the largest clique returned by brute force
search has to be the hidden clique. The gap $\Theta(\log_2 n)$ vs $\Theta(\sqrt{n})$ is quite substantial. It is often called statistics-to-computation gap
and has been discovered for many many other problems~\cite{BrennanBresler}. 
While it is believed to be fundamental, namely, it is unlikely
that a polynomial time algorithm for clique recovery exists for $K=n^\alpha, \alpha<1/2$, it remains just a conjecture. 

The main focus of~\cite{GamarnikZadik} was to investigate the Hidden Clique problem from the solution space geometry perspective. It was conjectured
in this paper that the model exhibits the so-called Overlap Gap Property (OGP) when  $\alpha<2/3$. In particular not matching the algorithmic threshold.
At the same time it was also conjectured that the model exhibits the gap when $\alpha<1/2$, though for the so-called overparametrized regime. 
We refer to~\cite{GamarnikZadik} 
for the definition and significance of the overparametrized regime. 
 To describe the OGP, consider the following
parametrized optimization problem. For every value $z\in [0,K]$ let $\Psi_K(z)$ denote the largest density among subgraphs of $\G(n,K, \Bern(1/2))$,
supported by $K$-node subsets which share precisely $z$ nodes with the hidden clique. In particular, when $z=K$ this is the hidden clique itself
with $\Psi_K(K)={K\choose 2}$, 
and when $z=0$, this corresponds to  subgraphs which share no common nodes with the hidden graph. In the latter case the maximum is clearly at most
${K\choose 2}$. The overparametrized regime mentioned above corresponds to changing the value of the support size $K$ to large value $\bar K$, 
while keeping the size of the clique to be the same value $K$. 

We say that the model exhibits the OGP
if the function is non-monotonic in the interval $[K^2/n, K]$. In particular,  there exist $K^2/n\le z_1<z_2\le K$ such that $\max_{z\in [z_1,z_2]}\Psi_K(z)$
is smaller then $\Psi_K(K^2/n)\le \Psi_K(K)$ with a ''significant gap between the maximum and $\Psi_K(K^2/n)$ and with a ''significant''
value of $z_2-z_1$. The choice of $K^2/n$ as the lower end of the range of values for $z$ is motivated by the fact that it  
 is the expected  value of the overlap with the hidden clique obtained by randomly choosing a subset of $K$ nodes. It thus makes
sense to consider the range of overlaps within $[K^2/n,K]$. (We note that $K^2/n$ is asymptotically zero when $K=n^\alpha$ and $\alpha<1/2$) 
The presence  of the OGP is a barrier to certain classes of algorithms, specifically algorithms based on the Markov Chain Monte Carlo (MCMC)
methods, see~\cite{GamarnikZadik} for details. The presence of the OGP was proven in~\cite{GamarnikZadik}, but only when $\alpha< 0.0917$. 
The status of the other cases  remains a conjecture, and a partial resolution of this conjecture constitutes our last set of results.
Specifically, we confirm the presence of the OGP in the case $0<\alpha<1/2$. We further prove the OGP when $1/2\le \alpha<2/3$ as well,
but subject to an unproven conjecture (Conjecture~\ref{conjecture 3.3}), regarding the strength of our upper and lower bounds on the value
of the densest subgraph problem in the \ER $\G(n,1/2)$ when $\alpha\in [1/2,2/3)$. While, as stated above, we did derive upper and lower bounds 
matching up to $O(K\log n)$ gap, the conjecture states that the gap can be improved to $o(K\log n)$. We note that we do establish the validity of this conjecture albeit for the Gaussian case. Resolving the aforementioned conjecture is an interesting
question for future research. So is the status of the OGP when $\alpha\ge 2/3$. In this case we conjecture that OGP does not hold. Precisely
what this means needs to be clarified. See for example~\cite{RandomPartitionNumber} for one instantiation of this negative statement.
For a general survey about OGP based techniques we refer to~\cite{GamarnikPNAS}, where it is noted that the strongest applications of the OGP based negative results correspond
to models with no planted structure. Whether the presence of the OGP presents a barrier to broader than MCMC type algorithms for models
with planted structure is another very interesting question for future research.

We finish this section with some notational conventions. $[n]$ denotes the set $\{1,...,n\}$.
All order of magnitude notations $O(.), o(.), \omega(.), \Omega(.)$ should be understood with respect to $n$, that is the number of nodes in the graph. They are assumed
in the absolute sense, so that, for example, for a positive sequence $a_n>0$, $O(a_n)$ denotes any sequence $b_n$ with $\limsup_n|b_n|/a_n<\infty$.
Given sequences $a_n\ge 0, b_n>0, n\ge 1$ we write $a_n \lesssim b_n$ when $ \limsup_{n \to +\infty} \frac{a_n}{b_n} \leq 1$. $\gtrsim$ is defined similarly.
We write $a_n\sim b_n$ if $a_n\lesssim b_n$ and $a_n\gtrsim b_n$.
$\stackrel{d}{=}$ denotes equality in distribution. $|S|$ denotes the cardinality of a finite set $S$. Given a graph $G$ and a set of nodes
$S$ in it, we denote by $E[S]$ the set of edges spanned by $S$. 
We finish this section with some notational convention. $\mathcal{N}(\mu,\sigma^2)$ denotes the normal distribution with mean $\mu$ and
variance $\sigma^2$. $\Bern(p)$ denotes the Bernoulli distribution with parameter $p$. Namely $X=1$ with probability $p$ and $X=0$
with probability $1-p$ when $X\stackrel{d}{=} \Bern(p)$. $\mathcal{R}$ denotes the Rademacher distribution. Namely $X=\pm 1$ with
equal probability when $X\stackrel{d}{=} \mathcal{R}$. The probability distribution of a random variable $X$ is called sub-Gaussian if there are positive constants $C, v$ such that for every $t > 0$ it holds  $p(|X|>t)\leq Ce^{-vt^{2}}$. $h(x)$ denotes the binary entropy, i.e $h(x) = -x\log(x) - (1-x) \log(1-x)$. Finally, let $\mathfrak{S}_n$ be the set of permutations of $[n]$.


%% file: Problem_Formulation_and_Results.tex
$\mathbb{G}(n, 1/2)$ denotes  a dense \ER graph on $n$ nodes. 
Namely, it is a graph on $n$ nodes where each pair $(i,j), 1\le i<j\le n$ is an edge with probability $1/2$ and is not an edge
with probability $1/2$, independently across pairs. We write $Z_{ij}=1$ in the former case and $Z_{ij}=0$ in the latter case.
For each subset of nodes $S\subset [n]\triangleq \{1,\ldots,n\}$, let 
\begin{align}\label{eq:Z_S}
Z_S\triangleq\sum_{\substack{1\le i<j\le n\\ i,j \in S}} Z_{ij},
\end{align} 
be the number of edges
spanned by $S$, namely the density. (The notion of density is somewhat abused here, since typically it corresponds to the ratio 
of the number of edges to the number of nodes, or to the number of nodes squared. We maintain the term density here for simplicity). 
Our focus is obtaining the asymptotic values for 
\begin{equation}\label{eq:Psi_K}
\Psi_K\triangleq \max_{S \subset [n], |S|=K} Z_S, 
\end{equation}
when $K=n^\alpha$ in the regime $\alpha\in (0,1)$.

We generalize this model to the case of an arbitrary probability distribution $\mathcal{D}$ on $\R$ as follows. 
We generate $Z_{ij}, 1\le i<j\le n$ according to $\mathcal{D}$, i.i.d. and define $Z_S$ as in (\ref{eq:Z_S}). Borrowing from the statistical
physics literature, we call the matrix $(Z_{ij}, 1\le i<j\le n)$ the disorder of the model. 
The goal is obtaining asymptotics
in (\ref{eq:Psi_K}), and we denote the left-hand of it by $\Psi_K^{\mathcal{D}}$ instead. The case $\mathcal{D}=\Bern(1/2)$ corresponds
to the graph case above. Another distribution of a key focus for us is the normal distribution $\mathcal{N}(1/2,1/4)$. The parameters
$1/2$ and $1/4$ are chosen to match the first two moments with ones for $\Bern(1/2)$. We will sometimes also use the notation $\Psi^{\mathcal{D}}_K(G)$ where $G$ is the  graph associated to the disorder matrix $Z_{ij}, 1\leq i < j \leq n$.

For any distribution  $\mathcal{D}$ with finite 2nd moment, let $\bar{\mathcal{D}}$ correspond to the centered and rescaled version of $\mathcal{D}$, obtained
by $X\to (X-\E{X})/\sqrt{{\mathrm{Var}}(X)}$. In particular it is standard normal distribution for $\mathcal{N}(1/2,1/4)$ and the Rademacher
distribution for $\Bern(1/2)$. If the mean and the variance of $\mathcal{D}$ are $1/2$ and $1/4$, respectively, then the mean and the variance 
of $\bar{\mathcal{D}}$ are $0$ and $1$, respectively and we have
\begin{align*}
\Psi^{\mathcal{D}}_K = \frac{1}{2}\binom{K}{2} +\frac{1}{2}\Psi^{\bar{\mathcal{D}}}_K,
\end{align*}
so for convenience we focus on the asymptotics of $\Psi^{\bar{\mathcal{D}}}_K$ instead. Thus, for example, we have
\begin{align*}
\Psi^{\Bern(1/2)}_K &= \frac{1}{2}\binom{K}{2} +\frac{1}{2}\Psi^{\mathcal{R}}_K,   \\
\Psi^{\mathcal{N}(1/2,1/4)}_K &= \frac{1}{2}\binom{K}{2} +\frac{1}{2}\Psi^{\mathcal{N}}_K,
\end{align*}
where from this point on we use $\Psi^{\mathcal{N}}_K$ in place of $\Psi^{\mathcal{N}(0,1)}_K$ for short. 

Introduce
\begin{align*}
    V(n,K) &\triangleq\sqrt{2 \binom{K}{2} \log\left(\binom{n}{K}\right)} , \numberthis \label{V(n,K)} \\
    L(n,K)  &\triangleq \sqrt{2 \binom{K}{2} \log\left(\binom{n}{K}\frac{1}{K}\right)},  \numberthis \label{L(n,K)}\\
    U(n,K) &\triangleq\sqrt{2 \binom{K}{2} \log\left(   \binom{n}{K}\frac{1}{\sqrt{ K \log\left(\frac{n}{K}\right)}}  \right)  } \numberthis\label{U(n,K)}.
\end{align*}
Observe that $L(n,K) \leq U(n,K) \leq V(n,K)$. As $n,K \to +\infty$,   $K = o(n)$, we have
\begin{alignat*}{1}
\hspace*{-2cm}
&L(n,K), U(n,K) = V(n,K) \left( 1 + O\left(\frac{1}{K}\right) \right)^{-1}  \label{eq:Vnk}, \numberthis\\
     &V(n,K) - L(n,K),V(n,K) - U(n,K)  =  O\left( \sqrt{K \log(n)}\right)  \label{eq:VvsLUAbsolute}, \numberthis\\
&L(n,K), U(n,K), V(n,K) \sim   K^{3\over 2}\sqrt{\log \left({n\over K}\right)}  \numberthis\label{LNK}.
\end{alignat*}
The proof of the above asymptotic properties is given in Lemma \ref{Lemma 4.80}.

%% file: Main_Results.tex
We now turn to the description of our results. 
\subsection{Bernoulli/Rademacher Disorder}\label{Subsection 3.1}

\begin{theorem}\label{Theorem 3.1}
Suppose $\alpha \in \left(0,\frac{1}{2}\right)$. For any non-negative sequence $a_n=\omega(1)$ (that is any growing non-negative sequence),
the following holds w.h.p. as $n\to\infty$ 
\begin{align}\label{eq:L-lower}
L(n,K) -O\left( \sqrt{K \log(n)^3} \right) \le 
\Psi^{ \mathcal{R}}_K \le   
U(n,K) + a_n\sqrt{\frac{K}{\log(n)}}.
\end{align}
The above bounds can be combined to obtain the following asymptotics:
\begin{align}\label{eq:V-combined}
\Psi^{\mathcal{R}}_K  = V(n,K) + O\left(\sqrt{K \log(n)^3}\right).
\end{align}
\end{theorem}
In particular, according to (\ref{eq:L-lower}) and applying (\ref{eq:VvsLUAbsolute}),
 our upper and lower bounds match up to order $O(\sqrt{K\log (n)^3})$. This illustrate that the leading asymptotics (after the common $O(K^2)$ term) is
 given by $V(n,K)$, which is asymptotically  $K^{3\over 2}\sqrt{\log \left({n/ K}\right)}$ according to (\ref{LNK}).

Next we turn to the case $\alpha\in \left[ 1/2, 2/3\right)$, in which our bounds are  not as tight, unfortunately.
\begin{theorem}\label{Theorem 3.2}
Suppose $\alpha \in \left[\frac{1}{2},\frac{2}{3}\right)$. For any non-negative sequence $a_n=\omega(1)$,
the following holds w.h.p. as $n\to\infty$ 
\begin{align}
 L(n,K) - O(K \log(n)) \le 
 \Psi^{\mathcal{R}}_K \le 
 U(n,K) +  a_n\sqrt{\frac{K}{\log(n)}} .\label{eq:L-lower-2}
\end{align}
The above bounds can be combined to obtain the following asymptotics:
\begin{align*}
\Psi^{\mathcal{R}}_K = V(n,K) + O(K \log(n)).
\end{align*}
\end{theorem}
We see that the gap between the lower and upper bound is now $O(K\log n)$, as opposed to $O(\sqrt{K\log (n)^3})$ when $\alpha\in (0,1/2)$.
We conjecture that this can be improved: 
\begin{conjecture}\label{conjecture 3.3}
The lower bound  of Theorem~\ref{Theorem 3.2} can be improved to $o(K\log(n))$. I.e when $\alpha \in \left[\frac{1}{2},\frac{2}{3}\right)$  w.h.p. as $n\to\infty$, we have 
\begin{align*}
\Psi^{\mathcal{R}}_K \geq L(n,K) - o\left(K \log(n)\right).
\end{align*}
\end{conjecture}
As it turns out,  validating this conjecture would have ramification for the  Planted Clique Problem and the associated solution space geometry,
see Theorem~\ref{Theorem 3.14} below.

Turning to the remaining case $\alpha\in [2/3,1)$, our results are as follows.

\begin{theorem}\label{Theorem 3.4}
Suppose $\alpha \in \left[\frac{2}{3},1\right)$. For any non-negative sequence $a_n=\omega(1)$,
the following holds w.h.p. as $n\to\infty$:
\begin{align*}
 L(n,K) - O\left( \frac{K^{\frac{5}{2}}}{n} \sqrt{\log(n)} \right)
 \le
 \Psi^{\mathcal{R}}_K 
 \leq 
 U(n,K) +  a_n\sqrt{\frac{K}{\log(n)}}.
\end{align*}
\end{theorem}
In this case the gap between the upper and lower bounds is an even cruder quantity $(K^{5\over 2}/n)\sqrt{\log n}$. When $\alpha>2/3$ this is 
$\omega(K\log n)$, that is wider than the gap in the case $\alpha<2/3$. To verify this, note that   $K^{5\over 2}/(nK)=n^{\Theta(1)}=\omega(\sqrt{\log n})$.

All of the cases above can be summarized into one asymptotics provided below, with understanding the terms under $O(\cdot)$ are refined
in various ways in  cases $\alpha\in (0,1/2)$ vs $\alpha\in [1/2,2/3)$ vs $\alpha\in [2/3,1)$.

\begin{corollary}\label{Corollary 3.5}
For every $\alpha \in (0,1)$ w.h.p. as $n\to\infty$
\begin{align*}
\Psi^{ \Bern\left(\frac{1}{2} \right)}_K = \frac{K^2}{4} +  \frac{K^{\frac{3}{2}} \sqrt{\log(\frac{n}{K})}}{2} + 
o\left( K^{\frac{3}{2}} \sqrt{\log(n)} \right).
\end{align*}
\end{corollary}

\subsection{Gaussian Disorder}\label{Subsection 3.2}
We now turn to the case when the distribution of weights (disorder) $Z_{ij}$ is Gaussian.
\begin{theorem}\label{Theorem 3.6}
Suppose $\alpha \in \left(0, \frac{1}{2}\right)\cup \left( \frac{1}{2}, \frac{2}{3}\right)$.  For any non-negative sequence $a_n=\omega(1)$,
the following holds w.h.p. as $n\to\infty$:
\begin{align*}
U(n,K) - a_n \sqrt{ \frac{K}{\log(n)}}\le 
\Psi^{\mathcal{N}}_K 
\le
U(n,K) + a_n \sqrt{ \frac{K}{\log(n)}} . 
\end{align*}
\end{theorem}
The bounds above might appear similar to those in 
 (\ref{eq:L-lower}) for the case of the Rademacher disorder
and $\alpha\in (0,1/2)$. They are stronger however, since the lower bound term is also in terms of $U(n,K)$ and not $L(n,K)$, as was the 
case for the Rademacher disorder.
As a result the gap is of order $a_n \sqrt{ \frac{K}{\log(n)}} $ as opposed to $O\left(\sqrt{K\log (n)^3}\right)$.

The regime above does excludes the boundary case $\alpha=1/2$. We obtain upper and lower bounds in this case as well, but unfortunately,
our lower bound is not as tight as in the other case.  
\begin{theorem}\label{Theorem 3.8}
Suppose $\alpha= 1/2$. For any non-negative sequence $a_n=\omega(1)$,
the following holds w.h.p. as $n\to\infty$:
\begin{align*}
 U(n,K) -K a_n 
 \le 
\Psi^{\mathcal{N}}_K 
\leq 
U(n,K) + a_n \sqrt{ \frac{K}{\log(n)}} .
\end{align*}
\end{theorem}
Finally, we treat the remaining case $\alpha\in [2/3,1)$. 
\begin{theorem}\label{Theorem 3.7}
Suppose $\alpha \in \left[\frac{2}{3}, 1\right)$. For any non-negative sequence $a_n=\omega(1)$,
the following holds w.h.p. as $n\to\infty$:
\begin{align*}
U(n,K) -O\left( \frac{K^{\frac{5}{2}}}{n} \sqrt{\log(n)}\right)
\le 
\Psi^{\mathcal{N}}_K 
\leq 
U(n,K) + a_n \sqrt{ \frac{K}{\log(\frac{n}{K})}} .
\end{align*}
\end{theorem}

Again, all cases can be summarized (with some loss of error bounds) as follows:
\begin{corollary}\label{Corollary 3.9}
For $\alpha \in (0,1)$ the following holds w.h.p. as $n\to\infty$.
\begin{align*}
\Psi^{ \mathcal{N} \left( \frac{1}{2}, \frac{1}{4} \right)}_K = 
\frac{K^2}{4} +  \frac{K^{\frac{3}{2}} \sqrt{\log(\frac{n}{K})}}{2} + o\left(K^{\frac{3}{2}} \sqrt{\log(n)}\right).
\end{align*}
\end{corollary}

\subsection{Generally Distributed Disorder}
As mentioned in the introduction, the 
asymptotic estimates  in Subsections~\ref{Subsection 3.1}, \ref{Subsection 3.2} will be obtained  using the  first and second moment method,
combined with various concentration inequalities, specifically the Borell-TIS inequality for the Gaussian case and the Talagrand's concentration
inequality for the Rademacher case. Extending this to the case of a general distribution on the disorder appears problematic. Instead,
we resort to the Lindeberg's type argument which facilitates the reduction of the case of a general distribution to the case of a Gaussian distribution.  Our main result is below.
\begin{theorem}\label{Theorem 3.11}
Suppose $\alpha \in (0, 1)$ and  $\mathcal{D}$ is a sub-Gaussian distribution that satisfies 
$\E[Z]=\frac{1}{2}$ and $ \E[Z^2] = \frac{1}{2}$  when $Z\stackrel{d}{=} \mathcal{D}$.
Then 
\begin{align*}
\E\left[\Psi^{\mathcal{D}}_K \right]= \E\left[\Psi^{\mathcal{N}\left( \frac{1}{2}, \frac{1}{4} \right)}_K\right] + O\left(K^{\frac{4}{3}} \log(n)^{\frac{7}{6}}\right).
\end{align*}
Furthermore, if $\mathcal{D}$ corresponds to bounded random variables, the estimate above holds also w.h.p. as $n\to\infty$ as opposed to just expectation.
\end{theorem}

We now comment on our proof approach. A fairly direct application of the Lindeberg's argument (which we nevertheless describe in the body of the proof)
can be used to obtain the following bound.
\begin{align*}
\E\left[\Psi^{\mathcal{D}}_K \right]= \E\left[\Psi^{\mathcal{N}\left( \frac{1}{2}, \frac{1}{4} \right)}_K\right] + O\left(K^{\frac{2}{3}}n^{\frac{2}{3}} \log(n)^{\frac{2}{3}}\right).
\end{align*}
This bound unfortunately is not very meaningful when $\alpha< 4/5$ as in this case  $K^{2\over 3}n^{2\over 3}$ dominates $K^{3\over 2}=n^{3\alpha\over 2}$.
The latter  is the leading order  term in all of our asymptotic results, and thus such a bound is not informative. 
In order to obtain a tighter asymptotics which is claimed in Theorem~\ref{Theorem 3.11},
we resort to a certain ''averaged'' out version of the Lindeberg's argument, where the order of edges swapped from the distribution $\mathcal{D}$ 
to the Gaussian distribution is chosen uniformly at random, thus providing us with a   convenient algebraic symmetry
of  the associated terms. This symmetry is exploited to obtain tighter error bounds. 
We believe our proof approach is of interest in itself.

\subsection{The Planted Clique Problem and the Overlap Gap Property}\label{Subsection 3.4}

We now turn to the widely studied  setting of the Planted Clique model where one observes a graph  sampled in two steps, first the graph  is generated randomly according to the \ER Bernoulli model $\G(n,1/2)$, and then a clique of size $K=n^\alpha, \alpha\in (0,1)$ is planted in the obtained graph. Without  loss of generality
we assume that the set of nodes of the planted clique, denoted by $\mathcal{PC}$ occupies the nodes $1,\ldots,K$. The model also has the following 
natural extension to the case of general weight distributions $\mathcal{D}$. Fix a parameter $\mu>0$ and assume edges $(i,j)$ have weights
$Z_{ij}\stackrel{d}{=}\mathcal{D}$ when at least one of $i$ or $j$ (or both) exceeds $K$, and distributions $\mu+\mathcal{D}$ (shift by $\mu$)
when $1\le i,j\le K$. 
Technically, this is not a generalization of the Bernoulli case, since in the latter
case the distribution of the edge weights inside the clique changes from Bernoulli is deterministic. 
In all cases, the Hidden Clique model will be denoted by $\G(n,K,\mathcal{D})$, with $\mathcal{D}=\Bern(1/2)$ corresponding
to the (unweighted) Hidden Clique model, with some abuse of notation. 
Both the unweighted and the  Gaussian setting of a hidden ''clique'' model have been
studied widely, as already mentioned in the introduction.

Similarly to the  quantities $\Psi^{\mathcal{D}}_K$  we define  their following overlap-restricted version 
for parameter a $z$ ranging in $\{\lfloor{\frac{K^2}{n}\rfloor}, \lfloor{\frac{K^2}{n}\rfloor}+1,...,K \}$:
\begin{equation}
\Psi^{\mathcal{D}}_K(z) \triangleq \max_{S \subset [n], |S|=K , |S \cap \mathcal{PC}|=z} Z_S,
\end{equation}
where $Z_S$ is again defined by (\ref{eq:Z_S}). We note that $\lfloor{K^2/n\rfloor}=0$ when $K=n^\alpha$ and $\alpha<1/2$, but 
grows polynomially in $n$ when $\alpha>1/2$. We also note that trivially $\Psi^{\mathcal{D}}_K=\max_z\Psi^{\mathcal{D}}_K(z)$.

Using the overlap-restricted version $\Psi^{\mathcal{D}}_K(z)$ of the optimization value, we  define the overlap gap property as follows.
\begin{defn}\label{Definition 3.12}
Given sequences $0 < \zeta_{1,n}< \zeta_{2,n} <K$ and $r^2_n > r^1_n>0$, 
the Planted Clique model $\G(n,K,\mathcal{D})$ exhibits an Overlap Gap Property (OGP) with parameters $(\zeta_{1,n},\zeta_{2,n},r^1_n,r^2_n)$
if  the following holds w.h.p. as $n\to \infty$:
    \begin{enumerate}
        \item There exists subsets $A,A'\subset [n]$ with $|A|=|A'|=K, |A\cap \mathcal{PC}| \leq \zeta_{1,n} <  \zeta_{2,n} \leq |A'\cap \mathcal{PC}| $ 
        and $|E[A]|, |E[A']|  \geq r^2_n$.
        \item For any subset $A\subset [n]$ with $|A|=K, \zeta_{1,n} \leq | A\cap \mathcal{PC}| \leq \zeta_{2,n}  $ it holds  $|E[A]| \leq r^1_n$.
    \end{enumerate}
    Where $|E[A]| $ is the sum of $Z_{ij}$ over $i,j \in A$.
\end{defn}
The intuition behind this definition is as follows. The parameters $r^1_n$ and $r^2_n$ describe thresholds for optimization values
(so that they are non-vacuous only when they are below the optimum value $\Psi^{\mathcal{D}}_K$). The model exhibits the overlap 
gap property if every set with edge density at least $r^1_n$ is either ''close'' to the clique $\mathcal{PC}$ (case $|A\cap \mathcal{PC}| \leq \zeta_{1,n}$),
or far from it (case $|A\cap \mathcal{PC}| \geq \zeta_{2,n}$), or conversely every set with intermediate (between $\zeta_{1,n}$ and 
$\zeta_{2,n}$) overlap with the hidden clique $\mathcal{PC}$ has edge density at most $r^1_n$. This is the second part of the definition.
The first part says that both within the ''close'' range  and within the ''far'' range there exist sets with edge density ''significantly'' above
$r^1_n$ (that is above $r^2_n$). In the ''close'' range case for the non-weighted Hidden Clique model the value of $\Psi^{\mathcal{D}}_K$
is clearly ${K\choose 2}$, so the definition is non-trivial only when $r^2_n$ at most this value. Figure~\ref{Figure 1} illustrates the idea
of this definition.

\begin{minipage}[t]{\linewidth}
\begin{center}
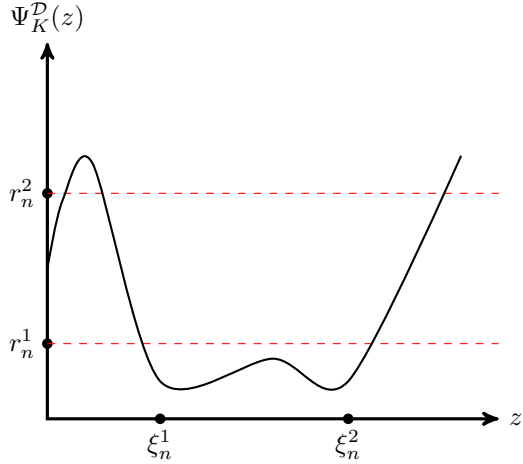

\begin{tikzpicture}[
        scale=2,
        IS/.style={blue, thick},
        axis/.style={very thick, ->, >=stealth', line join=miter},
        important line/.style={thick}, dashed line/.style={dashed, thin},
        every node/.style={color=black},
        dot/.style={circle,fill=black,minimum size=4pt,inner sep=0pt,
            outer sep=-1pt},
    ]

\hspace*{13em}
\draw[axis,<->] (3,0) node(xline)[right] {$z$} -|
                    (0,2.5) node(yline)[above] {$\Psi_K^{\mathcal{D}}(z)$};
\node[dot,label=left:$r_n^1$] at (0,0.5) {};
\node[dot,label=left:$r_n^2$] at (0,1.5) {};

\node[dot,label=below:$\xi_n^1$] at (0.75,0) {};
\node[dot,label=below:$\xi_n^2$] at (2,0) {};
\draw [thick,dashed line, red] (0,0.5) -- (3,0.5);
\draw [thick,dashed line, red] (0,1.5) -- (3,1.5);

\draw [thick, xshift=0cm] plot [smooth, tension=0.5] coordinates 
  {(0, 1)  (0.1,1.45) (0.3,1.7) (0.75,0.25) (1.5,0.4) (2,0.25) (2.75,1.75)};

\end{tikzpicture} 
\end{center}
\end{minipage}\hfill

\captionof{figure}{Overlap-Gap Property Illustration}\label{Figure 1}

\begin{theorem}\label{Theorem 3.13}
Suppose
 $\alpha \in \left(0, \frac{1}{2}  \right)$ and $K=n^\alpha$. There exists  $0 <C_1$, $0 < D_1 <D_2$  
 such that the following holds w.h.p. as $n\to +\infty$: 

$$ \max_{z \in I} \Psi^{\Bern(\frac{1}{2})}_K(z)  \leq \min \left\{ \Psi^{\Bern(\frac{1}{2})}_K\left( 0\right) , \Psi^{\Bern(\frac{1}{2})}_K\left( \left\lfloor \frac{K}{2} \right\rfloor \right)    \right\} - C_1 K\log(K), $$
where $ I\triangleq \mathbb{Z} \cap [D_1 \sqrt{K\log(K)}, D_2 \sqrt{K\log(K)}]$. Namely, the  Planted Clique model $\G(n,K,\Bern(1/2))$
 exhibits OGP as defined in Definition~\ref{Definition 3.12} with parameters

\begin{align*}
    \zeta_{i,n} &= D_i \sqrt{K \log (K)}, i=1,2, \\
    r^2_n& = \min \left\{ \Psi^{\Bern(\frac{1}{2})}_K\left(0 \right) , \Psi^{\Bern(\frac{1}{2})}_K\left( \left\lfloor\frac{K}{2} \right\rfloor\right)    \right\} , \\
    r^1_n &= r^2_n - C_1 K \log(K).
\end{align*}
\end{theorem}

The claim above extends to the case $\alpha\in [1/2,2/3)$ modulo the validity of Conjecture~\ref{conjecture 3.3}.
\begin{theorem}\label{Theorem 3.14}
Assume  Conjecture \ref{conjecture 3.3} is valid. There exists $C_0>0$ such that the statement of Theorem \ref{Theorem 3.13} holds for the case 
$\alpha \in \left[\frac{1}{2}, \frac{2}{3} \right)$ by replacing 
$\Psi^{\Bern(\frac{1}{2})}_K\left(0 \right)$ with $\Psi^{\Bern(\frac{1}{2})}_K\left(\left\lfloor{C_0 \frac{K^2}{n}} \right\rfloor \right)$.
\end{theorem}
Theorems \ref{Theorem 3.13}, \ref{Theorem 3.14} establish the presence of OGP for the Bernoulli planted clique problem. It is notable that in the Bernoulli case, the OGP holds up to $\alpha=2/3$, whereas the conjectured algorithmic threshold is
$\alpha=1/2$. We refer to \cite{GamarnikZadik} for an extensive  discussion of this issue. We conjecture that the OGP does not hold for 
$\alpha>2/3$. The sense in which the ''non-OGP'' property holds is also discussed in \cite{GamarnikZadik}.

%% file: Combinatorial_Results.tex
\begin{lemma}\label{Lemma 4.1}
For $K \in [n]$  we have:
$$  \sum_{0\leq l \leq K }  \binom{K}{l} \binom{n-K}{K-l} = \binom{n}{K} .$$
\end{lemma}
\begin{proof}
The right hand side is the number of subsets of size $K$ in $\{1,...,n\}$. Now assume that we color the elements $\{1,...,K\}$ as blue and the remaining elements $\{K+1,...,n\}$ as red, for $l\in [0, K]$ there are $\binom{K}{l} \binom{n-K}{K-l}$ unique ways of constructing a $K$ size subset of $\{1,...,n\}$ that contains exactly $l$ blue elements (and thus $K-l$ red elements). Hence the left hand side counts the number of unique subsets of size $K$ (by conditioning on the number of picked blue elements), the proof of the Lemma follows.
\end{proof}

\begin{lemma}\label{Lemma 4.6}
Let $0 \leq K \leq n$ and  $0 \leq l \leq K $ be integers, then the following identity holds:
$$ \binom{n}{l} \binom{n-l}{K-l} \binom{n-K}{K-l} \binom{n}{K}^{-2} = \binom{K}{l} \binom{n-K}{K-l} \binom{n}{K}^{-1} .$$
\end{lemma}
\begin{proof}
\begin{align*}
    \binom{n}{l} \binom{n-l}{K-l} \binom{n-K}{K-l} \binom{n}{K}^{-2} &= \frac{n!}{l!(n-l)!} \frac{(n-l)!}{(K-l)!(n-K)!} \frac{(n-K)!}{(K-l)!(n-2K+l)!} \frac{K!^2 (n-K)!^2 }{n!^2}\\
    &= \frac{1}{l!} \frac{1}{(K-l)!} \frac{1}{(K-l)!(n-2K+l)!} \frac{K!^2 (n-K)!^2 }{n!}\\
    &= \frac{K!}{l! (K-l)!} \frac{(n-K)!}{(K-l)!(n-2K+l)!} \frac{K! (n-K)!}{n!}\\
    &= \binom{K}{l} \binom{n-K}{K-l} \binom{n}{K}^{-1}.
\end{align*}
\end{proof}

\begin{lemma}\label{Lemma 4.2}
Define $\phi: [0, +\infty) \to \mathbb{R} : x \mapsto -x\log(x) + x -1$ (extended by continuity at $x=0$ : $\phi(0)=-1$). $\phi$ is strictly increasing on $[0,1]$, strictly decreasing on $[1,+\infty)$ with maximum value $0$ uniquely achieved at $x=1$.
\end{lemma}
\begin{proof}
We have for $x\in (0, +\infty)$ $\phi'(x) = -\log(x)$, the result then readily follows.
\end{proof}

\begin{lemma}\label{Lemma 4.4}
Let $\alpha\in (0,1)$, then for $l = \eta \frac{K^2}{n}$ and $\eta \leq \frac{n}{K}$, we have:
$$\left( \frac{Ke}{l}\right)^l  \left(\frac{(n-K)e}{K-l} \right) ^{K-l}  \left( \frac{K}{n} \right)^K \left( \frac{n-K}{n} \right)^{n-K} \leq   \exp\left(W_{n,K,l}  \right), $$
where $W_{n,K,l} \triangleq l \left(-\log(\eta) + 1 + \frac{K}{n} \right) $.
\end{lemma}
\begin{proof}

We have:
\begin{align*}
    \left( \frac{Ke}{l}\right)^l  \left(\frac{(n-K)e}{K-l} \right) ^{K-l}  \left( \frac{K}{n} \right)^K \left( \frac{n-K}{n} \right)^{n-K}
    &= \frac{ e^K K^{l+K} (n-K)^{n-l}}{l^l n^n (K-l)^{K-l}}\\
    &= \exp(\hat W_{n,K,l}),
\end{align*}
where:
\begin{align*}
    \hat W_{n,K,l} &\triangleq K + (l+K)\log(K)  - l\log(l) -n\log(n) + (n-l)\log(n-K)- (K-l)\log(K-l)
\end{align*}
Next we analyze the last two terms $U \triangleq (n-l)\log(n-K)$ and $V\triangleq (K-l)\log(K-l)$. We have:
\begin{align*}
    \log(n-K) &= \log(n) + \log\left( 1 - \frac{K}{n}\right)\\
    &= \log(n) - \frac{K}{n} - \frac{K^2}{2n^2} -\sum_{i\geq 3} \frac{K^i}{n^i}\\
    &\leq\log(n) - \frac{K}{n} - \frac{K^2}{2n^2}.
\end{align*}
Thus:
\begin{align*}
    U &= (n-l)\log(n-K)\\
    &\leq (n-l) \left[\log(n) - \frac{K}{n} - \frac{K^2}{2n^2}   \right]\\
    &= n\log(n) - K - \frac{K^2}{2n} - l\log(n) + \frac{lK}{n} + \frac{lK^2}{2n^2}.
\end{align*}
Next we have
\begin{align*}
    V &= (K-l) \log(K-l)\\
    &= (K-l)\left( \log(K) + \log\left(1-\frac{l}{K}\right)\right)\\
    &= K\log(K) - l \log(K) - (K-l) \sum_{m\geq 1}  \frac{1}{m}\left(\frac{l}{K}\right)^m\\
    &= K\log(K) - l \log(K) - \sum_{m\geq 1}  \frac{K}{m}\left(\frac{l}{K}\right)^m+ \sum_{m\geq 2}  \frac{K}{m-1}\left(\frac{l}{K}\right)^m\\
    &=  K\log(K) - l \log(K) - l +\sum_{m\geq 2}  \left( \frac{K}{m-1} - \frac{K}{m}\right)\left(\frac{l}{K}\right)^m\\
    &= K\log(K) - l \log(K) - l +\sum_{m\geq 2}  \frac{K}{m(m-1)}\left(\frac{l}{K}\right)^m\\
    &\geq K\log(K) - l \log(K) - l.
\end{align*}

Therefore
\begin{align*}
    \hat W_{n,K,l} &\leq  K + (l+K)\log(K)  - l\log(l) -n\log(n)  \\
    &+ n\log(n) - K - \frac{K^2}{2n} - l\log(n) + \frac{lK}{n} + \frac{lK^2}{2n^2}\\
    &-K \log(K) + l\log(K) + l\\ 
    &\leq  2l \log(K) -l\log(l) +l-l\log(n) + \frac{lK}{n}+\frac{lK^2}{2n^2} - \frac{K^2}{2n} \\
    &\leq l \left( \log \left( \frac{K^2}{nl}\right) +1 +   \frac{K}{n}  \right)  \\
    &= l\left(-\log(\eta) + 1 + \frac{K}{n}\right) \\
    &= W_{n,K,l},
\end{align*}
which completes the proof.

\end{proof}





\begin{lemma}\label{Lemma 4.5}
let $K  \in  [n]$. The following asymptotic relations hold:

\begin{alignat*}{2}
   \binom{n}{K} &= (1+o(1)) \frac{1}{\sqrt{2 \pi K }}\left( \frac{ne}{K}\right)^K   \text{ , }&& \omega(1) = K =o(\sqrt n )\\
   \binom{n}{K} &= (1+o(1)) \sqrt{\frac{n}{2 \pi K(n-K)} } \left( \frac{n}{K}\right)^k\left( \frac{n}{n-K}\right)^{n-K} \text{ , }&&\omega(1) = K =o(n)\\
     \log \left( \binom{n}{K}\right) &\sim K \log\left(\frac{n}{K}\right) \text{ , }&& K =o(n)\\
    \binom{n}{K} &\leq \left( \frac{K e}{n} \right)^K \text{ , }&& 1 \leq K \leq n\\
    \frac{\binom{n-K}{K}}{\binom{n}{K}} &= 1 + o(1) \text{ , }&& K=o(\sqrt{n})\\
     \frac{K^2}{n} \frac{\binom{n-K}{K}}{\binom{n}{K}} &=o(1) \text{ , } && K=\omega(\sqrt{n})
\end{alignat*}
\end{lemma}

\begin{proof}
The proof of the first 4 statements can be found  in \cite{das2020brief}. We now prove the remaining two statements. We have using Stirling formula:
\begin{align*}
 \frac{\binom{n-K}{K}}{\binom{n}{K}} &=  \frac{(n-K)! (n-K)!}{n! (n-2K)!}\\
&\sim \sqrt{ \frac{(n-K)(n-K)}{n (n-K)}}  \frac{(n-K)^{2n-2K}}{n^n (n-2K)^{n-2K}}\\
&\sim  \frac{(n-K)^{2n-2K}}{n^n (n-2K)^{n-2K}}\\
&=  \frac{(1-\frac{K}{n})^{2n-2K}}{(1-\frac{2K}{n})^{n-2K}}\\
&=  \exp \left(  (2n-2K) \log \left(1-\frac{K}{n}\right) - (n-2K)\log\left(1-\frac{2K}{n}\right)\right)\\
&=  \exp \left(  (2n-2K) \left(-\frac{K}{n} - \frac{K^2}{2 n^2}\right) - (n-2K)\left(-\frac{2K}{n} -\frac{2 K^2}{n^2}\right) + O\left( \frac{K^3}{n^2}\right)\right)\\
&=  \exp \left(  -\frac{K^2}{n} + O\left(\frac{K^3}{n^2}\right) \right) .\numberthis \label{ratio:binom}
\end{align*}
If $K=o(\sqrt{n})$ then $ \frac{\binom{n-K}{K}}{\binom{n}{K}} = 1 + o(1)$, and if $K=\omega(\sqrt{n})$ then $ \frac{K^2}{n} \frac{\binom{n-K}{K}}{\binom{n}{K}} \sim \frac{K^2}{n} \exp \left(  -\frac{K^2}{n} + O\left(\frac{K^3}{n^2}\right) \right) = o(1)$.

\end{proof}

\begin{lemma}\label{Lemma 4.80}
Recall $V(n,K), L(n,K)$ and $U(n,K)$ defined by (\ref{V(n,K)}), (\ref{L(n,K)}), (\ref{U(n,K)}) respectively. For  $\omega(1) = K  =o(n)$ identities (\ref{eq:Vnk}), (\ref{eq:VvsLUAbsolute}), and (\ref{LNK}) hold.
\end{lemma}
\begin{proof}
We first prove property (\ref{eq:Vnk}). We have:
\begin{align*}
    \frac{V(n,K)}{L(n,K)} &= \sqrt{\frac{ \log(\binom{n}{K})} {\log(\binom{n}{K} \frac{1}{K})}}\\
    &= \sqrt{ 1 + \frac{\log(K)}{\log(\binom{n}{K} \frac{1}{K})  } }\\
    &= 1 + O\left(\frac{\log(K)}{\log(\binom{n}{K} \frac{1}{K}) } \right)\\
    &= 1 + O\left( \frac{\log(K)}{K \log(\frac{n}{K}) - \log(K)}  \right) \numberthis \label{10.0} \\
    &= 1 + O\left(\frac{1}{K} \right).
\end{align*}
Where we used part 3 of Lemma \ref{Lemma 4.5} in (\ref{10.0}). Similar computations yield $\frac{V(n,K)}{U(n,K)}=1 + O\left(\frac{1}{K} \right)$. Next we prove property (\ref{eq:VvsLUAbsolute}). We have:
\begin{align*}
    V(n,K) - L(n,K) &= \sqrt{2 \binom{K}{2} \log\left(\binom{n}{K}\right)}  - \sqrt{2 \binom{K}{2} \log\left(\binom{n}{K} \frac{1}{K}\right)} \\
    &= \sqrt{2 \binom{K}{2} \log\left(\binom{n}{K}\right)}  \left( 1 - \sqrt{1 - \frac{\log(K)}{\log(\binom{n}{K})}}\right)\\
    &\sim \sqrt{K^2 K\log\left(\frac{n}{K}\right)} \frac{1}{2}\frac{\log(K)}{K\log(\frac{n}{K}) } \numberthis \label{11.0}\\
    &=O \left(  \sqrt{K \log(n)}\right),
\end{align*}
where we used part 3 of Lemma \ref{Lemma 4.5} in (\ref{11.0}). Similarly, we get $V(n,K) - U(n,K) = O\left(\sqrt{K \log(n)}\right)$. Finally, note that property (\ref{LNK}) follows immediately from property (\ref{eq:Vnk}) and part 3 of Lemma \ref{Lemma 4.5}.
\end{proof}

\begin{lemma}\label{Lemma 4.9}
Let $h : \left[0,  \frac{1}{2} \right] \to [0,1] $ be the binary entropy function defined by $h(x) = -x \log(x) - (1-x) \log(1-x)$. Then for $\epsilon = \epsilon_n \to 0$, it holds:
$$ h^{-1}(\log 2 - \epsilon) = \frac{1}{2} + \frac{1}{\sqrt{2}}\sqrt{\epsilon} - \frac{1}{6\sqrt{2}}\epsilon^{\frac{3}{2}} + O\left(\epsilon^{\frac{5}{2}} \right) .$$
\end{lemma}
\begin{proof}
We reproduce here the proof given for Lemma 11 in \cite{GamarnikZadik}.
Let $\Phi(x) \triangleq \sqrt{\log(2) - h(\frac{1}{2}) + x}, x \in[0, \frac{1}{2}]$. We calculate that for the sequence of derivatives at $0$, $a_i \triangleq \Phi^{(i)}(0), i\in \mathbb{Z}_{\geq 0}$ it holds $a_0 =0, a_1=\sqrt{2}, a_2=0, a_3=2\sqrt{2}$ and $a_4=0$. Notice that for all $\epsilon \in (0, \log(2))$ and $\Phi^{-1}$ the inverse of $\Phi$,
$$h^{-1} (\log(2) - \epsilon) = 1 + \Phi^{-1}(\sqrt{2}).$$
The result of the lemma follows if we establish that the Taylor expansion of $\Phi^{-1}$ around $y=0$ is given by 
\begin{align*} \Phi^{-1}(y) =\frac{1}{\sqrt{2}}y - \frac{1}{6\sqrt{2}}y^3  + O(y^5). \numberthis \label{Taylor}
\end{align*}
Clearly $\Phi^{-1}(0)=0$. For $b_i = \left(\Phi^{-1}\right)^{(i)}(0), i\in \mathbb{Z}_{\geq 0}$ by standard calculations using Lagrange inversion theorem we have $b_0=0$,
$$b_1=\frac{1}{a_1}=\frac{1}{\sqrt{2}},$$
$$b_2 = -\frac{a_2}{2a_1}=0,$$
$$b_3 = \frac{1}{2\sqrt{2}}\left[-\frac{a_3}{a_1}  + 3 \left(\frac{a_2}{a_1}\right)^2\right]=-\frac{1}{\sqrt{2}},$$
and
$$b_4 = \frac{1}{4} \left[-\frac{a_4}{a_1}+ \frac{10}{3} \frac{a_2a_3}{a_1^2} -60 \frac{a_2}{a_1}\right].$$
From this point, Taylor expansion yields that for small $y$
$$\Phi^{-1}(y) = b_0 + b_1y + \frac{b_2}{2}y^2 + \frac{b_3}{6}y^3 + \frac{b_4}{24} y^4 + O(y^5),$$
which given the values of $b_i, i=0,1,2,3,4$ yields (\ref{Taylor}). Which concludes the proof.

\end{proof}

%% file: Tail_Bounds.tex



\begin{theorem}\label{Theorem 4.8}
There exists a universal constant $ \beta$ such that for $x\ge \beta  K$:
$$ \p\left(\left| \Psi^{\Bern(\frac{1}{2})}_K(G) - \E[      \Psi^{\Bern(\frac{1}{2})}_K(G)   ] \right|\geq x\right) \leq 4\exp\left(-\frac{x^2}{4 K^2}\right) .$$
\end{theorem}
\begin{proof}
 We recall some definitions and notations related to the application of Talagrand's inequality in our case, as stated in Section 7.7 in \cite{AlonProbabilistic}:

\begin{defn}
Let $\Omega = \prod_{i=1}^{N} \Omega_i$, where $N \triangleq \binom{n}{2}$ and each $\Omega_i$ is a probability space associated to a $\Bern(1/2)$ random variable and $\Omega$ has the product measure. We call $h : \Omega \to \mathbb{R}$ Lipschitz if $|h(x)- h(y)|\leq 1 $ whenever $x,y$ differ in at most one coordinate.
Let $f: \mathbb{N} \to \mathbb{N}$. We say that $h$ is $f$-certifiable if, whenever $h(x) \geq s$, there exists $I \subset \{ 1,...,\binom{n}{2} \}$ with $|I| \leq f(s) $ so that all $y \in \Omega$  that agree with $x$ on the coordinates $I$ have $h(y) \geq s$.
\end{defn}

For our purposes, we consider $X = h(.)\triangleq \Psi^{\Bern(\frac{1}{2})}_K(G)$ where we think of $G$ as a vector in $\{0,1\}^{\binom{n}{2}}$, and note in particular that $|h(G) - h(G')|\leq 1$ whenever $G,G'$ differ in at most 1 coordinate (edge).

We claim that $h$  defined as above is $f$-certifiable with $f(s) \triangleq s$. Indeed, if $h(x) \geq s$, then at least $s$ edges associated with a size $K$ subset of $[n]$ have value $1$ according to $x$. We can then pick $I$ to be the subset of $s$ of these edges. Then any configuration $y$ of edges which includes these $s$ edges must correspond to a densest subgraph of value at least $s$, i.e $h(y)\geq s$. Next, we recall the Talagrand inequality, stated as Theorem 7.7.1 in \cite{AlonProbabilistic}:

\begin{theorem}
Assuming that $h$ is $f$-certifiable, we have for all $b,t$:
$$ \p \left(X \leq b - t\sqrt{f(b)} \right)\p\left( X \geq b \right) \leq e^{-\frac{t^2}{4}} . $$
\end{theorem}
Applying the above theorem in our case with $b \triangleq Median(X)$ and $f(s)\triangleq s$ yields:
$$ \p \left(X \leq Median(X) - t\sqrt{Median(X)} \right)\leq 2e^{-\frac{t^2}{4}} , $$
or equivalently:
$$ \p \left(X - Median(X) \leq - t \right)\leq 2e^{-\frac{t^2}{4 Median(X)}} . $$
Switching $X$ with $\binom{K}{2}-X$ (the certifiability argument still holds but this time we consider the non existing edges as certificates...) yields :
$$ \p \left(  -X + Median(X) \leq -t \right)\leq 2e^{-\frac{t^2}{4 Median\left(\binom{K}{2} - X\right)}} . $$
Since $0 \leq X \leq \binom{K}{2} \leq \frac{K^2}{2}$ we see that $Median(X), Median\left( \binom{K}{2}-X\right) \leq \frac{K^2}{2}$ so that the previous concentration inequalities yield:
$$ \p \left( | X - Median(X)| \geq t \right)\leq 4e^{-\frac{t^2}{2K^2}}  $$
We next claim that $|\E[X] - Median(X)|\leq 4K\sqrt{\pi }$. Indeed using the  concentration inequality above we have:
\begin{align*}
    |\E[X] - Median(X)| &\leq \E [| X -Median(X)| ]\\
    &= \int_{0}^{+\infty} \p \left( | X - Median(X)| \geq t \right) dt\\
    &\leq  \int_{0}^{+\infty}4e^{-\frac{t^2}{2K^2}} \\
    &= 4K\sqrt{2} \sqrt{\frac{\pi}{2}}\\
    &= 4K \sqrt{\pi}.
\end{align*}
Hence, for $t = 4K \sqrt{\pi } + u$ where  $u> 4K \sqrt{\pi }   $ we have:
$$
    \p \left( | X - \E[X]| \geq t \right) \leq \p \left( | X - Median(X)| \geq u \right) 
    \leq 4 e^{-\frac{u^2}{2K^2}} = 4e^{- \frac{(t-4K\sqrt{\pi})^2}{2K^2}} .
$$
Provided  $t \geq \beta K$ where $\beta$ is a constant, we obtain the claim.
\end{proof}

\begin{lemma}\label{Lemma 4.10}
Let $\epsilon_1,..., \epsilon_n$ be i.i.d Rademacher r.v, then if $S_n = \epsilon_1+...+\epsilon_n$, there exists a universal constant $\theta$ such that we have for $x \geq 1$ :
$$\p(S_n \geq x\sqrt{n}) \leq \theta \p(\mathcal{N}(0,1) \geq x).$$
\end{lemma}
\begin{proof}
Using Theorem 1.1 in \cite{Pinelis}, we see that there exists a universal constant $C\approx 14,10...$ such that $\p(S_n \geq x\sqrt{n}) \leq \p(\Nc(0,1) > x) \left(1 +\frac{C}{x}\right) \leq (1+C) \p(\Nc(0,1)\geq 0)$. The result of the Lemma follows by taking $\theta \triangleq 1 +C$.

\end{proof}

\begin{lemma}\label{Lemma 4.11}
Let $X_0 $ be a sum of $N_0$ rademacher i.i.d r.v and $ \hat X_1, \hat X_2$ be two sums of $\hat N$ rademacher i.i.d r.v. Then for $\beta \geq 0$:
$$ \p(X_1, X_2 \geq \beta) \leq 3(N_0+\hat N)\exp\left(- \frac{\gamma^2}{1+\rho}\right), $$
where $X_i \triangleq X_0 + \hat X_i , i=1,2$ , $\gamma \triangleq \frac{\beta}{\sqrt{N_0 + \hat N}}$ , and $\rho \triangleq \frac{N_0}{\hat N +  N_0 }$.
\end{lemma}
\begin{proof}

We have:
\begin{align*}
    \p(X_1, X_2 \geq \beta) &= \sum_{x} \p(X_0=x)\p(\hat X_1, \hat X_2 \geq (\beta  - x))\\
    &= \sum_{x, \beta-x \geq 0} \p(X_0=x)\p(\hat X_1 \geq (\beta  - x))^2 +\sum_{x, \beta-x < 0} \p(X_0=x)\p(\hat X_1 \geq (\beta  - x))^2 \\ 
    &\leq \sum_{x, \beta-x \geq 0} \p(|X_0|\geq |x|)\p(\hat X_1 \geq (\beta  - x))^2 +\sum_{x, \beta-x < 0} \p(X_0=x) .
\end{align*}
Using Hoeffding inequality we see that $\p( |X_0|\geq |x|)\leq 2\exp\left(- \frac{x^2}{2N_0} \right)$ and $\p(\hat X_1 \geq (\beta  - x)) \leq  \exp \left( -\frac{(\beta-x)^2}{\hat N} \right)$. Therefore:
\begin{align*}
      \p(X_1, X_2 \geq \beta)  &\leq \sum_{x, \beta-x \geq 0} 2\exp \left(- \frac{x^2}{2N_0} - \frac{(\beta - x)^2}{\hat N}\right) + \p(X_0  > \beta)\\
    &\leq 2(N_0+\hat N) \max_{x, \beta -x \geq 0} \exp\left(- \frac{x^2}{2N_0} - \frac{(\beta - x)^2}{\hat N}\right) + \exp\left(-\frac{\beta ^2}{2 N_0}\right) .
\end{align*}
The minimum of the polynomial $\frac{x^2}{2N_0} + \frac{(\beta - x)^2}{\hat N}$ is reached at $x= \frac{2\beta N_0}{2N_0 + \hat N}$. For this value we have
$$  - \frac{x^2}{2N_0} - \frac{(\beta - x)^2}{\hat N} = - \frac{2 \beta ^2 N_0}{(2N_0 + \hat N)^2} - \frac{\beta^2 \hat N}{(2N_0 + \hat N)^2} = -\frac{\beta^2}{2N_0 + \hat N}= - \frac{\gamma^2}{1+\rho}.$$
Moreover, we have:
$$-\frac{\beta ^2}{2 N_0} \leq -\frac{\beta^2}{2N_0 + \hat N} =   - \frac{\gamma^2}{1+\rho},$$
therefore:
\begin{align*}
    \p(X_1, X_2 \geq \beta)  &\leq 2(N_0 + \hat N) \exp\left(- \frac{\gamma^2}{1+\rho}\right) + \exp\left(- \frac{\gamma^2}{1+\rho}\right)\\
    &\leq (2(N_0+\hat N) +1) \exp\left(- \frac{\gamma^2}{1+\rho}\right)\\
    &\leq 3(N_0+\hat N) \exp\left(- \frac{\gamma^2}{1+\rho}\right).
\end{align*}
Which concludes the proof.
\end{proof}

\begin{lemma}\label{Lemma 4.12}
Let $X_1,...,X_n$ be Rademacher i.i.d r.v and $S_n = X_1+...+X_n$. Suppose $0 \leq \gamma = o(\sqrt{n})$. Then:
$$ \p(S_n \geq \gamma \sqrt{n} ) \geq \frac{1}{\sqrt{2n}} \exp\left( - \frac{\gamma^2}{2} -\frac{\gamma^4}{12n} + O\left(\frac{\gamma^5}{n^{3\over 2}}\right)\right). $$
\end{lemma}
\begin{proof}
We first cite the statement of Lemma 4.7.2 (p.115) in \cite{InformationTheory} :
\begin{lemma}\label{Lemma 4.13}
Suppose $\lambda > \frac{1}{2}$. Then:
$$ \sum_{k= \lambda n}^{n}  \binom{n}{k}  \geq  \frac{e^{- n  D(\lambda || \frac{1}{2})}}{\sqrt{8n \lambda(1-\lambda)}}, $$
where $D(x||p) \triangleq x \log\left(\frac{x}{p}\right) + (1-x) \log\left(\frac{1-x}{1-p}\right)$
\end{lemma}
Using the above Lemma:
\begin{align*}
    \p(S_n \geq \gamma \sqrt{n} ) &= \p\left(Binom\left(n, \frac{1}{2}\right) \geq \frac{\gamma \sqrt{n}}{2} + \frac{n}{2}\right)\\
    &\geq \frac{1}{\sqrt{8n\lambda(1-\lambda)}} \exp\left(-n D\left(\frac{\gamma }{2\sqrt{n}} + \frac{1}{2} || \frac{1}{2}\right)\right)\\
    &\geq \frac{1}{\sqrt{2n}} \exp\left(-n D\left(\frac{\gamma }{2\sqrt{n}} + \frac{1}{2} || \frac{1}{2}\right)\right).\\
\end{align*}
where we used  $\max_{x \in (0,1)} 8 x(1-x)=2$.
For $t=o(1)$:
\begin{align*}
D\left(t + \frac{1}{2} || \frac{1}{2}\right) &=  \left(t + \frac{1}{2}\right) \log(2t + 1) + \left(-t + \frac{1}{2}\right) \log( -2t+ 1) \\
&= \left(t + \frac{1}{2}\right)\left(2t - 2t^2 + \frac{8}{3}t^3 - 4t^4+O(t^5)\right) + \left(-t + \frac{1}{2}\right) \left(-2t - 2t^2-\frac{8}{3}t^3 - 4t^4 +O(t^5)\right) \\
&= 2t^2 + \frac{4}{3}t^4+  O(t^5),\\
\end{align*}
Hence:
$$D\left(\frac{\gamma }{2\sqrt{n}} + \frac{1}{2} || \frac{1}{2}\right)=  \frac{\gamma^2}{2 n} + \frac{\gamma^4}{12 n^2} + O\left (\frac{\gamma^5}{n^{5\over 2}}\right), $$
leading to:
$$\p(S_n \geq \gamma \sqrt{n} ) \geq \frac{1}{\sqrt{2n}} \exp\left(- \frac{\gamma^2}{2} -\frac{\gamma^4}{12n} + O\left (\frac{\gamma^5}{n^{3\over 2}}\right)\right) .$$

\end{proof}
The following Theorem can be found in \cite{Savage}.
\begin{theorem}\label{Theorem 4.14}
Let $X \sim \mathcal{N}(\mathbf{0}_{\mathbb{R}^n}, \Sigma)$ be a centered Gaussian random vector in $\mathbb{R}^n$ with non singular covariance matrix $\Sigma$, and let $ C \in \mathbb{R}^n$. Suppose $\Delta \triangleq C^T \Sigma^{-1}>0$ then:
$$ \p( X \geq C) \leq  \left(\prod_{i=1}^{n} \Delta_i\right)^{-1} \frac{|\Sigma^{-1}|^{\frac{1}{2}}}{(2\pi)^{\frac{n}{2}}} \exp\left(- \frac{1}{2} C^T \Sigma^{-1} C\right),$$
where $|A|$ denotes the absolute value of the determinant of $A$.
\end{theorem}

\begin{corollary}\label{Corollary 4.15}
If $X,Y$ are two centered Gaussian r.v with variance $\binom{K}{2}$ and covariance $\binom{l}{2}$ ($2 \leq l \leq K-1$) then for any $\gamma >0$ we have:

$$\p \left( X,Y \geq \gamma \binom{K}{2} \right) \leq \frac{1}{ 2 \pi \gamma^2} \frac{\left[ \binom{K}{2} + \binom{l}{2} \right]^2}{\binom{K}{2}^2 \sqrt{\binom{K}{2}^2 - \binom{l}{2}^2}} \exp \left(- \gamma^2 \frac{\binom{K}{2}^2}{\binom{K}{2} + \binom{l}{2}}\right) . $$
\end{corollary}
\begin{proof}
We apply Theorem \ref{Theorem 4.14} with $n=2, \mu = \begin{pmatrix}
0\\
0
\end{pmatrix} , \Sigma= \begin{pmatrix}
\binom{K}{2} & \binom{l}{2}\\
\binom{l}{2} & \binom{K}{2}
\end{pmatrix}$ and $ C = \begin{pmatrix}
\gamma \binom{K}{2}\\
\gamma \binom{K}{2}
\end{pmatrix}$. Note that $\Sigma$ is not singular since $ l \leq K-1$ and 
$|\Sigma^{-1}| = \frac{1}{\binom{K}{2}^2 - \binom{l}{2}^2}$ .We thus have $\Delta = \frac{\gamma \binom{K}{2}}{\binom{K}{2}^2 -\binom{l}{2}^2}  ( \binom{K}{2} - \binom{l}{2}, \binom{K}{2} - \binom{l}{2}  )$. Therefore we get:
\begin{align*}
    \p\left(X, Y \geq \gamma \binom{K}{2}\right) &\leq \left(\frac{\gamma \binom{K}{2}}{\binom{K}{2}^2 -\binom{l}{2}^2} \left(\binom{K}{2} - \binom{l}{2}\right) \right)^{-2} \frac{|M|^{\frac{1}{2}}}{2\pi} \exp\left(-\frac{1}{2} C^T M C\right)\\
    &= \frac{\left( \binom{K}{2}^2 - \binom{l}{2}^2\right)^2}{\gamma^2 \binom{K}{2}^2 \left( \binom{K}{2}- \binom{l}{2}\right)^2} \frac{1}{2\pi \sqrt{\binom{K}{2}^2 - \binom{l}{2}^2}} \exp\left( -\frac{1}{2} \gamma^2 \binom{K}{2}^2 \frac{2\binom{K}{2} - 2 \binom{l}{2}}{\binom{K}{2}^2 - \binom{l}{2}^2} \right)\\
    &= \frac{1}{ 2 \pi \gamma^2} \frac{\left[ \binom{K}{2} + \binom{l}{2} \right]^2}{\binom{K}{2}^2 \sqrt{\binom{K}{2}^2 - \binom{l}{2}^2}} \exp \left(- \gamma^2 \frac{\binom{K}{2}^2}{\binom{K}{2} + \binom{l}{2}}\right).
\end{align*}
\end{proof}
The following bound known as Borell-TIS inequality is standard in the theory of Gaussian processes and  can be found for example in  \cite{GaussianIneqs}.
\begin{theorem}\label{Theorem 4.16}
Given a sequence of centered Gaussian processes $(f_t)_{t\in T}$ (where $T$ is a topological space) such that $||f||_T \triangleq \sup_{t\in T} |f_t|$ is a.s finite. Let $\sigma^2 _T \triangleq \sup_{t\in T} \E[|f_t|^2]$. Then $\E[||f||_T] $ and $\sigma_T$ are both finite and for each $t>0$:
$$  \p\left(\sup_{t\in T} f_t > \E[\sup_{t\in T}f_t] + t\right) \leq \exp \left( -\frac{t^2}{2\sigma^2_T}\right),$$
and by symmetry:
$$  \p\left(|\sup_{t\in T} f_t - \E[\sup_{t\in T}f_t] | > t\right) \leq 2\exp \left( -\frac{t^2}{2\sigma^2_T}\right).$$
\end{theorem}

\begin{lemma}\label{Lemma 4.17}
 Let $S_n$ be a sum of $n$ Rademacher i.i.d random variables. Let  $w_n, m_n\geq 1 $ be such that $w_n m_n =o(\sqrt{n})$ and $w_n^5=o(n^{\frac{3}{2}})$. It holds
$$ \frac{\p\left( S_n \geq w_n \sqrt{n} -m_n \right)}{\p \left( S_n \geq w_n \sqrt{n} + m_n \right)} =O(1).$$
\end{lemma}
\begin{proof}
Note that $m_n = o(w_n\sqrt{n})$. We have:
\begin{align*}
\frac{\p\left( S_n \geq w_n \sqrt{n} -m_n \right)}{\p \left( S_n \geq w_n \sqrt{n} + m_n \right)}  &\leq 1 + \frac{\p\left( w_n \sqrt{n} +m_n \geq  S_n \geq w_n \sqrt{n} -m_n \right)}{\p \left( S_n \geq w_n \sqrt{n} + m_n \right)} \\
&\leq  1 + \frac{\p\left( w_n \sqrt{n} +m_n \geq  S_n \geq w_n \sqrt{n} -m_n \right)}{\p \left( w_n \sqrt{n} + 2m_n \geq S_n \geq w_n \sqrt{n} + m_n \right)} \\
&\leq 1 +  \frac{(2m_n+1) \p(S_n = w_n \sqrt{n} -m_n)}{(2m_n+1) \p(S_n = w_n \sqrt{n}+ 2m_n)}\\
&= 1 +  \frac{ \p(S_n = w_n \sqrt{n} -m_n)}{ \p(S_n = w_n \sqrt{n} +2m_n)}\\
&= 1 + \frac{ \p(Bin(n,\frac{1}{2}) = \frac{n}{2} +\frac{w_n \sqrt{n}}{2} -\frac{m_n}{2})}{ \p(Bin(n,\frac{1}{2}) = \frac{n}{2} + \frac{w_n \sqrt{n}}{2} +m_n)}\\
&=  1 + \frac{ \binom{n}{\frac{n}{2} + \frac{w_n \sqrt{n}}{2} -\frac{m_n}{2}}}{ \binom{n}{\frac{n}{2} + \frac{w_n \sqrt{n}}{2} +m_n}}\\
&\sim 1 + \frac{\sqrt{\frac{1}{n} } \exp(nh(\frac{1}{2} + \frac{w_n}{2\sqrt{n}} - \frac{m_n}{2n})) }{\sqrt{\frac{1}{n}}\exp(nh(\frac{1}{2} + \frac{w_n}{2\sqrt{n}} + \frac{m_n}{n})) }\\
&\sim 1 + \exp\left( n\left[h\left(\frac{1}{2} + \frac{w_n}{2\sqrt{n}} - \frac{m_n}{2n}\right) -h\left(\frac{1}{2} + \frac{w_n}{2\sqrt{n}} + \frac{m_n}{n}\right)  \right] \right)\\
&= 1 + \exp\left(n \left[ -2\left(\frac{w_n}{2\sqrt{n}} - \frac{m_n}{2n}\right)^2+ 2\left(\frac{w_n}{2\sqrt{n}} + \frac{m_n}{n}\right)^2 \right] \right)\\
&\times \exp \left (n\left[ - \frac{4}{3}\left(\frac{w_n}{2\sqrt{n}} - \frac{m_n}{2n}\right)^4 + \frac{4}{3}\left(\frac{w_n}{2\sqrt{n}} + \frac{m_n}{n}\right)^4 + O\left(\frac{w_n^5}{n^{\frac{5}{2}}}\right) \right]\right) \numberthis \label{hpart}\\
&= 1 + \exp\left(n \left[ -\frac{m_n^2}{2n^2} + \frac{w_n m_n}{n\sqrt{n}} + \frac{2m_n^2}{n^2} + \frac{2w_n m_n}{n\sqrt{n}} + O\left(\frac{w_n^3}{n^{3\over 2}} \frac{m_n}{n}\right)+ O\left(\frac{w_n^5}{n^{\frac{5}{2}}}\right) \right] \right)\\
&= 1 + \exp\left(- \frac{m_n^2}{2n} + \frac{w_n m_n}{\sqrt{n}} + \frac{2m_n^2}{n} + \frac{2w_n m_n}{\sqrt{n}} + o(1)\right) \numberthis \label{o1part1}\\
&= 1 + \exp( o(1)) \numberthis \label{o1part2}\\
&=O(1),
\end{align*}
where we used $h(\frac{1}{2}+ \delta) = \log(2) - 2\delta^2 - \frac{4}{3}\delta^4 + O(\delta^5)$ in (\ref{hpart}), and $\frac{w_n^3 m}{n}, \frac{w_n^5}{n^{\frac{3}{2}}} =o(1)$ in (\ref{o1part1}), and $m_n=o(\sqrt{n}), w_n m_n =o(\sqrt{n}) $ in (\ref{o1part2}). Which yields the claim of the Lemma.
\end{proof}

\begin{corollary}\label{Corollary 3.19}
Let $ K =n^{\alpha}$ with $\alpha\in(0,1), n \in \mathbb{Z}_{\geq 0}$ and $Z_S,Z_T$ be sums of $\binom{K}{2}$ Rademacher random variables such that they share exactly $\binom{l}{2}$ of them. Let $\gamma_n = (2 + \delta_n) \sqrt{\frac{\log(\frac{n}{K})}{K}}$ with $\delta_n =o(1)$ and let $ C_0 \geq 2$ be a constant independent of $n$. Then for any $l\leq C_0 \log(n)$:
$$ \frac{ \p\left( Z_S,Z_T \geq \gamma_n \binom{K}{2}\right)  }{ \p\left( Z_S \geq \gamma_n \binom{K}{2} \right)^2} = O(1) .$$
\end{corollary}
\begin{proof}
Write $Z_S = Z_1 + X, Z_T = Z_2 + X$ where $X$ is the sum of the shared $\binom{l}{2}$ Rademacher r.v so that $Z_1, Z_2$ are independents. We then have:
\begin{align*}
     \p\left( Z_S,Z_T \geq \gamma_n \binom{K}{2}\right)  &\leq  \p\left( Z_1,Z_2 \geq \gamma_n \binom{K}{2} - \binom{l}{2}\right)\\
     &= \p\left( Z_1 \geq \gamma_n \binom{K}{2} - \binom{l}{2}\right)^2.
\end{align*}
Similarly:
$$ \p\left( Z_S \geq \gamma_n \binom{K}{2} \right) \geq  \p\left( Z_1 \geq \gamma_n \binom{K}{2} + \binom{l}{2} \right).  $$
We then have:
\begin{align*}
  \frac{ \p\left( Z_S,Z_T \geq \gamma_n \binom{K}{2}\right)  }{ \p\left( Z_S \geq \gamma_n \binom{K}{2} \right)^2} &\leq \frac{ \p\left( Z_1 \geq \gamma_n \binom{K}{2} - \binom{l}{2}\right)^2}{ \p\left( Z_1 \geq \gamma_n \binom{K}{2} + \binom{l}{2} \right)^2}\\
  &= \left( \frac{\p\left( S_N \geq w_N \sqrt{N} -m_N \right)}{\p \left( S_N\geq w_N \sqrt{N} + m_N \right)} \right) ^2,
\end{align*}
where $S_N \triangleq Z_1, N \triangleq \binom{K}{2} - \binom{l}{2}$ ,  $m \triangleq \binom{l}{2}$ and $w_n \triangleq \gamma_n \frac{\binom{K}{2}}{\sqrt{ \binom{K}{2} -\binom{l}{2}  }}$. Note:
$$ w_N m_N = \gamma_n\frac{\binom{K}{2}}{\sqrt{\binom{K}{2} - \binom{l}{2}}} \binom{l}{2} \lesssim \gamma_n \sqrt{\binom{K}{2}}  \frac{l^2}{2}\lesssim 2 \sqrt{\frac{\log(\frac{n}{K})}{K}} \frac{K}{\sqrt 2} \frac{C_0^2\log(n)^2}{2} =\Theta\left( \sqrt{ K \log(n)^{5\over 2}} \right)=o\left(K \right) = o(\sqrt{N}).  $$
We are thus in the setting of Lemma \ref{Lemma 4.17}, which implies
$$\frac{\p\left( S_N \geq w_N \sqrt{N} -m_N \right)}{\p \left( S_N \geq w_N \sqrt{N} + m_N \right)} = O(1). $$
This concludes the proof. 
\end{proof}

%% file: Proof_of_Upper_Bounds_in_Theorems_3.1,3.2,3.3.tex
Applying first moment inequality yields:
\begin{align*}
     \p\left(U_{\gamma_n} \geq 1\right) &\leq \E[U_{\gamma_n}] \\
     &= \binom{n}{K} \p\left( Z_S \geq \gamma_n \binom{K}{2}\right)\\
     & \leq \binom{n}{K} \theta \p\left(\mathcal{N}(0,1) \geq \gamma_n \sqrt{\binom{K}{2}} \right)\numberthis \label{13}\\
     &\leq \theta \binom{n}{K} \frac{1}{ \sqrt{2 \pi}\gamma_n \sqrt{\binom{K}{2}}} \exp\left( - \frac{\gamma_n^2 \binom{K}{2}}{2}\right) \numberthis \label{14}\\
    &\sim \theta \binom{n}{K} \frac{1}{ \sqrt{2 \pi} 2\sqrt{\frac{\log(\frac{n}{K})}{K}} \sqrt{\frac{K^2}{2}}} \exp\left( - \frac{\gamma_n^2 \binom{K}{2}}{2}\right) \numberthis \label{15}\\
     &= \frac{\theta}{2\sqrt{\pi}} \binom{n}{K} \frac{1}{\sqrt{K \log(\frac{n}{K})}} \exp\left( - \frac{\gamma_n^2 \binom{K}{2}}{2}\right),
\end{align*}
where we used Lemma \ref{Lemma 4.10} in (\ref{13}), standard Gaussian tail bound in $(\ref{14})$\footnote{Recall that for $x>0$ it holds $\p(\mathcal{N}(0,1) \geq x )\leq \frac{1}{x\sqrt{2\pi}} \exp\left(-\frac{x^2}{2}\right)$}, and replaced $\gamma_n$ by it's asymptotic equivalent in (\ref{15}). We see then that in order to have $U_{\gamma_n}=0$ w.h.p as $n\to +\infty$, it suffices to pick $\gamma_n $(i.e, pick $\delta_n$) such that $\binom{n}{K} \frac{1}{\sqrt{K \log(\frac{n}{K})}} \exp\left( - \frac{\gamma_n^2 \binom{K}{2}}{2}\right) = \epsilon_n$ where $ \epsilon_n$ is any positive sequence s.t $\epsilon_n  =o(1)$. Taking the log of the former yields:
$$ \binom{n}{K} \frac{1}{\sqrt{K \log(\frac{n}{K})}} \exp\left( - \frac{\gamma_n^2 \binom{K}{2}}{2}\right) = \epsilon_n \iff   \gamma_n = \sqrt{ \frac{2}{\binom{K}{2}}} \sqrt{       -\log(\epsilon_n)  + \log\left(   \binom{n}{K}\frac{1}{\sqrt{ K \log(\frac{n}{K})}}  \right)    } .$$
We choose $\epsilon_n$ s.t $ \log(\epsilon_n) = o\left( \log \left( \binom{n}{K} \frac{1}{\sqrt{K \log(\frac{n}{K})}} \right)\right)  = o\left( K \log(n) \right) $  where the second equality comes from part 3 of Lemma \ref{Lemma 4.5}. We then have: 

\begin{align*}
 &\gamma_n = \sqrt{ \frac{2}{\binom{K}{2}}}\sqrt{       -\log(\epsilon_n)  + \log\left(   \binom{n}{K}\frac{1}{\sqrt{ K \log(\frac{n}{K})}}  \right)    }  \\
 &\iff \gamma_n = \sqrt{ \frac{2}{\binom{K}{2}}} \sqrt{ \log\left(   \binom{n}{K}\frac{1}{\sqrt{ K \log(\frac{n}{K})}}  \right)   \left( 1 + \frac{-\log(\epsilon_n)}{\log\left( \binom{n}{K} \frac{1}{\sqrt{K \log(\frac{n}{K})}}\right)}\right) } \\
 &\iff \gamma_n = \sqrt{ \frac{2}{\binom{K}{2}}} \sqrt{ \log\left(   \binom{n}{K}\frac{1}{\sqrt{ K \log(\frac{n}{K})}}  \right) }   \left( 1 + O\left( \frac{-\log(\epsilon_n)}{\log\left( \binom{n}{K} \frac{1}{\sqrt{K \log(\frac{n}{K})}}\right)}\right) \right) \numberthis \label{16}\\
 &\iff \gamma_n = \sqrt{ \frac{2}{\binom{K}{2}}} \left(    \sqrt{      \log\left(   \binom{n}{K}\frac{1}{\sqrt{ K \log(\frac{n}{K})}}  \right)    }  + O\left( \frac{-\log(\epsilon_n)}{ \sqrt{      \log\left(   \binom{n}{K}\frac{1}{\sqrt{ K \log(\frac{n}{K})}}  \right)    }}\right)  \right)\\
    &\iff \gamma_n =   \sqrt{ \frac{2}{\binom{K}{2}}     \log\left(   \binom{n}{K}\frac{1}{\sqrt{ K \log(\frac{n}{K})}}  \right)    }  + O\left(\frac{-\log(\epsilon_n)}{K^{\frac{3}{2}}\sqrt{\log n}}\right), \\
\end{align*}
where \ref{16} follows from using the Taylor expansion of $\sqrt{1 + u}$ when $u\to 0$ with $u =\frac{-\log(\epsilon_n)}{\log\left( \binom{n}{K} \frac{1}{\sqrt{K \log(\frac{n}{K})}}\right)}=o(1) $. It remains to check that the above expression of $\gamma_n$ is consistent with its definition in terms of $\delta_n$. That is it suffices to prove that :
$$\sqrt{ \frac{2}{\binom{K}{2}}     \log\left(   \binom{n}{K}\frac{1}{\sqrt{ K \log(\frac{n}{K})}}  \right)    }  + O\left(\frac{-\log(\epsilon_n)}{K^{\frac{3}{2}}\sqrt{\log n}}\right) \sim 2 \sqrt{\frac{\log(\frac{n}{K})}{K}}.$$
We have by part 3 of Lemma \ref{Lemma 4.5}:
\begin{align*}
    \sqrt{ \frac{2}{\binom{K}{2}}     \log\left(   \binom{n}{K}\frac{1}{\sqrt{ K \log(\frac{n}{K})}}  \right)    } &= \sqrt{ \frac{2}{\binom{K}{2}}     \log\left(   \binom{n}{K}\right) -  \frac{2}{\binom{K}{2}}    \log\left(\sqrt{K \log\left(\frac{n}{K}\right)}\right)   }\\
    &\sim \sqrt{ \frac{2}{\frac{K^2}{2}} K \log\left( \frac{n}{K} \right) }\\
    &=2\sqrt{\frac{\log(\frac{n}{K})}{K}},
\end{align*}
and:
\begin{align*}
    O\left(\frac{-\log(\epsilon_n)}{K^{\frac{3}{2}}\sqrt{\log n}}\right) &= o\left( \frac{K \log(n)}{K^{\frac{3}{2}} \sqrt{\log(n)}} \right)\\
    &= o\left(\sqrt{\frac{\log(\frac{n}{K})}{K}}\right).
\end{align*}
Thus $\gamma_n \triangleq \sqrt{ \frac{2}{\binom{K}{2}}     \log\left(   \binom{n}{K}\frac{1}{\sqrt{ K \log(\frac{n}{K})}}  \right)    }  + O\left(\frac{-\log(\epsilon_n)}{K^{\frac{3}{2}}\sqrt{\log n}}\right) \sim 2 \sqrt{\frac{\log(\frac{n}{K})}{K}}$. We then have  w.h.p as $n\to +\infty$ that $U_{\gamma_n}=0$, equivalently:
\begin{align*}
\Psi^{\Rc}_{K}(G)  &\leq  \binom{K}{2} \left( \sqrt{ \frac{2}{\binom{K}{2}}     \log\left(   \binom{n}{K}\frac{1}{\sqrt{ K \log(\frac{n}{K})}}  \right)    }  + O\left(\frac{-\log(\epsilon_n)}{K^{\frac{3}{2}}\sqrt{\log(n)}}\right) \right)\\
&= \sqrt{2 \binom{K}{2} \log\left(   \binom{K}{n}\frac{1}{\sqrt{ K \log(\frac{n}{K})}}  \right)  } + O\left(-\log(\epsilon_n) \sqrt{\frac{K}{\log(n)}}\right).
\end{align*}

Recalling that $\epsilon_n$ was any sequence satisfying $\epsilon_n \to 0, \log(\epsilon_n) = o(K \log n)$, we obtain an upper bound:
$$\Psi^{\Rc}_{K}(G)  \leq \sqrt{2 \binom{K}{2} \log\left(   \binom{K}{n}\frac{1}{\sqrt{ K \log(\frac{n}{K})}}  \right)  } + O\left(a_n \sqrt{\frac{K}{\log(n)}}\right) ,$$
where $a_n$ is any sequence satisfying $a_n \to +\infty, a_n = o(K \log n)$. Furthermore, since this is an upper bound the condition $a_n = o(K \log n)$ can be dropped.

%% file: Main_Proposition_Proof.tex
The second moment arguments we are about to use  will require us to control the following type of sums
$$\sum_{2 \leq l \leq  K-1} \binom{K}{l} \binom{n-K}{K-l} \binom{n}{K}^{-1} K^m \exp\left( \gamma_n ^2 \frac{\binom{K}{2} \binom{l}{2}}{\binom{K}{2} + \binom{l}{2}} \right) \triangleq \sum_{2 \leq l \leq K-1} S_{n,K,m,l},$$
where $\gamma_n \triangleq (1+\delta_n) 2\sqrt{\frac{\log(\frac{n}{K})}{K}}$, $\delta_n$ is a sequence s.t $
\delta_n =o(1)$ and $m$ is a fixed integer. We summaries the main properties of the above summations in the following proposition.
\begin{proposition}\label{Proposition 6.1}
The following holds
\begin{alignat*}{5}
      \exp(K)&\smashoperator[l]{\sum_{l = \lceil \log(K) \rceil}^{K-1}} &&S_{n,K,m,l}  &&= o(1),  \quad &&\alpha \in \left(0,\frac{1}{2}\right) \\
     &\smashoperator[l]{\sum_{l = 2}^{\lfloor \frac{K^2}{n}\log(K) \rceil} }&&S_{n,K,m,l} &&\lesssim K^m,  \quad&& \alpha \in \left [\frac{1}{2}, \frac{2}{3}\right) \\
    &\smashoperator[l]{\sum_{ l = \lceil \frac{K^2}{n}\log(K) \rceil}^{  K-1} }&&S_{n,K,m,l} &&= o(1),  \quad&& \alpha \in \left[\frac{1}{2}, \frac{2}{3}\right)\\
    &\smashoperator[l]{\sum_{l = 2}^{\ K-1}} &&S_{n,K,m,l}&&= O\left(\exp\left(\frac{K^3}{n^2} \log(n) \right)\right),  \quad&& \alpha \in \left[\frac{2}{3}, 1\right)
\end{alignat*}
\end{proposition}
\begin{proof}[Proof of Proposition \ref{Proposition 6.1}]$ $
\subsection{Case : \texorpdfstring{$\alpha\in \left(0, \frac{1}{2}\right)$}{TEXT}}
Write each index $l$ as $\eta \frac{K^2}{n}$ where $ \frac{n \log(K)}{K^2} \leq \eta \leq \frac{n(K-1)}{K^2 } $, in particular note that $\eta = \omega(1)$. We have by part 4 of Lemma \ref{Lemma 4.5} :
\begin{align*}
    \binom{K}{l} \binom{n-K}{K-l} \binom{n}{K}^{-1}  &\leq   \left( \frac{Ke}{l} \right)^l \left( \frac{(n-K)e}{K-l}\right)^{K-l} \binom{n}{K}^{-1}.
\end{align*}
Hence:
\begin{align*}
\hspace*{-0.5cm}
     \sum_{ \lceil \log(K) \rceil \leq l \leq K-1} S_{n,K,m,l}   &\leq \binom{n}{K}^{-1}\sum_{\lceil \log(K) \rceil  \leq l \leq K-1} \left( \frac{Ke}{l} \right)^l \left( \frac{(n-K)e}{K-l}\right)^{K-l}  K^m \exp\left( \gamma_n ^2 \frac{\binom{K}{2} \binom{l}{2}}{\binom{K}{2} + \binom{l}{2}} \right)\\
    &\sim \sqrt{2\pi K} \left(\frac{K}{n} \right)^{K} \left(\frac{n-K}{n} \right)^{n-K} \\ & \times\sum_{\lceil \log(K) \rceil  \leq l \leq K-1} \left( \frac{Ke}{l} \right)^l \left( \frac{(n-K)e}{K-l}\right)^{K-l}  K^m \exp\left( \gamma_n ^2 \frac{\binom{K}{2} \binom{l}{2}}{\binom{K}{2} + \binom{l}{2}} \right) \numberthis \label{17}\\ 
     &\leq \sum_{\lceil \log(K) \rceil  \leq l \leq K-1} K^m\sqrt{2\pi K} \exp\left(W_{n,K,l} +\gamma_n^2 \frac{\binom{K}{2} \binom{l}{2}}{\binom{K}{2} + \binom{l}{2}} \right) \numberthis \label{18}\\
      &\leq \sum_{\lceil \log(K) \rceil  \leq l \leq K-1} 3 K^{m+1} \exp\left(W_{n,K,l} +\gamma_n^2 \frac{\binom{K}{2} \binom{l}{2}}{\binom{K}{2} + \binom{l}{2}} \right).\\
\end{align*}
Where $W_{n,K,l} \triangleq l\left(-\log(\eta) + 1+\frac{K}{n}\right) $ and we used part 2 of  Lemma \ref{Lemma 4.5} in (\ref{17}) on $\binom{n}{K}^{-1}$, and Lemma \ref{Lemma 4.4} in (\ref{18}). Note that for $a,b \geq 1$ it holds
$$ \frac{ \binom{a}{2} \binom{b}{2}}{\binom{a}{2} + \binom{b}{2}} \leq \frac{1}{2}\frac{a^2 b^2}{a^2 + b^2} .$$
We then  have for $l\geq \lceil \log(K) \rceil $ :
\begin{align*}
    \gamma_n^2 \frac{\binom{K}{2} \binom{l}{2}}{\binom{K}{2} + \binom{l}{2}} &\leq (1+ \delta_n)^2 4\frac{ \log\left(\frac{n}{K}\right)}{K} \frac{1}{2}\frac{K^2 l^2}{K^2 + l^2}\\
    &= 2\log\left(\frac{n}{K}\right) \frac{K}{n} \eta  \frac{l}{1+\eta^2 \frac{K^2}{n^2}} (1+ \delta_n)^2\\
    &= 2\log\left(\frac{n}{K}\right) \frac{K}{n} \eta  \frac{l}{1+\eta^2 \frac{K^2}{n^2}} (1+ \xi_n),
\end{align*}
where $\xi_n \triangleq 2\delta_n +\delta_n^2 = o(1)$. For large enough $n$ we have:
\begin{align*}
    W_{n,K,l} +\gamma_n^2 \frac{\binom{K}{2} \binom{l}{2}}{\binom{K}{2} + \binom{l}{2}} &\leq l\left(-\log(\eta) +1 +\frac{K}{n}\right) + 2(1+\xi_n)\log\left(\frac{n}{K}\right) \frac{K}{n} \eta  \frac{l}{1+\eta^2 \frac{K^2}{n^2}} .
\end{align*}
Let $x \triangleq \eta \frac{K}{n} = \frac{l}{K}$. Note that $o(1) = \log(K)/K \leq x < 1$ and
$$\log\left(\frac{n}{K}\right) \frac{K}{n} \eta  \frac{l}{1+\eta^2 \frac{K^2}{n^2}} = \log\left(\frac{n}{K} \right) \frac{xl}{1+x^2} . $$
Therefore:
\begin{align*}
    W_{n,K,l}+\gamma_n^2 \frac{\binom{K}{2} \binom{l}{2}}{\binom{K}{2} + \binom{l}{2}} &\leq l\left(-\log\left(\frac{nx}{K}\right) +1+ 2\log\left(\frac{n}{K}\right) \frac{x(1+\xi_n)}{1+x^2}  + \frac{K}{n} \right)\\
    &= l\left( - \log(x)  - \frac{(x-1)^2}{1+x^2} \log\left(\frac{n}{K}\right)+ \frac{2x\xi_n}{1+x^2}\log\left( \frac{n}{K} \right) +1+ \frac{K}{n}\right).
\end{align*}
For $K,n$ large enough we have:
\begin{align*}
    &l\left( - \log(x)  - \frac{(x-1)^2}{1+x^2} \log\left(\frac{n}{K}\right)+ \frac{2x\xi_n}{1+x^2}\log\left( \frac{n}{K} \right)  +1+ \frac{K}{n}\right)\\  &\leq l\left( - \log(x)  - \frac{(x-1)^2}{1+x^2} \log\left(\frac{n}{K}\right) + \frac{2x\xi_n}{1+x^2}\log\left( \frac{n}{K} \right) +2\right)  \\
    &= K \left( - x\log(x)  - \frac{x(x-1)^2}{1+x^2} \log\left(\frac{n}{K}\right) + \frac{2x^2\xi_n}{1+x^2}\log\left( \frac{n}{K} \right) +2x\right) \numberthis \label{d(x)}\\
    &\leq K g_n(x),
\end{align*}
where  
\begin{align*}
    g_n(x) \triangleq  - x\log(x)  - \frac{x(x-1)^2}{1+x^2} \log\left(\frac{n}{K}\right) + 2|\xi_n|\log\left( \frac{n}{K} \right) +2x. \numberthis \label{g_n}
\end{align*}
Let $x_1$ be the first solution to the equation $\frac{4x}{(1+x^2)^2} = \frac{1-2\alpha}{1-\alpha}$ and note then that on $(0, x_1)$ we have by  monotonicity $\frac{4x}{(1+x^2)^2} < \frac{1-2\alpha}{1-\alpha}$ . We claim that $g_n(x)$ is strictly decreasing on $\left(\frac{\log(K)}{K}, x_1\right)$. Indeed, we have:
$$ \frac{\partial g_n(x)}{\partial x} = 1 - \log \left( \frac{n}{K} \right) - \log(x) + \frac{4x}{(1+x^2)^2} \log\left(\frac{n}{K}\right) .$$
For $x \in \left(\frac{\log(K)}{K}, x_1\right)$:
\begin{align*}
    \frac{\partial g_n(x)}{\partial x} &\leq 1 - \log\left(\frac{n}{K}\right) - \log\left(\frac{\log(K)}{K}\right) + \frac{4x}{(1+x^2)^2} \log\left(\frac{n}{K}\right)\\
    &= 1-\log(\log(K)) - \log\left(\frac{n}{K^2}\right)+ \frac{4x}{(1+x^2)^2} \log\left(\frac{n}{K}\right)\\
    &= 1- \log(\log(K)) - (1-2\alpha) \log(n) + (1-\alpha) \frac{4x}{(1+x^2)^2}  \log(n)\\
    &= 1- \log(\log(K)) +(1-2\alpha) \log(n)\left[-1 + \frac{1-\alpha}{1-2\alpha}\frac{4x}{(1+x^2)^2}  \right] ,
\end{align*}
and since $\frac{1-\alpha}{1-2\alpha}\frac{4x}{(1+x^2)^2} -1 <0$ for $x< x_1$ we see that for sufficiently large $n,K$ the above is negative, and thus $g_n$ is strictly decreasing over $\left(\frac{\log(K)}{K}, x_1\right)$.

Now let $x_2 \triangleq \frac{1}{\sqrt{3}}$. For $x \in \left( x_2, 1 \right)$ the function $x \mapsto\frac{4x}{(1+x^2)^2}$ reaches it maximum $\frac{3\sqrt 3}{4}>1$ at $x = x_2$ and takes values strictly larger than $1$ on the interval $(x_2, 1)$, we then have for $x \in(x_2,1)$:

\begin{align*}
     \frac{\partial g_n(x)}{\partial x} &\geq 1 - \log\left(\frac{n}{K}\right) - \log(1) +  \frac{4x}{(1+x^2)^2} \log\left(\frac{n}{K}\right)\\
     &\geq 1  - \log\left(\frac{n}{K}\right) +  \log\left(\frac{n}{K}\right)\\
     &\geq 1.
\end{align*}
Therefore $g_n$ is strictly increasing over $(x_2,1)$. We can then upper bound the summations of $S_{n,K,m,l}$ over the values $x\in \left(\frac{\log(K)}{K}, x_1\right) \cup (x_2, 1)$ by their associated integrals. For clarity we use the notation $x(l) = x = \frac{l}{K}$ to indicate that $x(l)$ depends on $l$. Noting that $l \mapsto 3K^{m+1} \exp(K g_n (x(l)))$ is decreasing over $l \in \left(\log(K), K x_1  \right)$ we have:

\begin{align*}
    \sum_{ \lceil \log(K) \rceil \leq l \leq \lfloor Kx_1 \rfloor} S_{n,K,m,l} \exp(K)&\lesssim \sum_{\lceil \log(K) \rceil\leq l \leq \lfloor Kx_1 \rfloor} 3 K^{m+1} \exp\left(W_{n,K,l} +\gamma_n^2 \frac{\binom{K}{2} \binom{l}{2}}{\binom{K}{2} + \binom{l}{2}} \right) \exp(K)\\
    &\leq\sum_{\lceil \log(K) \rceil \leq l \leq \lfloor Kx_1 \rfloor} 3K^{m+1} \exp(K g_n(x(l)) + K) \\
    &\leq  \sum_{\lceil \log(K) \rceil \leq l \leq \lfloor Kx_1 \rfloor}  \int_{l-1}^{l} 3K^{m+1} \exp(K g_n(x(u))+K)  du\\
    &\leq \int_{\log(K) -1}^{K x_1} 3K^{m+1}  \exp( K g_n(x(u))+K) du\\
    &= \int_{\frac{\log(K)}{K}- \frac{1}{K}}^{x_1}3K^{m+1}  \exp( K  g_n(x)+K) K dx\\
    &= \int_{\frac{\log(K)}{K}-\frac{1}{K}}^{x_1} 3K^{m+2}  \exp( K  g_n(x)+K) dx \\
    &= o(1),
\end{align*}
where the last line follows by applying the Monotone Convergence Theorem. Indeed, note that for any $x \in (0,1)$, the sequence of functions $x\to K^{m+2}\exp(K g_n(x) +K)$ is monotonically decreasing in $n$. Moreover, for large enough $n$ we have $g_n(x)<-2$. Therefore, this is a monotonically decreasing sequence of functions converging pointwise to the null function over $(0,1)$. Thus the application of MCT. One proves Similarly that
$$ \sum_{ \lceil Kx_2 \rceil \leq l \leq K-1} S_{n,K,m,l} \leq  \sum_{\lceil K x_2 \rceil\leq l \leq K-1} 3K^{m+1}  \exp(K g_n(x(l))+K) = o(1) .$$
Hence, in order to establish that $\sum_{\lceil \log(K) \rceil \leq l \leq K-1} S_{n,K,m,l} \exp(K)= o(1) $, it suffices to prove that:
$$    \sum_{\lceil K x_1 \rceil  \leq l \leq \lfloor Kx_2 \rfloor}3K^{m+1}  \exp(K g_n(x(l))+K) =o(1) .$$
Let $x\in [x_1, x_2]$ and note that $\max_{ x \in (0,1)} -x\log(x) = e^{-1}$, therefore:
\begin{align*}
    g_n(x) &=  - x\log(x)  - \frac{x(x-1)^2}{1+x^2} \log\left(\frac{n}{K}\right)+ 2|\xi_n| \log\left(\frac{n}{K}\right) +2x\\
    &\leq e^{-1}+ \left[2|\xi_n|- \min_{x \in (x_1,x_2)}   \left(\frac{x(x-1)^2}{1+x^2}\right)\right] \log\left(\frac{n}{K}\right) + 2\\
    &= e^{-1} +2 + (2|\xi_n| - \beta)  \log\left(\frac{n}{K}\right),
\end{align*}
where $\beta \triangleq \min_{x \in (x_1,x_2)}  \left(\frac{x(x-1)^2}{1+x^2}\right) >0$ and $\beta$ only depends on the parameter $\alpha$.  Therefore, for large enough $K,n$ we have $ g_n(x) \leq - \frac{\beta}{2} \log(\frac{n}{K})-1$ on $[x_1, x_2]$, and thus:
\begin{align*}
W_{n,K,l} + \gamma_n^2 \frac{\binom{K}{2} \binom{l}{2}}{\binom{K}{2}+\binom{l}{2}} +K&\leq K g_n(x) +K\\
&\leq - K \log\left(\frac{n}{K}\right) \frac{\beta}{2}.
 \end{align*}
 Hence:
 \begin{align*}
     \sum_{ \lceil K x_1 \rceil  \leq l \leq \lfloor K x_2 \rfloor} 3 K^{m+1} \exp(K g_n(x(l))+K) &\leq \sum_{\lceil K x_1 \rceil \leq l \leq \lfloor K x_2 \rfloor} 3K^{m+1}  \exp\left(- K \log\left(\frac{n}{K}\right) \frac{\beta}{2}\right)\\
     &\leq (Kx_2 - K x_1+1) 3K^{m+1}  \exp\left(- K \log\left(\frac{n}{K}\right) \frac{\beta}{2}\right)\\
     &\leq 3K^{m+2}  \exp\left(- K \log\left(\frac{n}{K}\right) \frac{\beta}{2}\right)\\
     &= o(1),
 \end{align*}
which concludes the proof of $\sum_{\lceil \log(K) \rceil \leq l \leq K-1} S_{n,K,m,l} \exp(K)= o(1)$.\\

\subsection{Case : \texorpdfstring{$\alpha\in \left[ \frac{1}{2}, \frac{2}{3}\right)$}{TEXT}} $ $\newline
Let: 
$$ \sum_{2 \leq l \leq K-1} S_{n,K,m,l} = B_1 +B_2, $$
where $B_1$ is the sum of $S_{n,K,m,l}$ for $2 \leq l \leq  \lfloor \frac{K^2}{n} \log(K) \rfloor $, $B_2$ for $\lceil \frac{K^2}{n} \log(K) \rceil\leq l \leq K-1$.

\subsubsection*{Analysis of $B_1$}

For large enough $n$ we have $\gamma_n \leq 3 \sqrt{\frac{\log(\frac{n}{K})}{K}} \leq 3 \sqrt{\frac{\log(n)}{K}}$, therefore:
\begin{align*}
    \gamma_n^2  \frac{\binom{K}{2} \binom{l}{2}}{\binom{K}{2} + \binom{l}{2}} &\leq 9 \frac{\log(n)}{K} \binom{l}{2}\\
    &\leq 9 \frac{\log(n)}{K} \frac{l^2}{2}\\
    &\leq \frac{9}{2} \frac{\log(n)}{K} \frac{K^4}{n^2}\log(K)^2\\
    &\leq \frac{9}{2}  \frac{K^3}{n^2} \log(n)^3.
\end{align*}
We thus have:
\begin{align*}
B_1& \leq \sum_{2\leq l \leq\frac{K^2}{n} \log(n)}  \binom{K}{l} \binom{n-K}{K-l} \binom{n}{K}^{-1} K^m  \exp \left(   \frac{9}{2} \frac{K^3}{n^2}\log\left(n\right)^3\right) \\
&\leq K^m \exp \left( \frac{9}{2}\frac{K^3}{n^2}\log\left(n\right)^3\right) \sum_{0\leq l \leq K }  \binom{K}{l} \binom{n-K}{K-l} \binom{n}{K}^{-1} \\
&= K^m \exp \left( \frac{9}{2} \frac{K^3}{n^2}\log\left(n\right)^3 \right) \numberthis \label{19} ,
\end{align*}
where we used the combinatorial identity of Lemma (\ref{Lemma 4.1}) in \ref{19}. Since $\alpha < \frac{2}{3}$ we have $ \exp \left(  \frac{9}{2} \frac{K^3}{n^2}\log\left(n\right)^3 \right)= \exp(o(1)) =1$, henceforth
$$  B_1 \lesssim K^m.$$
\subsubsection*{Analysis of $B_2 $}
Using part 4 of Lemma \ref{Lemma 4.5} we have :
\begin{align*}
    \binom{K}{l} \binom{n-K}{K-l} \binom{n}{K}^{-1}  &\leq   \left( \frac{Ke}{l} \right)^l \left( \frac{(n-K)e}{K-l}\right)^{K-l} \binom{n}{K}^{-1}.
\end{align*}
Hence:
\begin{align*}
\hspace*{-0.3cm}
     B_2  &\leq \binom{n}{K}^{-1}\sum_{  l = \lceil \frac{K^2}{n} \log(K)\rceil}^{ K-1} \left( \frac{Ke}{l} \right)^l \left( \frac{(n-K)e}{K-l}\right)^{K-l}  K^m \exp\left( \gamma_n ^2 \frac{\binom{K}{2} \binom{l}{2}}{\binom{K}{2} + \binom{l}{2}} \right)\\
    &\sim \sqrt{2\pi K} \left(\frac{K}{n} \right)^{K} \left(\frac{n-K}{n} \right)^{n-K}\sum_{  l = \lceil \frac{K^2}{n} \log(K) \rceil}^{K-1} \left( \frac{Ke}{l} \right)^l \left( \frac{(n-K)e}{K-l}\right)^{K-l}  K^m \exp\left( \gamma_n ^2 \frac{\binom{K}{2} \binom{l}{2}}{\binom{K}{2} + \binom{l}{2}} \right) \numberthis \label{20}\\ 
     &\leq \sum_{ l=\lceil  \frac{K^2}{n} \log(K) \rceil}^{ K-1} K^m\sqrt{2\pi K} \exp\left(W_{n,K,l} +\gamma_n^2 \frac{\binom{K}{2} \binom{l}{2}}{\binom{K}{2} + \binom{l}{2}} \right)  \numberthis \label{21} \\
     &\leq \sum_{  l= \lceil \frac{K^2}{n} \log(K) \rceil }^{ K-1} 3K^{m+1} \exp\left(W_{n,K,l} +\gamma_n^2 \frac{\binom{K}{2} \binom{l}{2}}{\binom{K}{2} + \binom{l}{2}} \right),
\end{align*}
where $W_{n,K,l} \triangleq l\left(-\log(\eta) + 1+ \frac{K}{n}\right), l= \eta \frac{K^2}{n}$ and we used part 2 of Lemma \ref{Lemma 4.5} in (\ref{20}) and Lemma \ref{Lemma 4.4} in (\ref{21}). 
Similarly to previous computations (see case $\alpha \in \left(0, \frac{1}{2}\right)$), we have for large enough $n$:
\begin{align*}
    W_{n,K,l} +\gamma_n^2 \frac{\binom{K}{2} \binom{l}{2}}{\binom{K}{2} + \binom{l}{2}}  &\leq l\left( - \log(x)  - \frac{(x-1)^2}{1+x^2} \log\left(\frac{n}{K}\right)  + 2|\xi_n| \log\left(\frac{n}{K}\right)+2\right)\\
    &= K g_n(x),
\end{align*}
where $x \triangleq \eta \frac{K}{n} = \frac{l}{K}$ and $g_n(x)$ is given by (\ref{g_n}).
Let $x^{n}_1 \triangleq \frac{p}{\log(\frac{n}{K})}$ where $p$ is a constant that we will fix later, and $x_2 \triangleq \frac{1}{\sqrt 3}$. We claim that  $g_n(x)$ is strictly decreasing on $\left(\frac{K}{n} \log(K), x^n_1\right)$ and strictly increasing on $(x_2,1)$. Indeed, we have:
$$ \frac{\partial g_n(x)}{\partial x} = 1 - \log \left( \frac{n}{K} \right) - \log(x) + \frac{4x}{(1+x^2)^2} \log\left(\frac{n}{K}\right) .$$
For $x \in \left(\frac{K}{n} \log(K), x^n_1\right)$:
\begin{align*}
    \frac{\partial g_n(x)}{\partial x} &\leq 1 - \log\left(\frac{n}{K}\right) - \log\left(\frac{K}{n} \log(K)\right) +4x^n_1 \log\left(\frac{n}{K}\right)\\
    &= 1-\log(\log(K)) +4p.
\end{align*}
Therefore for large enough $K$ we have $\frac{\partial g_n(x)}{\partial x} \leq -\frac{\log(\log(K))}{2}$, and $g_n$ is strictly decreasing on $\left( \frac{K}{n} \log(K), x^n_1\right)$ for large enough $n,K$. Now let $x \in \left( x_2, 1 \right)$. Note that $\frac{4x}{(1+x^2)^2}$ reaches it maximum $\frac{3\sqrt 3}{4}>1$ at $x = x_2$ and takes values strictly higher than $1$ on the interval $(x_2, 1)$. We then have for $x \in(x_2,1)$:
\begin{align*}
     \frac{\partial g_n(x)}{\partial x} &\geq 1 - \log\left(\frac{n}{K}\right) - \log(1) +  \frac{4x}{(1+x^2)^2} \log\left(\frac{n}{K}\right)\\
     &\geq 1  - \log\left(\frac{n}{K}\right) + \log\left(\frac{n}{K}\right)\\
     &\geq 1,
\end{align*}
therefore $g_n$ is strictly increasing over $(x_2,1)$. We can then upper bound the summations  over $x\in \left(\frac{K}{n}\log(K), x^n_1\right) \cup (x_2, 1)$ by their associated integrals as follows (we use the notation  $x(l) = x = \frac{l}{K}$):
\begin{align*}
   \sum_{ l = \lceil \frac{K^2}{n} \log(K) \rceil}^{ \lfloor Kx_1 \rfloor} 3K^{m+1} \exp\left(W_{n,K,l} +\gamma_n^2 \frac{\binom{K}{2} \binom{l}{2}}{\binom{K}{2} + \binom{l}{2}} \right) &\leq \sum_{ l = \lceil \frac{K^2}{n}\log(K) \rceil}^{ \lfloor Kx_1\rfloor} 3K^{m+1} \exp(K g_n(x(l))) \\
   &\leq  \sum_{ l = \lceil \frac{K^2}{n}\log(K) \rceil  -1}^{\lfloor Kx_1 \rfloor} 3K^{m+1}  \int_{l-1}^{l} \exp(K g_n(x(u)))  du\\
    &\leq \int_{ \frac{K^2}{n} \log(K)-1 }^{K x_1} 3K^{m+1}  \exp( K g_n(x(u))) du\\
    &= \int_{\frac{K}{n} \log(K)-\frac{1}{K} }^{x_1}3K^{m+1}  \exp( K  g_n(x)) K dx\\
    &= \int_{ \frac{K}{n} \log(K) - \frac{1}{K} }^{x_1} 3K^{m+2}  \exp( K  g_n(x)) dx \\
    &= o(1),
\end{align*}
where the last line follows by applying the Monotone Convergence Theorem since the sequence of functions $x\to K^{m+2}\exp(K g_n(x))$ is monotonically decreasing and for any fixed $x$ we have $g_n(x)<-1$ for sufficiently large $n,K$.  One proves Similarly that:
$$  \sum_{l = \lceil K x_2 \rceil }^{ K-1} 3K^{m+1} \exp\left(W_{n,K,l} +\gamma_n^2 \frac{\binom{K}{2} \binom{l}{2}}{\binom{K}{2} + \binom{l}{2}} \right)  \leq \sum_{l=\lceil K x_2 \rceil   }^{K-1} 3K^{m+1}  \exp(K g_n(x(l))) = o(1) .$$
Hence, in order to establish that $ B_2 = o(1) $ it suffices to prove that:
$$   \sum_{ l = \lceil K x^n_1 \rceil  }^{\lfloor Kx_2\rfloor}3K^{m+1}  \exp\left( W_{n,K,l} +\gamma_n^2 \frac{\binom{K}{2} \binom{l}{2}}{\binom{K}{2} + \binom{l}{2}}\right) =o(1) $$
Using (\ref{d(x)}), we have
\begin{align*}
    W_{n,K,l} +\gamma_n^2 \frac{\binom{K}{2} \binom{l}{2}}{\binom{K}{2} + \binom{l}{2}}  &\leq l\left( - \log(x)  - \frac{(x-1)^2}{1+x^2} \log\left(\frac{n}{K}\right)  + \frac{2x\xi_n}{1+x^2} \log\left(\frac{n}{K}\right)+2\right)\\
    &= K d_n(x)
\end{align*}
where $d(x) \triangleq - x\log(x)  - \frac{x(x-1)^2}{1+x^2} \log\left(\frac{n}{K}\right) +\frac{2x^2\xi_n}{1+x^2} \log\left(\frac{n}{K}\right)+2x $. Since $\max_{ x \in (0,1)} -x\log(x) = e^{-1}$, we have for $x\in [x^n_1, x_2]$:
\begin{align*}
    d_n(x) 
    &\leq e^{-1} - \min_{x \in (x^n_1,x_2)}  \left(\frac{x(x-1)^2 -2x^2\xi_n}{1+x^2}\right) \log\left(\frac{n}{K}\right) + 2.
\end{align*}
Note that for large enough $n$, we have $\min_{x \in (x^n_1,x_2)} \left[ (1-x)^2-2x\xi_n  \right]>0$. Therefore
\begin{align*}
   \min_{x \in (x^n_1,x_2)}  \left(\frac{x(x-1)^2 -2x^2\xi_n}{1+x^2}\right) \log\left(\frac{n}{K}\right)&= \min_{x \in (x^n_1,x_2)}  \left(\frac{(x-1)^2 -2x\xi_n}{1+x^2} \frac{x}{1+x^2}\right) \log\left(\frac{n}{K}\right)\\
   & \geq  \min_{x \in (x^n_1,x_2)} \left[ (1-x)^2-2x\xi_n  \right] \frac{x_1^n}{2} \log\left(\frac{n}{K}\right)\\
   &=  \min_{x \in (x^n_1,x_2)} \left[ (1-x)^2-2x\xi_n  \right] \frac{p}{2\log\left(\frac{n}{K}\right)} \log\left(\frac{n}{K}\right)\\
   &=  \frac{p}{2}\min_{x \in (x^n_1,x_2)} \left[ (1-x)^2-2x\xi_n  \right] ,
\end{align*}
For $n$ large enough the function $x \mapsto (1-x)^2 -2x \xi_n$ is strictly decreasing on $(0, x_2)$, therefore
\begin{align*}
    \min_{x \in (x^n_1,x_2)} \left[ (1-x)^2-2x\xi_n  \right]  &= \left(1-\frac{1}{\sqrt{3}}\right) ^2-\frac{2\xi_n}{\sqrt{3}},
\end{align*}
hence, for large enough $n$ we have 
\begin{align*}
    \min_{x \in (x^n_1,x_2)} \left[ (1-x)^2-2x\xi_n  \right]  &\geq \frac{1}{2}\left(1-\frac{1}{\sqrt{3}}\right)^2 ,
\end{align*}
and thus
\begin{align*}
    d_n(x)&\leq e^{-1} - \min_{x \in (x_1,x_2)}  \left(\frac{x(x-1)^2 -2x^2\xi_n}{1+x^2}\right) \log\left(\frac{n}{K}\right) + 2\\
    &\leq  e^{-1}  - \frac{p}{4}\left(1-\frac{1}{\sqrt{3}}\right)^2 +2
\end{align*}
We can pick $p$ large enough so that $d_n(x) \leq -1$ on $(x^n_1, x_2)$. Therefore, for big enough $K,n$ it holds that:
 $$ \exp\left(W_{n,K,l} + \gamma_n^2 \frac{\binom{K}{2} \binom{l}{2}}{\binom{K}{2}+\binom{l}{2}}  \right) \leq \exp\left(- K\right).$$
Then, the summation in $B_2$ over $x\in (x^n_1, x_2)$ is exponentially decreasing to $0$:
 \begin{align*}
     \sum_{ l =\lceil K x^n_1 \rceil }^{\lfloor K x_2 \rfloor} 3K^{m+1} \exp\left(W_{n,K,l} +\gamma_n^2 \frac{\binom{K}{2} \binom{l}{2}}{\binom{K}{2} + \binom{l}{2}} \right) &\leq \sum_{ l = \lceil K x^n_1 \rceil }^{\lfloor K x_2 \rfloor} 3 K^{m+1} \exp(K d_n(x)) \\
     &\leq \sum_{ l = \lceil K x^n_1 \rceil }^{\lfloor K x_2 \rfloor} 3K^{m+1}  \exp\left(-K\right)\\
     &\leq (Kx_2 - K x^n_1+1) 3K^{m+1}  \exp\left(- K \right)\\
     &\leq 3K^{m+2}  \exp\left(- K \right)\\
     &= o(1).
 \end{align*}
We thus have proven that $B_2 = o(1)$, which then yields  $\sum_{l = \lceil \frac{K^2}{n} \log(K) \rceil}^{K-1} S_{n,K,m,l} \leq K^m$, concluding the proof of the 2nd statement of Proposition \ref{Proposition 6.1}.\\

\subsection{Case : \texorpdfstring{$\alpha\in \left [\frac{2}{3},1\right)$}{TEXT}}$ $\newline
We divide the summation around $M \frac{K^2}{n}$ where $M$ is a positive constant that we will pick later. Let $B_1$ be the sum of $S_{n,K,m,l}$ for $2 \leq l \leq \lfloor M \frac{K^2}{n} \rfloor$ and $B_2$ the sum for $\lceil M \frac{K^2}{n} \rceil \leq l \leq K-1 $.
\subsection*{Analysis of $B_2$}
Using Lemmas \ref{Lemma 4.5}, \ref{Lemma 4.4} we have as per previous computations:
$$B_2  \lesssim  \sum_{  l = \lceil \frac{K^2}{n} M  \rceil }^{K-1} 3K^{m+1}\exp\left(W_{n,K,l} +\gamma_n^2 \frac{\binom{K}{2} \binom{l}{2}}{\binom{K}{2} + \binom{l}{2}} \right), $$
where  $W_{n,K,l} \triangleq  l\left(- \log(\eta) +1 + \frac{K}{n}\right) $. We show similarly to the analysis of the case $\alpha \in \left[ \frac{1}{2}, \frac{2}{3}\right)$ that $B_2=o(1)$. Indeed, we have for $l \geq \lceil M \frac{K^2}{n} \rceil$ and sufficiently large $n$:
$$ W_{n,K,l}+ \gamma_n^2 \frac{\binom{K}{2}\binom{l}{2}}{\binom{K}{2}+ \binom{l}{2}} \leq K g_n(x) ,$$
where $ x \triangleq \frac{\eta}{K} = \frac{l}{K} \in \left( M \frac{K}{n}, 1\right) $ and $g_n(x)$ is given by (\ref{g_n}). We claim that we can pick $p, M$ constants so that $g_n$ is decreasing on $\left(M \frac{K}{n}, \frac{p}{\log(\frac{n}{K})}\right)$ and increasing on $\left(\frac{1}{\sqrt 3}, 1\right)$. Note that we only need to check the 1st part as the 2nd has been established previously in the analysis of the case $\alpha \in \left[\frac{1}{2} , \frac{2}{3}\right)$. We have for $x \in \left(M \frac{K}{n}, \frac{p}{\log(\frac{n}{K}} \right)$:
\begin{align*}
    \frac{\partial g_n(x)}{\partial x} &\leq 1 - \log\left(\frac{n}{K}\right) - \log\left(M\frac{K}{n} \right) + \frac{4x}{(1+x^2)^2} \log\left(\frac{n}{K}\right)\\
    &\leq 1-\log(M) +\frac{4p}{(1+x^2)^2} \\
    &\leq 1 - \log(M) +4p.
\end{align*}
Hence if we pick $M > \exp(4p +1)$ we see that $g_n$ is strictly decreasing on $\left(M \frac{K}{n}, \frac{p}{\log(\frac{n}{K})}\right)$, this allows us to bound the summations in $B_2$ for $x \in \left(M \frac{K}{n}, \frac{p}{\log(\frac{n}{K})}\right) \cup \left(\frac{1}{\sqrt 3}, 1\right)$ with their associated integrals and use the Monotone Convergence Theorem to conclude that:
$$ \sum_{x \in  \left( M \frac{K}{n} ,  \frac{p}{\log(\frac{n}{K})}  \right) \cup \left( \frac{1}{\sqrt 3} , 1\right)} 3K^{m+1}  \exp\left( Kg_n(x)\right) = o(1).$$
The case  $x \in \left[\frac{p}{\log(\frac{n}{K})} , \frac{1}{\sqrt 3}\right)$ is dealt with similarly as in the argument given for $ \alpha\in \left[ \frac{1}{2}, \frac{2}{3}\right)$, by picking $p$ large enough so  that $ e^{-1}  - \frac{p}{4}\left(1-\frac{1}{\sqrt{3}}\right)^2 +2 \leq -1$.  Which readily yields  $ B_2 =o(1) $. 
\subsection*{Analysis of $B_1$}
It remains to consider the case $l \leq \lfloor M\frac{K^2}{n} \rfloor$. We have for large enough $n$ : $\gamma_n \leq 3 \sqrt{\frac{\log(\frac{n}{K})}{K}} \leq 3 \sqrt{\frac{\log(n)}{K}}$, therefore:
Hencefoth:
\begin{align*}
    \gamma_n^2  \frac{\binom{K}{2} \binom{l}{2}}{\binom{K}{2} + \binom{l}{2}} &\leq 9 \frac{\log(n)}{K} \binom{l}{2}\\
    &\leq 9 \frac{\log(n)}{K} \frac{l^2}{2}\\
    &\leq \frac{9}{2} \frac{\log(n)}{K} \frac{K^4}{n^2}\\
    &\leq \frac{9M^2}{2}  \frac{K^3}{n^2} .
\end{align*}
Therefore
\begin{align*}
    \sum_{2 \leq l \leq K-1} S_{n,K,m,l} &= B_1 + B_2\\
    &=  B_1 +   o(1) \\
    &\leq \left( \sum_{2 \leq l \leq \lfloor M\frac{K^2}{n} \rfloor} \binom{K}{l} \binom{n-K}{K-l} \binom{n}{K}^{-1} 3K^{m+1}  \exp\left(9M^2 \frac{K^3}{n^2  }\log\left(n\right)\right) \right) +   o(1)\\
    &\leq 3K^{m+1} \exp\left(9M^2 \frac{K^3}{n^2  }\log\left(n\right)\right) \sum_{0\leq l \leq K} \binom{K}{l} \binom{n-K}{K-l} \binom{n}{K}^{-1} + o(1)\\
    &\leq 3K^{m+1} \exp\left(O\left(\frac{K^3}{n^2} \log\left(\frac{n}{K}\right)\right)\right) + o(1) \numberthis \label{22}\\
    &\leq \exp\left(O\left( \frac{K^3}{n^2} \log \left(\frac{n}{K} \right) \right)  \right),
\end{align*}
where we used Lemma \ref{Lemma 4.1} in (\ref{22}). We thus have for $\alpha \in \left[\frac{2}{3}, 1\right)$ that $\sum_{2\leq l \leq K-1} \leq \exp\left(O\left(\frac{K^3}{n^2} \log\left(n\right)\right)\right) .$
\end{proof}

%% file: Proof_of_Lower_Bound_in_Theorem_3.1.tex
We reuse here the notations and variables introduced in Section \ref{Proof of Upper Bounds in Theorems 3.1, 3.2, 3.4}. We consider $\gamma_n = (1+ \delta_n) 2\sqrt{\frac{\log(\frac{n}{K})}{K}}$ where $\delta_n =o(1)$ will be chosen later. Recall the notation $ U_{\gamma_n}  \triangleq \sum_{S \subset V(G), |S|=K} \mathbf{1}\left(Z_S \geq \gamma_n \binom{K}{2}\right) $. We  use the second moment method to lower bound $\Psi^{\Rc}_K(G)$ and we first start by computing the second moment :
\begin{align*}
    \E [U_{\gamma_n}^2] &= \sum_{0 \leq l \leq K} \sum_{\substack{S,T \subset V(G) \\|S|=|T|=K, |S\cap T|=l}} \p\left(Z_S, Z_T \geq \gamma_n \binom{K}{2}\right). 
\end{align*}
The dependence of each $\p\left(Z_S, Z_T \geq \gamma_n \binom{K}{2}\right) $ in terms of the subsets $S,T$ is uniquely determined by the overlap size $|S\cap T|$. When $l\in \{0,1,K\}$ the term $\p\left(Z_S, Z_T \geq \gamma_n \binom{K}{2}\right)$ is easy to evaluate, thus we partition the sum when $l=0,1,K$ and when $2 \leq l \leq K-1$. We write $\E [U_{\gamma_n}^2] = A(\gamma_n) + B(\gamma_n)$ with:
\begin{align*}
    A(\gamma_n) & \triangleq \binom{n}{K} \binom{n-K}{K} \p\left(Z_S \geq \gamma_n \binom{K}{2}\right)^2 + \binom{n}{1}\binom{n-1}{K-1}\binom{n-K}{K-1} \p \left(Z_S \geq \gamma_n \binom{K}{2}\right)^2 \\&+ \binom{n}{K} \p \left(Z_S \geq \gamma_n \binom{K}{2}\right),\label{A}\tag{A}\\
B(\gamma_n) &\triangleq  \sum_{2 \leq l \leq K-1} \binom{n}{l} \binom{n-l}{K-l} \binom{n-K}{K-l} \p\left(Z_S, Z_T \geq \gamma_n \binom{K}{2}\right) ,\label{B}\tag{B}
\end{align*}
where $|S \cap T| = l$. Straightforward computations yield
\begin{align*}
    \frac{ n \binom{n-1}{K-1} \binom{n-K}{K-1}  }{\binom{n}{K} \binom{n-K}{K}} = \frac{K^2}{n-2K+1} \sim \frac{K^2}{n} = o(1), \numberthis \label{23}
\end{align*}
since $K =o(\sqrt{n})$. We then have
 \begin{align*}
     A(\gamma_n) \sim  \binom{n}{K} \binom{n-K}{K} \p\left(Z_S \geq \gamma_n \binom{K}{2}\right)^2 + \binom{n}{K} \p \left(Z_S \geq \gamma_n \binom{K}{2}\right). \numberthis \label{eq:A:term}
 \end{align*}
 We now turn to bounding $B(\gamma_n)$. We have by Lemma \ref{Lemma 4.11}:
\begin{align*}
     \p\left(Z_S, Z_T \geq \gamma_n \binom{K}{2} \right) &\leq 3\binom{K}{2} \exp\left( - \frac{\gamma_n^2 \binom{K}{2}}{1 + \frac{\binom{l}{2}}{\binom{K}{2}}} \right)\\
     &\leq 3 K^2 \exp\left( - \frac{\gamma_n^2 \binom{K}{2}^2}{\binom{K}{2}+\binom{l}{2}}\right),
\end{align*}
Since $\gamma_n \sim 2\sqrt{\frac{\log(\frac{n}{K})}{K}}$ we see that $\gamma_n \sqrt{\binom{K}{2}} = O\left(\sqrt{K \log(n)}\right) = o\left(\sqrt{\binom{K}{2}}\right)$. We then have by Lemma \ref{Lemma 4.12}
\begin{align*}
 \p\left(Z_S \geq \gamma_n \binom{K}{2}\right) &= \p \left(Z_S \geq \left[\gamma_n \sqrt{\binom{K}{2}} \right]  \sqrt{\binom{K}{2}} \right)\\
 &\geq  \frac{1}{\sqrt{2 \binom{K}{2}}} \exp\left(- \frac{\gamma_n^2 \binom{K}{2}}{2} - \frac{\gamma_n^4 \binom{K}{2}^2}{12 \binom{K}{2}} + O\left(\frac{\gamma^5 \binom{K}{2}^{5\over 2}}{\binom{K}{2}^{3\over 2}}\right)\right) \\
 &\geq  \frac{1}{K} \exp\left( - \frac{\gamma_n^2 \binom{K}{2}}{2} - \frac{\gamma_n^4 \binom{K}{2}}{12 }  +o(1) \right), \numberthis \label{25}
 \end{align*}
 where we used $\frac{\gamma^5 \binom{K}{2}^{5\over 2}}{\binom{K}{2}^{3\over 2}} =\Theta\left( \frac{ K^{-5\over 2}\log(n/K)^{5\over 2} K^5}{K^3}\right)= \Theta(\frac{\log(n/K)^{5\over 2}}{K^{1\over 2 }})=o(1)$. Therefore:
 \begin{align*}
 \frac{ \p\left(Z_S, Z_T \geq \gamma_n \binom{K}{2} \right)} { \p\left(Z_S \geq \gamma_n \binom{K}{2}\right)^2} &\leq 3K^4 \exp\left(   - \frac{\gamma_n^2 \binom{K}{2}^2}{\binom{K}{2} +\binom{l}{2}} +\gamma_n^2 \binom{K}{2} +\frac{\gamma_n^4 \binom{K}{2}}{6 } + o(1)\right) \\
 &\leq 3K^4 \exp\left(    \frac{\gamma_n^2 \binom{K}{2}\binom{l}{2}}{\binom{K}{2} +\binom{l}{2}} +\frac{\gamma_n^4 \binom{K}{2}}{6 }  + o(1) \right),
  \end{align*}  
We then have:
\begin{align*}
\bar B(\gamma_n) & \triangleq \frac{B(\gamma_n)}{\binom{n}{K}^2\p\left(Z_S \geq \gamma_n \binom{K}{2}\right)^2} \numberthis \label{Bbar}\\
&\leq  \sum_{2 \leq l \leq K-1} \binom{n}{l} \binom{n-l}{K-l} \binom{n-K}{K-l} \binom{n}{K}^{-2} 3K^4 \exp\left(    \frac{\gamma_n^2 \binom{K}{2}\binom{l}{2}}{\binom{K}{2} +\binom{l}{2}} +\frac{\gamma_n^4 \binom{K}{2}}{6 }  + o(1)  \right) \\
&\lesssim  \sum_{2 \leq l \leq K-1} \binom{n}{l} \binom{n-l}{K-l} \binom{n-K}{K-l} \binom{n}{K}^{-2} 3K^4 \exp\left(    \frac{\gamma_n^2 \binom{K}{2}\binom{l}{2}}{\binom{K}{2} +\binom{l}{2}}  +\frac{\gamma_n^4 \binom{K}{2}}{6 } \right) \numberthis \label{oterm}\\
&= \sum_{2 \leq l \leq K-1} \binom{K}{l} \binom{n-K}{K-l}  \binom{n}{K}^{-1} 3K^4 \exp\left(    \frac{\gamma_n^2 \binom{K}{2}\binom{l}{2}}{\binom{K}{2} +\binom{l}{2}}  +\frac{\gamma_n^4 \binom{K}{2}}{6 } \right)  \numberthis\label{26}\\
&=\sum_{2 \leq l \leq K-1} \binom{K}{l} \binom{n-K}{K-l} \binom{n}{K}^{-1} 3K^4 \exp\left(    \frac{\gamma_n^2 \binom{K}{2}\binom{l}{2}}{\binom{K}{2} +\binom{l}{2}} +\Theta(\log(n)^2) \right), \numberthis \label{28}
\end{align*}
where we used Lemma \ref{Lemma 4.6} in (\ref{26}), and the fact that the $o(1)$ term in the exponent is uniform in $l$ in (\ref{oterm}). It follows
\begin{align*}
\frac{\E[U_{\gamma_n}]^2}{\E[U_{\gamma_n} ^2]} &= \frac{ \binom{n}{K}^{2} \p\left( Z_S \geq \gamma_n \binom{K}{2}\right)^2  }{A(\gamma_n)+B(\gamma_n)} \\
&\gtrsim   \frac{ \binom{n}{K}^{2} \p\left( Z_S \geq \gamma_n \binom{K}{2}\right)^2  }{\binom{n}{K} \binom{n-K}{K} \p\left(Z_S \geq \gamma_n \binom{K}{2}\right)^2 + \binom{n}{K} \p \left(Z_S \geq \gamma_n \binom{K}{2}\right)+ B(\gamma_n)} \\
&=     \frac{ 1 }{ \frac{\binom{n-K}{K}}{\binom{n}{K}}  +  \frac{1}{ \binom{n}{K}\p\left( Z_S \geq \gamma_n \binom{K}{2}\right) } + \bar B(\gamma_n)}        \\
&=     \frac{ 1 }{1+o(1)  +  \frac{1}{ \binom{n}{K}\p\left( Z_S \geq \gamma_n \binom{K}{2}\right) } + \bar B(\gamma_n)},  \numberthis \label{27}
\end{align*}
where we used  part 5 of Lemma \ref{Lemma 4.5} in (\ref{27}).  

\begin{lemma}\label{Lemma 7.1}
Under the assumptions of Theorem \ref{Theorem 3.1} we have:
$$ \bar B(\gamma_n) = o(1).$$

\end{lemma}
We shall for now skip the proof of the above Lemma and show how it leads to asymptotic lower bounds on $\Psi^{\Rc}_K(G)$. We have by Paley-Zygmund inequality and (\ref{27}) combined with Lemma \ref{Lemma 7.1}:
$$ \p(U_{\gamma_n} \geq 1) \gtrsim\frac{ 1 }{1+o(1)  +  \frac{1}{ \binom{n}{K}\p\left( Z_S \geq \gamma_n \binom{K}{2}\right) } } . $$
If $\binom{n}{K} \p\left(Z_S \geq \gamma_n \binom{K}{2}\right) = w_n$ for some positive sequence $w_n= \omega(1)$ then $U_{\gamma_n} \geq 1$ w.h.p as $n\to \infty$. Using inequality (\ref{25}) note that it suffices to have:
\begin{align*}
    \binom{n}{K}\frac{1}{K} \exp\left( - \frac{\gamma_n^2 \binom{K}{2}}{2} \left(1 +\frac{\gamma_n^2}{6}\right) \right) = w_n &\iff \log\left(\binom{n}{K} \frac{1}{K}\right) - \frac{\gamma_n^2}{2}\left(1 +\frac{\gamma_n^2}{6}\right) \binom{K}{2} = w_n\\
    &\iff \gamma_n \sqrt{1 +\frac{\gamma_n^2}{6}} =\sqrt{\frac{2}{\binom{K}{2}}} \sqrt{ \log\left(\binom{n}{K} \frac{1}{K}\right) - w_n} .
\end{align*}
If we pick $w_n$ s.t $w_n = o\left( \log\left( \binom{n}{K} \frac{1}{K} \right)     \right) = o(K \log(n)) $, then we have by Taylor expansion
\begin{align*}
  \gamma_n \sqrt{1 +\frac{\gamma_n^2}{6}}&= \sqrt{\frac{2}{\binom{K}{2}}} \sqrt{ \log\left(\binom{n}{K} \frac{1}{K}\right)} - O\left(  \sqrt{\frac{2}{\binom{K}{2}}} \frac{w_n}{\sqrt{\log(\binom{n}{K}\frac{1}{K}})} \right)  \\
  &= \sqrt{\frac{2}{\binom{K}{2}} \log\left(\binom{n}{K} \frac{1}{K}\right)} - O\left( \frac{w_n}{K^{\frac{3}{2}} \sqrt{\log(\frac{n}{K})}} \right).
\end{align*}
Where the last line follows from part 3 of  Lemma \ref{Lemma 4.5}. In particular, we see that $\gamma_n\sim 2 \sqrt{\frac{\log(n/K)}{K}} = o(1)$. Therefore, we have
\begin{align*}
    \gamma_n &= \frac{1}{ \sqrt{1 + \frac{\gamma_n^2}{6}}} \left[ \sqrt{\frac{2}{\binom{K}{2}} \log\left(\binom{n}{K} \frac{1}{K}\right)} - O\left( \frac{w_n}{K^{\frac{3}{2}} \sqrt{\log(\frac{n}{K})}} \right)\right]\\
    &= \left(1 - \frac{\gamma_n^2}{12} + O(\gamma_n^4)\right )\left[ \sqrt{\frac{2}{\binom{K}{2}} \log\left(\binom{n}{K} \frac{1}{K}\right)} - O\left( \frac{w_n}{K^{\frac{3}{2}} \sqrt{\log(\frac{n}{K})}} \right)\right]\\
    &= \left(1 +O\left(\frac{\log(n)}{K}\right) \right )\left[ \sqrt{\frac{2}{\binom{K}{2}} \log\left(\binom{n}{K} \frac{1}{K}\right)} - O\left( \frac{w_n}{K^{\frac{3}{2}} \sqrt{\log(\frac{n}{K})}} \right)\right]\\
    &= \sqrt{\frac{2}{\binom{K}{2}} \log\left(\binom{n}{K} \frac{1}{K}\right)} - O\left( \frac{w_n}{K^{\frac{3}{2}} \sqrt{\log(\frac{n}{K})}} \right) + O\left( \frac{\log(n)}{K}\sqrt{\frac{2}{\binom{K}{2}} \log\left(\binom{n}{K} \frac{1}{K}\right)} + \frac{\log(n)}{K}\frac{w_n}{K^{\frac{3}{2}} \sqrt{\log(\frac{n}{K})}}   \right)\\
    &=\sqrt{\frac{2}{\binom{K}{2}} \log\left(\binom{n}{K} \frac{1}{K}\right)} - O\left( \frac{w_n}{K^{\frac{3}{2}} \sqrt{\log(\frac{n}{K})}} \right) + O\left( \sqrt{ \frac{\log(n)^3}{K^3}} + w_n\sqrt{\frac{\log(n)}{K^5}} \right)\\
    &=\sqrt{\frac{2}{\binom{K}{2}} \log\left(\binom{n}{K} \frac{1}{K}\right)} - O\left( \frac{w_n}{K^{\frac{3}{2}} \sqrt{\log(\frac{n}{K})}} \right) + O\left( \sqrt{ \frac{\log(n)^3}{K^3}}  \right).
\end{align*}

For the above choice of $\gamma_n$ we have with high probability as $n\to +\infty$:
\begin{align*}
    \Psi^{\Rc}_{K}(G) &\geq \binom{K}{2} \left(      \sqrt{\frac{2}{\binom{K}{2}}\log\left(\binom{n}{K} \frac{1}{K}\right)} - O\left( \frac{w_n}{K^{\frac{3}{2}} \sqrt{\log(\frac{n}{K})}}  +  \sqrt{ \frac{\log(n)^3}{K^3}} \right)       \right) \\
    &=  \sqrt{2\binom{K}{2}\log\left(\binom{n}{K} \frac{1}{K}\right)} - O\left( w_n \sqrt{\frac{K}{\log(n)}} + \sqrt{K\log(n)^{3}}\right)  ,
\end{align*}
since $w_n$ was chosen arbitrary with the conditions $w_n = o(K \log(n)), w_n = \omega(1)$, we see that $w_n$ can be taken so that 
$$w_n \sqrt{\frac{K}{\log(n)}} = o\left(  \sqrt{K\log(n)^{3}} \right), $$
 which concludes the proof of the lower bound of Theorem \ref{Theorem 3.1}. We now give a proof of Lemma \ref{Lemma 7.1}.

\begin{proof}[Proof of Lemma \ref{Lemma 7.1}]

Write each index $l$ as $\eta \frac{K^2}{n}$ where $ \frac{2n}{K^2} \leq \eta \leq \frac{n(K-1)}{K^2 } $ and divide the summation in $\bar B(\gamma_n)$ as $\bar B _1(\gamma_n) +\bar B _2(\gamma_n)$ where $\bar B _1(\gamma_n)$ is the sum over indices $2\leq l \leq \lfloor \log(K) \rfloor$ and $\bar B _2(\gamma_n)$ the sum over indices $\lceil \log(K) \rceil\leq l \leq K-1$. From part $1$ of Proposition \ref{Proposition 6.1} and (\ref{28}) we have $\bar B_2(\gamma_n) \lesssim 3 \sum_{\log(K) \leq l \leq K-1} S_{n,K,4, l} \exp(K)= o(1)$ . We now analyze the term $\bar B _1(\gamma_n)$.  We have 
using  Corollary \ref{Corollary 3.19}
$$ \exists M>0, \forall l\leq  \log(K), \frac{\p(Z_S,Z_T \geq \gamma_n \binom{K}{2})}{\p(Z_S \geq \gamma_n \binom{K}{2})^2} \leq M.$$
Therefore 
\begin{align*}
    \bar B_1(\gamma_n) &= \binom{n}{K}^{-1}\sum_{2 \leq l \leq \lfloor \log(K) \rfloor} \binom{K}{l} \binom{n-K}{K-l}\frac{\p(Z_S,Z_T \geq \gamma_n \binom{K}{2})}{\p(Z_S \geq \gamma_n \binom{K}{2})^2}\\
    &\leq M \binom{n}{K}^{-1}\sum_{2 \leq l \leq \lfloor \log(K) \rfloor} \binom{K}{l} \binom{n-K}{K-l}\\
    &\lesssim M \binom{n}{K}^{-1} \sum_{2 \leq l \leq \lfloor \log(K) \rfloor} K^l \sqrt{\frac{1}{2\pi (K-l)}} \left(\frac{(n-K)e}{K-l} \right)^{K-l} \numberthis \label{38}\\
    &\lesssim M \sqrt{2\pi K} \left( \frac{K}{ne}\right)^K \sum_{2 \leq l \leq \lfloor \log(K) \rfloor} K^l \sqrt{\frac{1}{2\pi (K-l)}} \left(\frac{(n-K)e}{K-l} \right)^{K-l} \numberthis \label{39}\\
    &\leq M\sqrt{2\pi}  \sum_{2 \leq l \leq \lfloor \log(K) \rfloor} \sqrt{\frac{K}{K-l}}   \left(\frac{K(K-l)}{(n-K)e} \right)^l \left(\frac{(n-K)K}{(K-l)n}  \right)^K,
\end{align*} 
where (\ref{38}) follows from noting $\binom{K}{l} \leq K^l$ and part 4 of Lemma \ref{Lemma 4.5}, and (\ref{39}) from part 1 of Lemma \ref{Lemma 4.5}. For large enough $n,K$, we can assume that $n-K \geq n/ 2$ and $K-l \geq K/2$ so that
\begin{align*}
\bar B_1(\gamma_n) \lesssim M\sqrt{2\pi}  \sum_{2 \leq l \leq \lfloor \log(K) \rfloor}  \left(\frac{2K^2}{ne} \right)^l \left( \frac{1 - \frac{K}{n}}{1 - \frac{l}{K}} \right)^K.
\end{align*}
Note that
\begin{align*}
    \left( \frac{1 - \frac{K}{n}}{1 - \frac{l}{K}} \right)^K &= \exp\left(K \left[\log\left( 1- \frac{K}{n}\right) - \log\left( 1- \frac{l}{K} \right)\right] \right)\\
    &=\exp\left( K\left[ - \frac{K}{n} + O\left(\frac{K^2}{n^2} \right) - \left( - \frac{l}{K} + O\left(\frac{l^2}{K^2}\right)\right) \right] \right)\\
    &= \exp \left( l + O\left(\frac{\log(K)^2}{K}\right) \right),
\end{align*}
where the last line follows from noting $l^2/K = o(K^3/n^2)$. Therefore
\begin{align*}
    \bar B_1(\gamma_n)  &\lesssim M \sqrt{2\pi} \sum_{2 \leq l \leq \lfloor \log(K) \rfloor}  \left(\frac{2K^2}{ne} \right)^l \exp \left( l + O\left(\frac{\log(K^2)}{K}\right) \right)\\
    &= M \sqrt{2\pi} \sum_{2 \leq l \leq \lfloor \log(K) \rfloor}   \left(\frac{2K^2}{n} \right)^l \exp \left( O\left(\frac{\log(K)^2}{K}\right) \right)\\
    &\sim M \sqrt{2\pi} \sum_{2 \leq l \leq \lfloor \log(K) \rfloor}   \left(\frac{2K^2}{n} \right)^l\\
    &\leq M\sqrt{2\pi} \log(K)  \left(\frac{2K^2}{n} \right)^2\\
    &= o(1),
\end{align*}
which readily implies that $\bar B_1(\gamma_n) = o(1)$. This implies in combination with $\bar B_2(\gamma_n) =o(1)$ that  $\bar B(\gamma_n) = o(1)$ completing the proof of lemma \ref{Lemma 7.1}.
\end{proof}

%% file: Proof_of_Lower_Bound_in_Theorem_3.2.tex
We reuse the same notations as in Section \ref{Proof of Lower Bound in Theorem 3.1}, in particular we consider $\gamma_n = (1+ \delta_n) 2\sqrt{\frac{\log(\frac{n}{K})}{K}}$ where $\delta_n =o(1)$ will be chosen later. Since $\alpha \in [\frac{1}{2},  \frac{2}{3})$ we have by (\ref{23}):
$$ \frac{ n \binom{n-1}{K-1} \binom{n-K}{K-1}  }{\binom{n}{K} \binom{n-K}{K}} \sim \frac{K^2}{n}.$$
Therefore
\begin{align*}
    A(\gamma_n) &\sim  \left( 1 + \frac{K^2}{n}\right)\binom{n}{K} \binom{n-K}{K} \p\left(Z_S \geq \gamma_n \binom{K}{2}\right)^2 + \binom{n}{K} \p \left(Z_S \geq \gamma_n \binom{K}{2}\right) \numberthis \label{30} .
 \end{align*}
Using (\ref{28})\footnote{Note that although we used (\ref{28}) in the context of section \ref{Proof of Lower Bound in Theorem 3.1}, we did not make use of the assumption $\alpha \in \left(0, \frac{1}{2} \right)$ when establishing (\ref{28}). }
\begin{align*}
\bar B(\gamma_n) &\lesssim \sum_{2 \leq l \leq K-1} \binom{K}{l} \binom{n-K}{K-l} \binom{n}{K}^{-1} 3K^4 \exp\left(    \frac{\gamma_n^2 \binom{K}{2}\binom{l}{2}}{\binom{K}{2} +\binom{l}{2}} +\Theta(\log(n)^2)  \right)\\
&= 3 \exp(\Theta(\log(n)^2) )\sum_{2 \leq l \leq K-1} S_{n,K,4,l} \\
&\lesssim 3 K^4\exp( \Theta(\log(n)^2) )\numberthis \label{31}\\
&= \exp(\Theta(\log(n)^2) ),
\end{align*}
where (\ref{31}) follows from part 2 of Proposition \ref{Proposition 6.1}. We then have
    \begin{align*}
\frac{\E[U_{\gamma_n}]^2}{\E[U_{\gamma_n} ^2]} &= \frac{ \binom{n}{K}^{2} \p( Z_S \geq \gamma_n \binom{K}{2})^2  }{A(\gamma_n) + B(\gamma_n)}\\
&\gtrsim \frac{ \binom{n}{K}^{2} \p( Z_S \geq \gamma_n \binom{K}{2})^2  }{ \left( 1+\frac{K^2}{n}\right)\binom{n}{K} \binom{n-K}{K} \p(Z_S \geq \gamma_n \binom{K}{2})^2 + \binom{n}{K} \p (Z_S \geq \gamma_n \binom{K}{2})+ B(\gamma_n)} \\
&=     \frac{ 1 }{ \left(1 + \frac{K^2}{n} \right) \frac{\binom{n-K}{K}}{\binom{n}{K}}  +  \frac{1}{ \binom{n}{K}\p( Z_S \geq \gamma_n \binom{K}{2}) } + \bar B(\gamma_n)}        \\
&\geq      \frac{ 1 }{O(K)  +  \frac{1}{ \binom{n}{K}\p( Z_S \geq \gamma_n \binom{K}{2}) } + \exp(\Theta(\log(n)^2)) } \numberthis \label{32} \\
&=      \frac{ 1 }{ \frac{1}{ \binom{n}{K}\p( Z_S \geq \gamma_n \binom{K}{2}) } + \exp(\Theta(\log(n)^2)) }  \numberthis \label{43}.
\end{align*}
Where (\ref{32}) follows from $K^2/n = O(K)$ and noting that $ \binom{n-K}{K} / \binom{n}{K} \leq 1  $, while (\ref{43}) follows from $K=O(\exp(\Theta(\log(n)^2)))$. If we assume that $\gamma_n $ is picked such that  $\binom{n}{K} \p\left(Z_S \geq \gamma_n \binom{K}{2}\right) = w_n $ where $w_n =\omega(1)$, then by Paley-Zygmund inequality and (\ref{43})
\begin{align*}
 \p\left(\Psi^{\Rc}_K(G) \geq \gamma_n \binom{K}{2}\right) &=\p(U_{\gamma_n} \geq 1)\\
 &\geq \frac{\E[U_{\gamma_n}]^2}{\E[U_{\gamma_n} ^2]}\\
 &\geq \frac{ 1 }{o(1)  + \exp(\Theta(\log(n)^2) )}\\
 &\geq  \exp(- C_1\log(n)^2),
\end{align*}
 for some positive constant $C_1>0$. On the other hand, applying Theorem \ref{Theorem 4.8} yields the following inequality for  $t\geq \beta K$ and universal constants $C_0, \beta>0$:
\begin{align*}
    \p\left( |\Psi^{\Rc}_K(G) - \E[\Psi^{\Rc}_K(G)]| \geq t\right) &\leq \exp\left(- C_0\frac{t^2}{K^2}\right).
\end{align*}
In particular, for $t^* \triangleq \sqrt{\frac{C_1}{C_0}} K\log(n) $ and large enough $K$ (so that $t^* \geq \beta K$) we have:
$$   \p\left( |\Psi^{\Rc}_K(G) - \E[\Psi^{\Rc}_K(G)]| \geq t^*\right) \leq \p\left(\Psi^{\Rc}_K(G) \geq \gamma_n \binom{K}{2}\right)   ,$$
which implies
$$  \gamma_n \binom{K}{2} \leq t^* + \E[\Psi^{\Rc}_K(G)] .$$
Therefore
\begin{align*}
    \p\left(\Psi^{\Rc}_K(G) \geq \gamma_n \binom{K}{2} - 2t^* \right) &\geq \p\left( |\Psi^{\Rc}_K(G) - \E[\Psi^{\Rc}_K(G)] | \leq t^*\right)\\
    &\geq 1 - \exp\left(-C_1 \log( n)^2\right)\\
    &= 1 - o(1).
\end{align*} 
Hence, $ \Psi^{\Rc}_K(G) \geq \gamma_n \binom{K}{2} -O(K\log(n))$ with high probability for any $\gamma_n $ satisfying the previously mentioned properties. In particular, we can take $\gamma_n$ as in Section \ref{Proof of Lower Bound in Theorem 3.1} which would yield the following lower bound inequality for any $w_n$ s.t $w_n =\omega(1)$ and $w_n = o(K\log(n))$

$$ \Psi^{\Rc}_K(G) \geq \sqrt{2\binom{K}{2}\log\left(\binom{n}{K} \frac{1}{K}\right)} - O\left( \sqrt{K\log(n)^3}\right) - O(K\log(n))$$
i.e
$$ \Psi^{\Rc}_K(G) \geq \sqrt{2\binom{K}{2}\log\left(\binom{n}{K} \frac{1}{K}\right)}  - O(K\log(n)) \text{ w.h.p as } n\to +\infty$$
which completes the proof of the lower bound in Theorem \ref{Theorem 3.2}.

%% file: Proof_of_Lower_Bound_in_Theorem_3.3.tex
We reuse the same notations as in Sections \ref{Proof of Lower Bound in Theorem 3.1}, \ref{Proof of Lower Bound in Theorem 3.2}. We recall that by Paley-Zygmund inequality, the following holds
\begin{align*}
\frac{\E[U_{\gamma_n}]^2}{\E[U_{\gamma_n} ^2]} &\geq \frac{\binom{n}{K}^2 \p \left( Z_S \geq \gamma_n \binom{K}{2} \right)^2}{A(\gamma_n) + B(\gamma_n)}\\
&\geq \frac{1}{ \frac{A(\gamma_n)}{\binom{n}{K}^2 \p \left( Z_S \geq \gamma_n \binom{K}{2} \right)^2} + \bar B(\gamma_n) }, \numberthis \label{33}
\end{align*}
where $A(\gamma_n), B(\gamma_n), \bar B(\gamma_n)$ are given by (\ref{A}), (\ref{B}) and (\ref{Bbar}). Combining (\ref{30}) and part 6 of Lemma \ref{Lemma 4.5} we have
\begin{align*}
    A(\gamma_n) &= o \left( \binom{n}{K}^2 \p \left( Z_S \geq \gamma_n \binom{K}{2} \right)^2 \right) + \binom{n}{K} \p \left( Z_S \geq \gamma_n \binom{K}{2} \right) \numberthis \label{46}.
\end{align*}
Using (\ref{28})
\begin{align*}
\bar B (\gamma_n)&\leq \sum_{2 \leq l \leq K-1} \binom{K}{l} \binom{n-K}{K-l} \binom{n}{K}^{-1} 3K^{4} \exp \left(\gamma_n^2 \frac{\binom{K}{2}\binom{l}{2}}{\binom{K}{2}+\binom{l}{2}} + \Theta(\log(n)^2) \right) \\
&= 3\exp\left(\Theta(\log(n)^2) \right)  \sum_{2 \leq l \leq K-1} S_{n,K,4,l}\\
&\leq  \exp\left(\Theta(\log(n)^2) \right)\exp \left( O\left( \frac{K^3}{n^2} \log(n)  \right) \right) \numberthis \label{47}\\
&= \exp \left( O\left( \frac{K^3}{n^2} \log(n)  \right) \right),
\end{align*}
where (\ref{47}) follows from part 4 of Proposition \ref{Proposition 6.1}. We then have by (\ref{33}) and (\ref{46})
\begin{align*}
\frac{\E[U_{\gamma_n}]^2}{\E[U_{\gamma_n} ^2]} &\gtrsim   \frac{ 1 }{o(1)  +  \frac{1}{ \binom{n}{K}\p\left( Z_S \geq \gamma_n \binom{K}{2}\right) } + \bar B(\gamma_n) } \\
&\geq      \frac{ 1 }{o(1)  +  \frac{1}{ \binom{n}{K}\p( Z_S \geq \gamma_n \binom{K}{2}) } + \exp(C_2 \frac{K^3}{n^2} \log(n))}   .
\end{align*}    
For some big enough constant $C_2$. Recall that $\gamma_n = (1+\delta_n) 2\sqrt{\frac{\log\left( \frac{n}{K} \right)}{K}}$. If we pick $\delta_n$ such that $\binom{n}{K} \p\left(Z_S \geq \gamma_n \binom{K}{2}\right) = w_n $ where $w_n$ is a positive real sequence satisfying $w_n = \omega(1)$, then  by Paley-Zygmund inequality
\begin{align*}
 \p(U_{\gamma_n} \geq 1) &\geq \frac{\E[U_{\gamma_n}]^2}{\E[U_{\gamma_n} ^2]}\\
 &\geq \frac{ 1 }{o(1)  + \exp(C_2 \frac{K^3}{n^2} \log(n))}\\
 &\geq \exp\left(-C_1 \frac{K^3}{n^2}\log(n)\right ) ,
\end{align*}
 for some positive constant $C_1>0$. We then have by Theorem \ref{Theorem 4.8}  for  $t\geq\beta K$ and universal constants $C_0, \beta >0$:
\begin{align*}
    \p( |\Psi^{\Rc}_K(G) - \E[\Psi^{\Rc}_K(G)]| \geq t) &\leq \exp\left(-C_0\frac{t^2}{ K^2}\right)
\end{align*}
In particular, for $t^* \triangleq \sqrt{\frac{C_1}{C_0} }\frac{K^{5/2}}{n} \sqrt{\log(n)}$ and large enough $K$, we have
$$   \p\left( |\Psi^{\Rc}_K(G) - \E[\Psi^{\Rc}_K(G)]| \geq t^*\right) \leq \p\left(\Psi^{\Rc}_K(G) \geq \gamma_n \binom{K}{2}\right),   $$
which implies
$$  \gamma_n \binom{K}{2} \leq t^* + \E[\Psi^{\Rc}_K(G)] .$$
Therefore
\begin{align*}
    \p\left(\Psi^{\Rc}_K(G) \geq \gamma_n \binom{K}{2} - 2t^* \right) &\geq \p\left( |\Psi^{\Rc}_K(G) - \E[\Psi^{\Rc}_K(G)] | \leq t^*\right)\\
    &\geq 1 - \exp\left(-C_1 \frac{K^3}{n^2}\log(n)\right)\\
    &= 1 - o(1).
\end{align*} 
Hence, $ \Psi^{\Rc}_K(G) \geq \gamma_n \binom{K}{2} - O\left(\frac{K^{5/2}}{n} \sqrt{\log(n)}\right)$ for any choice of $\gamma_n$ satisfying the previously stated assumptions. In particular, we can take $\gamma_n$ as in Sections \ref{Proof of Lower Bound in Theorem 3.1}, \ref{Proof of Lower Bound in Theorem 3.2} which would yield the following lower bound inequality for any $w_n$ s.t $w_n =\omega(1)$ and $w_n = o(K\log(n))$
$$ \Psi^{\Rc}_K(G) \geq \sqrt{2\binom{K}{2}\log\left(\binom{n}{K} \frac{1}{K}\right)} - O\left( \sqrt{K\log(n)^3}\right) -O\left( \frac{K^{\frac{5}{2}}}{n} \sqrt{\log(n)} \right)$$
i.e
$$ \Psi^{\Rc}_K(G) \geq \sqrt{2\binom{K}{2}\log\left(\binom{n}{K} \frac{1}{K}\right)}  - O\left( \frac{K^{\frac{5}{2}}}{n} \sqrt{\log(n)} \right) \text{ w.h.p as } n\to +\infty$$
which ends the proof of the lower bound in Theorem \ref{Theorem 3.4}.

%% file: Proof_of_Upper_Bounds_in_Theorems_3.5,3.6,3.7.tex
We follow the same proof steps as in Section \ref{Proof of Upper Bounds in Theorems 3.1, 3.2, 3.4} to upper bound $\Psi^{\Nc}_K(G)$, and keep similar notations for $\gamma_n , U_{\gamma_n}$. Namely:
$$ U_{\gamma_n}  \triangleq \sum_{S \subset V(G), |S|=K,  Z_{i,j} \sim \mathcal{N}} \mathbf{1}\left(Z_S \geq \gamma_n \binom{K}{2}\right), $$
where $\gamma_n \triangleq (1 + \delta_n) 2\sqrt{\frac{\log \left( \frac{n}{K}\right)}{K}}$ and $\delta_n$ is some  sequence satisfying $\delta_n =o(1)$ that we will fix later. We have using standard Gaussian tail bounds
\begin{align*}
     \p\left(U_{\gamma_n} \geq 1\right) &\leq \E[U_{\gamma_n}] \\
    &= \binom{n}{K} \p\left( Z_S \geq \gamma_n \binom{K}{2}\right)\\
     &= \binom{n}{K} \p\left(\mathcal{N}(0,1) \geq \gamma_n \sqrt{\binom{K}{2}} \right)\\
     &\lesssim  \binom{n}{K} \frac{1}{ \sqrt{2\pi}\gamma_n \sqrt{\binom{K}{2}}} \exp\left( - \frac{\gamma_n^2 \binom{K}{2}}{2}\right)\\
     &\sim \frac{1}{2\sqrt{\pi}} \binom{n}{K} \frac{1}{\sqrt{K \log(\frac{n}{K})}} \exp\left( - \frac{\gamma_n^2 \binom{K}{2}}{2}\right).
\end{align*}
From here on we are in the same setting as in Section \ref{Proof of Upper Bounds in Theorems 3.1, 3.2, 3.4}, which then yields the same upper bound as in Theorem \ref{Theorem 3.1}.

%% file: Proof_of_Lower_Bound_in_Theorem_3.5.tex
Throughout this proof we are in the setting $\alpha \in \left(0,\frac{1}{2}\right) \cup \left(\frac{1}{2}, \frac{2}{3}\right)$ as stated in Theorem \ref{Theorem 3.6}. It will be important to note that we are then either in the regime $K = o(\sqrt{n})$ or $K=\omega(\sqrt{n})$. We will use the second moment method to lower bound $\Psi^{\Nc}_K(G)$ and keep the same notations as in Section \ref{Proof of Upper Bounds in Theorems 3.6, 3.7, 3.8} for $\gamma_n, U_{\gamma_n}$. Recall
\begin{align*}
    \E [U_{\gamma_n}^2] &= A(\gamma_n) + B(\gamma_n)
\end{align*}
where
\begin{align*}
    A(\gamma_n) &\triangleq \binom{n}{K} \binom{n-K}{K} \p\left(Z_S \geq \gamma_n \binom{K}{2}\right)^2 + \binom{n}{1}\binom{n-1}{K-1}\binom{n-K}{K-1} \p \left(Z_S \geq \gamma_n \binom{K}{2}\right)^2 \\
    &+\binom{n}{K} \p \left(Z_S \geq \gamma_n \binom{K}{2}\right),\numberthis\label{AG}\\
B(\gamma_n) &\triangleq  \sum_{2 \leq l \leq K-1} \binom{n}{l} \binom{n-l}{K-l} \binom{n-K}{K-l} \p\left(Z_S, Z_T \geq \gamma_n \binom{K}{2}\right), \numberthis\label{BG}
\end{align*}
with $|S \cap T| = l$. Recall from (\ref{23}) that:
$$ \frac{ n \binom{n-1}{K-1} \binom{n-K}{K-1}  }{\binom{n}{K} \binom{n-K}{K}}  \sim \frac{K^2}{n}$$
We can then distinguish between two asymptotic behaviors of $A(\gamma_n)$:
\begin{alignat*}{3}
    A(\gamma_n) &\sim  \frac{K^2}{n}&&\binom{n}{K} \binom{n-K}{K} \p\left(Z_S \geq \gamma_n \binom{K}{2}\right)^2 + \binom{n}{K} \p \left(Z_S \geq \gamma_n \binom{K}{2}\right)&&,  \text{ if } K = \omega(\sqrt{n})\\
    A(\gamma_n) &\sim  &&\binom{n}{K} \binom{n-K}{K} \p\left(Z_S \geq \gamma_n \binom{K}{2}\right)^2 + \binom{n}{K} \p \left(Z_S \geq \gamma_n \binom{K}{2}\right)&&, \text{ if } K = o(\sqrt{n})
\end{alignat*}
We have by corollary \ref{Corollary 4.15} and $2 \leq l \leq K-1$
\begin{align*}
\p \left( Z_S, Z_T \geq \gamma_n \binom{K}{2} \right) &\leq \frac{1}{ 2 \pi \gamma_n^2} \frac{\left[ \binom{K}{2} + \binom{l}{2} \right]^2}{\binom{K}{2}^2 \sqrt{\binom{K}{2}^2 - \binom{l}{2}^2}} \exp \left(- \gamma_n^2 \frac{\binom{K}{2}^2}{\binom{K}{2} + \binom{l}{2}}\right) .
\end{align*}
Therefore
$$ B(\gamma_n) \leq \sum_{2 \leq l \leq K-1} \binom{n}{l} \binom{n-l}{K-l} \binom{n-K}{K-l} \frac{1}{ 2 \pi \gamma_n^2} \frac{\left[ \binom{K}{2} + \binom{l}{2} \right]^2}{\binom{K}{2}^2 \sqrt{\binom{K}{2}^2 - \binom{l}{2}^2}} \exp \left(- \gamma_n^2 \frac{\binom{K}{2}^2}{\binom{K}{2} + \binom{l}{2}}\right) .  $$
We can then lower bound asymptotically the ratio $\frac{\E[U_{\gamma_n}]^2}{\E[U_{\gamma_n} ^2]}$  using Lemma \ref{Lemma 4.5} as follows: 
\begin{itemize}
    \item If $K = \omega(\sqrt{n})$ :
    \begin{align*}
\frac{\E[U_{\gamma_n}]^2}{\E[U_{\gamma_n} ^2]} &\gtrsim   \frac{ \binom{n}{K}^{2} \p\left( Z_S \geq \gamma_n \binom{K}{2}\right)^2  }{\frac{K^2}{n}\binom{n}{K} \binom{n-K}{K} \p\left(Z_S \geq \gamma_n \binom{K}{2}\right)^2 + \binom{n}{K} \p \left(Z_S \geq \gamma_n \binom{K}{2}\right)+ B(\gamma_n)} \\
&=     \frac{ 1 }{\frac{K^2}{n} \frac{\binom{n-K}{K}}{\binom{n}{K}}  +  \frac{1}{ \binom{n}{K}\p\left( Z_S \geq \gamma_n \binom{K}{2}\right) } + \bar B(\gamma_n)}        \\
&=     \frac{ 1 }{o(1)  +  \frac{1}{ \binom{n}{K}\p\left( Z_S \geq \gamma_n \binom{K}{2}\right) } + \bar B(\gamma_n)}  .  \numberthis  \label{34}
\end{align*}

    \item If $K = o(\sqrt{n})$ : 
    \begin{align*}
\frac{\E[U_{\gamma_n}]^2}{\E[U_{\gamma_n} ^2]} &\gtrsim   \frac{ \binom{n}{K}^{2} \p\left( Z_S \geq \gamma_n \binom{K}{2}\right)^2  }{\binom{n}{K} \binom{n-K}{K} \p\left(Z_S \geq \gamma_n \binom{K}{2}\right)^2 + \binom{n}{K} \p \left(Z_S \geq \gamma_n \binom{K}{2}\right)+ B(\gamma_n)} \\
&=     \frac{ 1 }{ \frac{\binom{n-K}{K}}{\binom{n}{K}}  +  \frac{1}{ \binom{n}{K}\p\left( Z_S \geq \gamma_n \binom{K}{2}\right) } + \bar B(\gamma_n)}        \\
&=     \frac{ 1 }{1+o(1)  +  \frac{1}{ \binom{n}{K}\p\left( Z_S \geq \gamma_n \binom{K}{2}\right) } + \bar B(\gamma_n)}.  \numberthis \label{35}
\end{align*}
\end{itemize}
Where
\begin{align*}
\bar B(\gamma_n) &\triangleq \frac{B(\gamma_n)}{\binom{n}{K}^{2} \p\left( Z_S \geq \gamma_n \binom{K}{2}\right)^2}\\
&\leq \sum_{2 \leq l \leq K-1} \binom{n}{l} \binom{n-l}{K-l} \binom{n-K}{K-l} \binom{n}{K}^{-2} \frac{1}{ 2 \pi \gamma_n^2} \frac{\left[ \binom{K}{2} + \binom{l}{2} \right]^2}{\binom{K}{2}^2 \sqrt{\binom{K}{2}^2 - \binom{l}{2}^2}} \frac{ \exp \left(- \gamma_n^2 \frac{\binom{K}{2}^2}{\binom{K}{2} + \binom{l}{2}}\right)}{\p\left(Z_S\geq \gamma \binom{K}{2}\right)^2}\\
&= \sum_{2 \leq l \leq K-1} \binom{K}{l} \binom{n-K}{K-l} \binom{n}{K}^{-1}\frac{1}{ 2 \pi \gamma_n^2} \frac{\left[ \binom{K}{2} + \binom{l}{2} \right]^2}{\binom{K}{2}^2 \sqrt{\binom{K}{2}^2 - \binom{l}{2}^2}} \frac{\exp \left(- \gamma_n^2 \frac{\binom{K}{2}^2}{\binom{K}{2} + \binom{l}{2}}\right)}{\p(Z_S\geq \gamma_n \binom{K}{2})^2},\numberthis \label{36}
\end{align*}
where we used Lemma \ref{Lemma 4.6} in (\ref{36}). We have using standard Gaussian tail bound
\begin{align*}
    \p\left(Z_S\geq \gamma_n \binom{K}{2}\right) &= \p\left(\mathcal{N}(0,1)\geq  \gamma_n \sqrt{\binom{K}{2}}\right) \\
&\geq  \frac{1}{\sqrt{2 \pi}} \frac{\gamma_n \sqrt{\binom{K}{2}}}{1 + \gamma_n^2 \binom{K}{2}} \exp\left(-\frac{\gamma_n^2 \binom{K}{2}}{2}\right)\\
&\sim \frac{1}{\sqrt{ \pi}} \frac{1}{\gamma_n K} \exp\left(-\frac{\gamma_n^2}{2} \binom{K}{2}\right).
\end{align*}
Henceforth
\begin{align*}
\bar B(\gamma_n) &\lesssim \sum_{2 \leq l \leq K-1} \binom{K}{l} \binom{n-K}{K-l} \binom{n}{K}^{-1} \frac{1}{ 2 \pi \gamma_n^2} \frac{\left[ \binom{K}{2} + \binom{l}{2} \right]^2}{\binom{K}{2}^2 \sqrt{\binom{K}{2}^2 - \binom{l}{2}^2}}  \pi \gamma^2  K^2               \exp \left(- \gamma_n^2 \frac{\binom{K}{2}^2}{\binom{K}{2} + \binom{l}{2}} + \gamma_n^2 \binom{K}{2}\right)  \\
&= \sum_{2 \leq l \leq K-1} \binom{K}{l} \binom{n-K}{K-l} \binom{n}{K}^{-1} \frac{K^2}{ 2 } \frac{\left[ \binom{K}{2} + \binom{l}{2} \right]^2}{\binom{K}{2}^2 \sqrt{\binom{K}{2}^2 - \binom{l}{2}^2}}            \exp \left( \gamma_n^2 \frac{\binom{K}{2} \binom{l}{2}}{\binom{K}{2} + \binom{l}{2}} \right) \numberthis \label{37}. 
\end{align*}
We claim the following result
\begin{lemma}\label{Lemma 12.1}
\begin{alignat*}{3}
    \bar B(\gamma_n) &= && o(1)  , &&\text{ If } \alpha \in \left( 0,\frac{1}{2} \right)\\
    \bar B(\gamma_n) &\lesssim && 1 , &&\text{ If }  \alpha \in \left( \frac{1}{2},\frac{2}{3} \right)
\end{alignat*}
\end{lemma}

\begin{proof}

Consider first the case $ \alpha \in \left( 0, \frac{1}{2} \right)$, we have for all $2 \leq l\leq K-1$:

\begin{align*}
     \frac{K^2}{ 2 } \frac{\left[ \binom{K}{2} + \binom{l}{2} \right]^2}{\binom{K}{2}^2 \sqrt{\binom{K}{2}^2 - \binom{l}{2}^2}}  &\leq \frac{K^2}{2} \frac{4 \binom{K}{2}^2}{\binom{K}{2}^2 \sqrt{\binom{K}{2}^2 - \binom{K-1}{2}^2}}.\\
     &\leq 2K^2.\numberthis \label{54}
\end{align*}
Using part 1 of Proposition \ref{Proposition 6.1} 
\begin{align*}
    \bar B(\gamma_n) &\leq 2 \sum_{2 \leq l \leq K-1} S_{n,K,2,l}\\
    &= o(1).
\end{align*}
Hence $\bar B(\gamma_n) =o(1)$. Consider now the case $\alpha \in \left( \frac{1}{2}, \frac{2}{3}\right)$. If $2 \leq l \leq \lfloor\frac{K^2}{n}\log(K)\rfloor$, then $l = o(K)$ and we have
\begin{align*}
     \frac{K^2}{ 2 } \frac{\left[ \binom{K}{2} + \binom{l}{2} \right]^2}{\binom{K}{2}^2 \sqrt{\binom{K}{2}^2 - \binom{l}{2}^2}} &\leq \frac{K^2}{ 2 } \frac{\left[ \binom{K}{2} + \binom{\lfloor\frac{K^2}{n}\log(K)\rfloor}{2} \right]^2}{\binom{K}{2}^2 \sqrt{\binom{K}{2}^2 - \binom{2}{2}^2}} \\
     &\sim \frac{K^2}{2} \frac{\binom{K}{2}^2}{\binom{K}{2}^2 \sqrt{\binom{K}{2}^2}}\\
     &\sim 1 \numberthis \label{simterm}.
\end{align*}
Therefore
\begin{align*}
    \sum_{l = 2 }^{\lfloor \frac{K^2}{n}\log(K)\rfloor} \binom{K}{l} \binom{n-K}{K-l} \binom{n}{K}^{-1}  \frac{K^2}{ 2 } \frac{\left[ \binom{K}{2} + \binom{l}{2} \right]^2}{\binom{K}{2}^2 \sqrt{\binom{K}{2}^2 - \binom{l}{2}^2}} \exp\left( \gamma_n^2 \frac{\binom{K}{2} \binom{l}{2}}{\binom{K}{2} + \binom{l}{2}} \right)&\lesssim \sum_{l = 2 }^{\lfloor\frac{K^2}{n}\log(K)\rfloor} S_{n,K,0,l}\\
    &\lesssim 1,
\end{align*}
where the last line follows from part 2 of Proposition \ref{Proposition 6.1}. If $ l \geq \lceil \frac{K^2}{n}\log(K) \rceil$, we have using (\ref{54})
\begin{align*}
    \sum_{\lceil \frac{K^2}{n}\log(K) \rceil  \leq l \leq K-1} \binom{K}{l} \binom{n-K}{K-l} \binom{n}{K}^{-1}  \frac{K^2}{ 2 } \frac{\left[ \binom{K}{2} + \binom{l}{2} \right]^2}{\binom{K}{2}^2 \sqrt{\binom{K}{2}^2 - \binom{l}{2}^2}} \exp\left( \gamma_n^2 \frac{\binom{K}{2} \binom{l}{2}}{\binom{K}{2} + \binom{l}{2}} \right)&\leq \sum_{\lceil \frac{K^2}{n}\log(K) \rceil  \leq l \leq K-1} S_{n,K,0,l}\\
    &=o(1),
\end{align*}
where the last line follows from part 3 of Proposition \ref{Proposition 6.1}. This readily implies that $\bar B(\gamma_n) \lesssim 1$, concluding the proof of Lemma \ref{Lemma 12.1}.
\end{proof}
We now show how Lemma \ref{Lemma 12.1} leads to asymptotic lower bounds for $\Psi^{\Nc}_K(G)$. We have by Paley-Zygmund inequality:
\begin{align*}
    \p\left(\Psi^{\Nc}_K(G) \geq \gamma_n \binom{K}{2}\right) &= \p\left( U_{\gamma_n} \geq 1\right)\\
    &\geq \frac{ \E[U_{\gamma_n}]^2}{\E[U_{\gamma_n}^2]}\\
    &\gtrsim\frac{1}{1 + o(1) + \frac{1}{\binom{K}{n} \p\left(Z_S \geq \gamma_n \binom{K}{2}\right)}}.
\end{align*}
Indeed, if $\alpha \in (\frac{1}{2}, \frac{2}{3})$, then the last line  follows from part 2 of Lemma \ref{Lemma 12.1} coupled with inequality (\ref{34}), and if $ \alpha\in (0, \frac{1}{2})$ from part 1 of Lemma \ref{Lemma 12.1} coupled with inequality (\ref{35}). If $\gamma_n $ also satisfies $ \binom{K}{n} \p(Z_S \geq \gamma_n \binom{K}{2})= w_n$ where $w_n$ is a positive sequence s.t $w_n =\omega(1)$, then 
$$ \frac{ \E[U_{\gamma_n}]^2}{\E[U_{\gamma_n}^2]}\gtrsim 1.$$
Note that since $\gamma_n \sqrt{\binom{K}{2}}  =\omega(1)$, we have by standard Gaussian tail bounds\footnote{Recall that for $x=\omega(1)$ it holds $\p(\mathcal{N}(0,1) \geq x )\geq \frac{x}{1+x^2} \frac{1}{\sqrt{2\pi}} \exp\left(-\frac{x^2}{2}\right) \sim \frac{1}{x\sqrt{2\pi}} \exp\left(-\frac{x^2}{2}\right)$} 
\begin{align*}
    \binom{n}{K} \p\left(Z_S \geq \gamma_n \binom{K}{2}\right) &= \binom{n}{K} \p\left(\mathcal{N}\left(0, \binom{K}{2}\right) \geq \gamma_n \binom{K}{2}\right)\\
    &= \binom{n}{K} \p\left(\mathcal{N}\left(0, 1\right) \geq \gamma_n \sqrt{\binom{K}{2}}\right)\\
    &\sim \binom{n}{K}\frac{1}{\sqrt{2\pi} \gamma_n \sqrt{\binom{K}{2}}} \exp\left(- \frac{\gamma_n^2}{2} \binom{K}{2}\right)\\
    &\sim \frac{1}{2} \binom{n}{K}\frac{1}{\sqrt{ K \log(\frac{n}{K})}}\exp\left(- \frac{\gamma_n^2}{2} \binom{K}{2}\right).
\end{align*}
 If assume that $w_n$ satisfies $ \log(w_n) = o\left(  \binom{n}{K} \frac{1}{\sqrt{K\log(\frac{n}{K})}}\right) = o(K \log(n)) $ where the last equality follows from part 3 of Lemma \ref{Lemma 4.5}, then
\begin{align*}
    \binom{n}{K} \p\left(Z_S \geq \gamma_n \binom{K}{2}\right) = w_n &\iff \binom{n}{K}\frac{1}{\sqrt{ K \log(\frac{n}{K})}}\exp\left(- \frac{\gamma_n^2}{2} \binom{K}{2}\right) = w_n\\
    &\iff \gamma_n = \sqrt{ \frac{2}{\binom{K}{2}}}\sqrt{       -\log(w_n)  + \log\left(   \binom{n}{K}\frac{1}{\sqrt{ K \log(\frac{n}{K})}}  \right)    } \\
    &\iff \gamma_n = \sqrt{ \frac{2}{\binom{K}{2}}} \left(    \sqrt{      \log\left(   \binom{n}{K}\frac{1}{\sqrt{ K \log(\frac{n}{K})}}  \right)    }  + O\left( \frac{-\log(w_n)}{ \sqrt{      \log\left(   \binom{n}{K}\frac{1}{\sqrt{ K \log(\frac{n}{K})}}  \right)    }}\right)  \right)\\
    &\iff \gamma_n =   \sqrt{ \frac{2}{\binom{K}{2}}     \log\left(   \binom{n}{K}\frac{1}{\sqrt{ K \log(\frac{n}{K})}}  \right)    }  + O\left(\frac{-\log(w_n)}{K^{\frac{3}{2}}\sqrt{\log(n)}}\right) .\\
\end{align*}
Note that the above choice of $\gamma_n$ satisfies $\gamma_n \sim 2\sqrt{\log(n/K)/K}$. There fore, for $\gamma_n$ as above we have  w.h.p as $n\to +\infty$
\begin{align*}
\Psi^{\Nc}_{K}(G)  &\geq  \binom{K}{2} \left( \sqrt{ \frac{2}{\binom{K}{2}}     \log\left(   \binom{n}{K}\frac{1}{\sqrt{ K \log(\frac{n}{K})}}  \right)    }  + O\left(\frac{-\log(w_n)}{K^{\frac{3}{2}}\sqrt{\log(n)}}\right) \right)\\
&= \sqrt{2 \binom{K}{2} \log\left(   \binom{K}{n}\frac{1}{\sqrt{ K \log(\frac{n}{K})}}  \right)  } + O\left(-\log(w_n) \sqrt{\frac{K}{\log(n)}}\right).
\end{align*}
Since $w_n$ was chosen arbitrary with the condition $-\log(w_n) = o(K \log(n))$, we see that $-\log(w_n)$ can be replaced with any negative real sequence $a_n$ satisfying $a_n =o(K\log(n))$, which concludes the proof of the lower bound in Theorem \ref{Theorem 3.6}.

%% file: Proof_of_Lower_Bound_in_Theorem_3.6.tex
We keep the same notations as in Section \ref{Proof of Lower Bound in Theorem 3.6}, in particular recall that for $K = w(\sqrt n)$ we have from (\ref{34})
\begin{align*}
\frac{\E[U_{\gamma_n}]^2}{\E[U_{\gamma_n} ^2]} &\gtrsim   \frac{ 1 }{o(1)  +  \frac{1}{ \binom{n}{K}\p\left( Z_S \geq \gamma_n \binom{K}{2}\right) } + \bar B(\gamma_n)} \numberthis \label{55}
\end{align*}
where
\begin{align*}
\bar B(\gamma_n) &\leq \sum_{2 \leq l \leq K-1} \binom{K}{l} \binom{n-K}{K-l} \binom{n}{K}^{-1}  \frac{K^2}{ 2 } \frac{\left[ \binom{K}{2} + \binom{l}{2} \right]^2}{\binom{K}{2}^2 \sqrt{\binom{K}{2}^2 - \binom{l}{2}^2}} \exp\left( \gamma_n^2 \frac{\binom{K}{2} \binom{l}{2}}{\binom{K}{2} + \binom{l}{2}} \right) \numberthis \label{56}\\
&\leq \sum_{2 \leq l \leq K-1} \binom{K}{l} \binom{n-K}{K-l} \binom{n}{K}^{-1} 2K^2 \exp\left( \gamma_n^2 \frac{\binom{K}{2} \binom{l}{2}}{\binom{K}{2} + \binom{l}{2}} \right)\\
&= 2 \sum_{2 \leq l \leq K-1} S_{n,K,2,l} \numberthis \label{57}\\
&= o\left( \exp \left( \frac{K^3}{n^2} \log(n) \right) \right),
\end{align*}
where (\ref{56}) follows from (\ref{54}), and (\ref{57}) follows froom
 part 4 of Proposition \ref{Proposition 6.1}. We then have by (\ref{38}) and for large enough constant $C_2$
\begin{align*}
\frac{\E[U_{\gamma_n}]^2}{\E[U_{\gamma_n} ^2]} &\gtrsim  \frac{ 1 }{o(1)  +  \frac{1}{ \binom{n}{K}\p( Z_S \geq \gamma \binom{K}{2}) } + \exp(C_2 \frac{K^3}{n^2} \log(n))}   .
\end{align*}
Recalling that $\gamma_n = (1+\delta_n) 2\sqrt{\frac{\log\left( \frac{n}{K} \right)}{K}}$, if we pick $\delta_n=o(1)$ such that $\binom{n}{K} \p\left(Z_S \geq \gamma_n \binom{K}{2}\right) = w_n $ where $w_n$ is a positive real sequence satisfying $w_n=\omega(1)$, we then have by Paley-Zygmund inequality
\begin{align*}
 \p(U_{\gamma_n} \geq 1) &\geq \frac{\E[U_{\gamma_n}]^2}{\E[U_{\gamma_n} ^2]}\\
 &\geq \frac{ 1 }{o(1)  + \exp\left(C_2 \frac{K^3}{n^2} \log(n)\right)}\\
 &\geq \exp\left(-C_1 \frac{K^3}{n^2}\log(n)\right ) ,
\end{align*}
 for some positive constant $C_1>0$. We then have by Borell-TIS inequality (Theorem \ref{Theorem 4.16}) for  $t>0$ :
\begin{align*}
    \p\left( \Psi^{\Nc}_K(G) - \E[\Psi^{\Nc}_K(G)] \geq t\right) &\leq \exp\left(-\frac{t^2}{2 \binom{K}{2}}\right)\\
    &= \exp\left(- \frac{t^2}{K(K-1)}\right)\\
    &\leq \exp\left( -C_0 \frac{t^2}{K^2}\right),
\end{align*}
where $C_0$ is a deterministic constant. 
In particular, for $t^* \triangleq \sqrt{\frac{C_1}{C_0}}\frac{K^{5/2}}{n} \sqrt{\log(n)}$ we have
$$   \p\left( \Psi^{\Nc}_K(G) - \E[\Psi^{\Nc}_K(G)] \geq t^*\right) \leq \p\left(\Psi^{\Nc}_G(K) \geq \gamma_n \binom{K}{2}\right),   $$
which implies 
$$  \gamma_n \binom{K}{2} \leq t^* + \E[\Psi^{\Nc}_K(G)].$$
Therefore
\begin{align*}
    \p\left(\Psi^{\Nc}_K(G) \geq \gamma_n \binom{K}{2} - 2t^* \right) &\geq \p\left( |\Psi^{\Nc}_K(G) - \E[\Psi^{\Nc}_K(G)] | \leq t^*\right)\\
    &\geq 1 - 2\exp\left(-C_1 \frac{K^3}{n^2}\log(n)\right)\\
    &= 1 - o(1).
\end{align*} 
Hence, $ \Psi^{\Nc}_K(G) \geq \gamma_n \binom{K}{2} - O\left(\frac{K^{5/2}}{n} \sqrt{\log(n)}\right)$ for any choice of $\gamma_n$ satisfying the previously stated assumptions. Moreover, we see in light of the computations in Section \ref{Proof of Lower Bound in Theorem 3.6} that we can pick
$$\gamma_n =   \sqrt{ \frac{2}{\binom{K}{2}}     \log\left(   \binom{n}{K}\frac{1}{\sqrt{ K \log(\frac{n}{K})}}  \right)    }  - \frac{a_n}{K^{\frac{3}{2}}\sqrt{\log(n)}} , $$
where $a_n$  is a positive real sequence s.t $a_n =o(K \log(n))$ and $a_n =\omega(1)$. Hence, with high probability as $n\ to +\infty $ and for any such sequence $a_n$ it holds
\begin{align*}
    \Psi^{\Nc}_K(G) &\geq \gamma_n  \binom{K}{2}\\
    &=   \sqrt{ 2 \binom{K}{2}     \log\left(   \binom{n}{K}\frac{1}{\sqrt{ K \log(\frac{n}{K})}}  \right)    }  -a_n \sqrt{\frac{K}{\log(n)}} - O\left( \frac{K^{\frac{5}{2}}}{n} \sqrt{\log(n)} \right)\\
    &=   \sqrt{ 2 \binom{K}{2}     \log\left(   \binom{n}{K}\frac{1}{\sqrt{ K \log(\frac{n}{K})}}  \right)    } - O\left( \frac{K^{\frac{5}{2}}}{n} \sqrt{\log(n)} \right),
\end{align*}
where the last follows from picking $a_n$ s.t $a_n \sqrt{\frac{K}{\log(n)}} = o\left( \frac{K^{\frac{5}{2}}}{n} \sqrt{\log(n)}  \right) $
This concludes the proof of the lower bound in Theorem \ref{Theorem 3.7}.

%% file: Proof_of_Lower_Bound_in_Theorem_3.7.tex
We keep the same notations as in Sections \ref{Proof of Lower Bound in Theorem 3.6}, \ref{Proof of Lower Bound in Theorem 3.7}. In particular recall  
\begin{align*}
\frac{\E[U_{\gamma_n}]^2}{\E[U_{\gamma_n} ^2]} &= \frac{\binom{n}{K}^2 \p\left( Z_S \geq \gamma_n \binom{K}{2} \right)^2}{A(\gamma_n) + B(\gamma_n)}\\
&\gtrsim   \frac{ \binom{n}{K}^{2} \p\left( Z_S \geq \gamma_n \binom{K}{2}\right)^2  }{\left(1+ \frac{K^2}{n}\right)\binom{n}{K} \binom{n-K}{K} \p\left(Z_S \geq \gamma_n \binom{K}{2}\right)^2 + \binom{n}{K} \p \left(Z_S \geq \gamma_n \binom{K}{2}\right)+ B(\gamma_n)} \numberthis \label{58}\\ 
&= \frac{1}{\left(1+\frac{K^2}{n} \right) \frac{\binom{n-K}{K}}{\binom{n}{K}} + \frac{1}{\binom{n}{K} \p(Z_S \geq \gamma_n \binom{K}{2})} + \bar B(\gamma_n)} \numberthis \label{59},
\end{align*}
where (\ref{58}) follows from using (\ref{23}) coupled with the definition of $A(\gamma_n)$ in (\ref{AG}).
From (\ref{ratio:binom}) it follows
\begin{align*}
    \left(1+\frac{K^2}{n} \right) \frac{\binom{n-K}{K}}{\binom{n}{K}} &= \Theta(1) \exp\left( - \frac{K^2}{n} + O\left(  \frac{K^3}{n^2} \right)\right)\\
    &= \Theta(1) \exp(-\Theta(1) + o(1))\\
    &= \Theta(1).
\end{align*}
Therefore
$$  \frac{\E[U_{\gamma_n}]^2}{\E[U_{\gamma_n} ^2]} \gtrsim \frac{1}{\Theta(1)+ \frac{1}{\binom{n}{K} \p(Z_S \geq \gamma_n \binom{K}{2})} + \bar B(\gamma_n)} $$
Next we show that $\bar B(\gamma_n) \leq 1 + o(1)$. We divide the summation in $\bar B(\gamma_n)$ around $\frac{K^2}{n} \log(K)$, and let $\bar B_1(\gamma_n)$ be the summation in $\bar B(\gamma_n)$ for $2 \leq l \leq \lfloor \frac{K^2}{n}\log(K) \rfloor$ and $\bar B_2(\gamma_n)$ for $\lceil \frac{K^2}{n}\log(K) \rceil \leq l \leq K-1$. Note that for $l\leq \lfloor \frac{K^2}{n} \log(K) \rfloor$ we have by (\ref{simterm}) (uniformly in $l$)  $\frac{K^2}{ 2 } \frac{\left[ \binom{K}{2} + \binom{l}{2} \right]^2}{\binom{K}{2}^2 \sqrt{\binom{K}{2}^2 - \binom{l}{2}^2}}   \lesssim 1$ and for $l \geq \lceil \frac{K^2}{n} \log(K) \rceil$ we have by (\ref{54})$\frac{K^2}{ 2 } \frac{\left[ \binom{K}{2} + \binom{l}{2} \right]^2}{\binom{K}{2}^2 \sqrt{\binom{K}{2}^2 - \binom{l}{2}^2}}  \leq 2K^2$. Using parts 2 and  3 of Proposition \ref{Proposition 6.1}, it holds
\begin{align*}
    \bar B(\gamma_n) &= \bar B_1(\gamma_n) + \bar B_2(\gamma_n)\\
    &\lesssim \sum_{2 \leq l \leq \lfloor \frac{K^2}{n} \log(K)\rfloor} S_{n,K,0,l} + \sum_{\lceil \frac{K^2}{n} \log(K)\rceil \leq l \leq K-1} 2S_{n,K,2,l}\\
    &\lesssim 1.
\end{align*}
Therefore
\begin{align*}
     \frac{\E[U_{\gamma_n}]^2}{\E[U_{\gamma_n} ^2]} &\gtrsim \frac{1}{\Theta(1) +1 + o(1) + \frac{1}{\binom{n}{K} \p(Z_S \geq \gamma_n \binom{K}{2})} }\\
     &= \frac{1}{\Theta(1) + \frac{1}{\binom{n}{K} \p(Z_S\geq \gamma_n \binom{K}{2})}}.
\end{align*}
Recall that $\gamma_n = (1+\delta_n) 2\sqrt{\frac{\log\left( \frac{n}{K} \right)}{K}}$. If we pick $\delta_n=o(1)$ such that 
 $\binom{n}{K} \p\left(Z_S \geq \gamma_n \binom{K}{2}\right) = w_n $ where $w_n$ is a positive real sequence satisfying $w_n =\omega(1)$ then by Paley-Zygmund inequality
\begin{align*}
 \p(U_{\gamma_n} \geq 1) &\geq \frac{\E[U_{\gamma_n}]^2}{\E[U_{\gamma_n} ^2]}\\
 &\geq \frac{ 1 }{\Theta(1) + o(1) }\\
 &=\Theta(1).
\end{align*}
Therefore, there exists a constant $ C_1 \in (0,1)$ such that for sufficiently large $n$
$$ \p(U_{\gamma_n} \geq 1) \geq C_1 .$$
 We then have by Borell-TIS (Theorem \ref{Theorem 4.16}) inequality for  $t>0$ 
\begin{align*}
    \p( \Psi^{\Nc}_K(G) - \E[\Psi^{\Nc}_K(G)] \geq t) &\leq \exp\left(-\frac{t^2}{2 \binom{K}{2}}\right)\\
    &= \exp\left(- \frac{t^2}{K(K-1)}\right)\\
    &\leq \exp\left( -C_0 \frac{t^2}{K^2}\right).
\end{align*}
In particular, for $t^* \triangleq \sqrt{ \frac{-\log(C_1)}{C_0}}K \sqrt{m_n}$ where $a_n$ is a positive sequence s.t $a_n =\omega(1), a_n >1$, we have
$$   \p\left( \Psi^{\Nc}_K(G) - \E[\Psi^{\Nc}_K(G)] \geq t^*\right) \leq \p\left(\Psi^{\Nc}_K(G) \geq \gamma_n \binom{K}{2}\right),  $$
which implies
$$  \gamma_n \binom{K}{2} \leq t^* + \E[\Psi^{\Nc}_K(G)]. $$
Therefore
\begin{align*}
    \p\left(\Psi^{\Nc}_K(G) \geq \gamma_n \binom{K}{2} - 2t^* \right) &\geq \p\left( |\Psi^{\Nc}_K(G) - \E[\Psi^{\Nc}_K(G)] | \leq t^*\right)\\
    &\geq 1 - 2\exp\left(- C_1 w_n\right)\\
    &= 1 - o(1).
\end{align*} 
Hence, $ \Psi^{\Nc}_K(G) \geq \gamma_n \binom{K}{2} - O(K a_n)$ for any choice of $a_n, \gamma_n$ satisfying the previously stated assumptions, moreover, we see in light of the computations in Sections \ref{Proof of Lower Bound in Theorem 3.6}, \ref{Proof of Lower Bound in Theorem 3.7} that we can pick
$$\gamma_n =   \sqrt{ \frac{2}{\binom{K}{2}}     \log\left(   \binom{n}{K}\frac{1}{\sqrt{ K \log(\frac{n}{K})}}  \right)    }  - \frac{m_n}{K^{\frac{3}{2}}\sqrt{\log(n)}} , $$
where $m_n$ is any positive real sequence s.t $m_n =o(K \log(n))$ and $m_n =\omega(1)$. Henceforth, we have with high probability as $n\to +\infty $ and for any such sequences $m_n, a_n$:

\begin{align*}
    \Psi^{\Nc}_K(G) &\geq \gamma_n  \binom{K}{2}\\
    &=   \sqrt{ 2 \binom{K}{2}     \log\left(   \binom{n}{K}\frac{1}{\sqrt{ K \log(\frac{n}{K})}}  \right)    }  -m_n \sqrt{\frac{K}{\log(n)}} -O(K a_n)\\
    &=   \sqrt{ 2 \binom{K}{2}     \log\left(   \binom{n}{K}\frac{1}{\sqrt{ K \log(\frac{n}{K})}}  \right)    } - K a_n,
\end{align*}
where the last follows from picking $m_n$ s.t $m_n \sqrt{\frac{K}{\log(n)}} = o( Ka_n ) $ and including the constant hidden by $O(K a_n)$ in $a_n$. This concludes the proof of Theorem \ref{Theorem 3.8}.

%% file: Proof_of_Corollary_3.4.tex
\begin{proof}[Proof of Corollary \ref{Corollary 3.5}]
In light of Theorems \ref{Theorem 3.1}, \ref{Theorem 3.2}, \ref{Theorem 3.4}, we see that for $K=n^{\alpha}, \alpha \in \left(0,1\right)$ we have:
$$ \Psi^{\Rc}_K(G) \gtrsim L(n,K) . $$
$$ \Psi^{\Rc}_K(G) \lesssim U(n,K) . $$
From (\ref{LNK}), we have $L(n,K), U(n,K) \sim K^{\frac{3}{2}} \sqrt{\log\left( \frac{n}{K} \right)}$, thus $ \Psi^{\Rc}_K(G) \sim K^{\frac{3}{2}} \sqrt{\log\left( \frac{n}{K} \right)} $. Therefore by definition of $\Psi^{\Bern\left(\frac{1}{2}\right)}_K(G)$
$$   \Psi^{ \Bern\left(\frac{1}{2} \right)}_K(G) = \frac{K^2}{4} +  \frac{K^{\frac{3}{2}} \sqrt{\log(\frac{n}{K})}}{2} + o\left(K^{\frac{3}{2}} \sqrt{\log(n)}\right) , \text{w.h.p as } n\to +\infty$$
\end{proof}

\begin{proof}[Proof of Corollary \ref{Corollary 3.9}]
In light of Theorems \ref{Theorem 3.6}, \ref{Theorem 3.7}, \ref{Theorem 3.8}, we see that for $K=n^{\alpha}, \alpha \in \left(0,1\right)$ we have
$$ \Psi^{\Nc}_K(G) \sim U(n,K) . $$
From (\ref{LNK}), we have $ U(n,K) \sim K^{\frac{3}{2}} \sqrt{\log\left( \frac{n}{K} \right)}$, thus $ \Psi^{\Nc}_K(G) \sim K^{\frac{3}{2}} \sqrt{\log\left( \frac{n}{K} \right)} $. Therefore by definition of $\Psi^{\Nc\left(\frac{1}{2}, \frac{1}{4}\right)}_K(G)$
$$   \Psi^{ \Nc\left(\frac{1}{2}, \frac{1}{4} \right)}_K(G) = \frac{K^2}{4} +  \frac{K^{\frac{3}{2}} \sqrt{\log(\frac{n}{K})}}{2} + o\left(K^{\frac{3}{2}} \sqrt{\log(n)}\right) , \text{w.h.p as } n\to +\infty$$
\end{proof}

%% file: Overlap_Gap_Property.tex
\subsection{Bernoulli Planted Clique Model. Preliminary Estimates.}
In this section we will always assume that $K=n^{\alpha}$ with $\alpha < \frac{2}{3}$ unless stated otherwise.
\begin{defn}{(First Moment curve)}
The first moment curve is the real-valued function $\Gamma_{K}$ defined on $ \{\lfloor{\frac{K^2}{n}\rfloor}, \lfloor{\frac{K^2}{n}\rfloor}+1,...,K \} \to \mathbb{R}$ and given by

\begin{alignat*}{3}
     \Gamma_{K}(z) &= \binom{K}{2}&&, &&\text{ If } z=K\\
     \Gamma_{K}(z) &= \binom{z}{2} + h^{-1} \left( \log(2) - \frac{\log(\binom{K}{z}) \binom{n-K}{ K-z}}{\binom{K}{2} - \binom{z}{2}}\right) \left( \binom{K}{2} - \binom{z}{2} \right)&&, &&\text{ If } z \in \left\{\lfloor{\frac{K^2}{n}\rfloor}, \lfloor{\frac{K^2}{n}\rfloor}+1,...,K \right\} 
\end{alignat*}

where $h^{-1}$ is the (rescaled) inverse of the binary entropy function $h : \left[0,  \frac{1}{2} \right] \to [0,1] $ defined by $h(x) = -x \log(x) - (1-x) \log(1-x)$.
\end{defn}

We first recall the following central proposition that can be obtained from part $1$ of Proposition $1$ in \cite{GamarnikZadik}. It establishes that the first moment curve upper bounds the density of subgraphs of size $K$ and overlap $z$. Evaluating the tightness of this inequality will be key in understanding the effect of overlap on optimal subgraphs.

\begin{proposition}\label{Proposition 15.2}
Let $K, n \in \mathbb{N} $ with $K \leq n$ we have with high probability as $n \to +\infty$ for every $z\in \{\lfloor{\frac{K^2}{n}\rfloor}, \lfloor{\frac{K^2}{n}\rfloor}+1,...,K \} $
$$ \Psi_{K}^{\Bern(\frac{1}{2})}(G)(z) \leq \Gamma_{K}(z). $$
\end{proposition}
For our purposes we present here some results on first moment curve approximation of the function $\Gamma_K(G)(z)$. We start by recalling the result of Lemma $7$ in the non overparametrization setting ($k=\bar k = K$) in \cite{GamarnikZadik}:
\begin{lemma}\label{Lemma 15.3}
For $K\leq n, \epsilon >0$ and a sufficiently big constant $C_0 = C_0(\epsilon)$, if $ C_0 \frac{K^2}{n} \leq z \leq (1-\epsilon) K$, then
$$ \Gamma_K(z+1)-\Gamma_K(z) = z \left(\frac{1}{2} + o(1)\right) - \Theta \left[  \sqrt{\frac{K}{\log(n)}} \log\left(\frac{(z+1)n}{K^2}\right)  \right]+O(1).$$
\end{lemma}

Note that we can assume that $C_0 \geq 1$ which implies that $\frac{(z+1)n}{K^2} \geq C_0 \frac{K^2}{n} \frac{n}{K^2} = C_0 \geq 1$ and therefore $ \sqrt{\frac{K}{\log(n)}} \log\left(\frac{(z+1)n}{K^2}\right) = \omega(1)$.

\begin{lemma}\label{Lemma 15.4}
Let $\epsilon >0$, we can find $D_1 < D_2$ positive reals (depending on $\epsilon, \alpha$) such that if we let $I \triangleq  \mathbb{Z} \cap [D_1 \sqrt{K \log(K)}, D_2 \sqrt{K \log(K)}]$, then there exists a constant $C_1 = C_1(\epsilon)$ s.t for sufficiently large $n,K$ it holds
$$ \max_{z \in I} \Gamma_K(z) + C_1 K\log(K) \leq \Gamma_K\left( \left \lfloor C_0 \frac{K^2}{n} \right\rfloor \right) \leq \Gamma_K((1-\epsilon)K) .$$
\end{lemma}
\begin{proof}
Let $z_1 = D \sqrt{K \log(K)} -1$ for some positive real $D$ and let $z_0 = \left \lfloor C_0 \frac{K^2}{n} \right\rfloor$. We then have by telescoping the result of Lemma \ref{Lemma 15.3}
\begin{align*}
    \Gamma_{K}(z_1+1) - \Gamma_K(z_0) &= \sum_{z= z_0}^{z_1}  z \left(\frac{1}{2} + o(1)\right) - \Theta \left[  \sqrt{\frac{K}{\log(n)}} \log\left(\frac{(z+1)n}{K^2}\right)  \right]+O(1)\\
    &\leq  \sum_{z= z_0}^{z_1}  z \left(\frac{1}{2} + o(1)\right) - C' \left[  \sqrt{\frac{K}{\log(n)}} \log\left(\frac{(z+1)n}{K^2}\right)  \right]+O(1),
\end{align*}
where $C'$ is a positive constant depending on $n,K$. We have:
\begin{align*}
    \sum_{z= z_0}^{z_1} z \left(\frac{1}{2} + o(1)\right) &\sim \int_{z_0}^{z_1} z \left(\frac{1}{2} + o(1)\right)dz \\
    &\sim \frac{1}{4} \left[z_1  ^2 - z_0^2 \right] \\
    &\sim \frac{z_1^2}{4} \numberthis \label{40} \\
    &= \frac{D^2}{4} K \log(K) \numberthis \label{62},
\end{align*}
where we used $\alpha < \frac{3}{2} \implies\frac{K^2}{n} = o\left(\sqrt{K \log(K)}\right) $ in (\ref{40}).
On the other hand, we have
\begin{align*}
    \sum_{z=z_0}^{z_1}  \log\left(\frac{(z+1)n}{K^2}\right) &\sim \int_{z_0}^{z_1} \log\left(\frac{(z+1)n}{K^2}\right)dz\\ 
    &= \int_{z_0+1}^{z_1+1} \log\left(\frac{nz}{K^2}\right)dz\\
    &= \left[  z\log\left(\frac{nz}{K^2}\right) - z\right]_{z_0+1}^{z_1+1}\\
    &\sim (z_1+1)\log\left(\frac{n(z_1+1)}{K^2}\right)\\
    &= D \sqrt{K \log(K)} \log\left( \frac{nD \sqrt{K\log(K)}}{K^2}\right)\\
    &\sim D \sqrt{K \log(K)} \log\left( \frac{n}{K^{\frac{3}{2}}}\right)\\
    &\sim D\left(\frac{1}{\alpha}  - \frac{3}{2}\right) \sqrt{K} \log(K)^{\frac{3}{2}}  \numberthis \label{63},            
\end{align*}
where the last line follows from $\log\left(\frac{n}{K^{\frac{3}{2}}}\right) = \log(n) -\frac{3}{2}\log(K) = (\frac{1}{\alpha} - \frac{3}{2}) \log(K)$. Note in particular that since we are interested in the regime $\alpha \in ( 0 , \frac{2}{3})$ we have $D(\frac{1}{\alpha}  - \frac{3}{2}) >0$. Moreover:
\begin{align*}
\sum_{z= z_0}^{z_1} O(1) &\sim O(z_1 - z_0)\\
&= O(\sqrt{K \log(K)}) \numberthis \label{64}.
\end{align*}
Combining (\ref{62}), (\ref{63}), and (\ref{64}) we have
\begin{align*}
\Gamma_{K}(z_1+1) - \Gamma_K(z_0)& \lesssim \frac{D^2}{4} K \log(K)  - C' D\left(\frac{1}{\alpha}  - \frac{3}{2}\right) \sqrt{\frac{K}{\log(n)}}  \sqrt{K} \log(K)^{\frac{3}{2}}  + O(\sqrt{K \log(K)}),
\end{align*}
equivalently
\begin{align*}
 \Gamma_K(D\sqrt{K\log(K)}) - \Gamma_K\left( \left \lfloor C_0\frac{K^2}{n} \right \rfloor\right)&\lesssim\left[ \frac{D^2}{4} - C' D \sqrt{\alpha}\left(\frac{1}{\alpha}  - \frac{3}{2}\right) \right] K \log(K)+ O(\sqrt{K \log(K)}).
\end{align*}
 Since $\sqrt{\alpha}\left(\frac{1}{\alpha}  - \frac{3}{2}\right) >0$, we can find some (small enough) $0< D_1 < D_2$ and $ D_0 >0$ such that $ D_0, D_1,D_2$ depend on $\alpha, \epsilon$ and satisfy 
 $$\forall D \in [D_1,D_2] , \left[ \frac{D^2}{4} - C' D \sqrt{\alpha}\left(\frac{1}{\alpha}  - \frac{3}{2}\right) \right] \leq -2D_0,$$
and 
 $$ \forall D \in [D_1,D_2] , \Gamma_K(D\sqrt{K\log(K)}) - \Gamma_K\left( \left\lfloor C_0\frac{K^2}{n} \right\rfloor\right) \leq -D_0 K\log(K),$$
 which implies
 $$ \max_{z \in\mathbb{Z} \cap [D_1 \sqrt{K\log(K)}, D_2 \sqrt{K\log(K)}]} \Gamma_K(z) + D_0 K\log(K) \leq \Gamma_K\left( \left\lfloor C_0 \frac{K^2}{n}\right\rfloor\right) \leq \Gamma_K((1-\epsilon)K).$$
 This concludes the proof of Lemma \ref{Lemma 15.4}.
\end{proof}
\subsection{Proof of Theorems \ref{Theorem 3.13}, \ref{Theorem 3.14}} \label{Subsection 15.2}
In order to establish the claim of Theorems \ref{Theorem 3.13}, \ref{Theorem 3.14} we first need the following  Lemma.
\begin{lemma}\label{Lemma 15.5}
Let $0 < m\leq K=n^{\alpha}\leq n$ be positive integers  with $\alpha\in \left(0, \frac{2}{3} \right)$ and $G \in \mathbb{G}(n,K,\Bern(1/2))$ be a random graph with a Planted Clique $\mathcal{PC}$ of size $K$. Define  $G_0 \triangleq G\setminus \mathcal{PC}$ to be the subgraph of $G$ obtained by removing $\mathcal{PC}$. Then we have the following with high probability as $n \to \infty$
$$ \Psi^{Bern(\frac{1}{2})}_K(G)(m) \geq \binom{m}{2} + \Psi^{Bern(\frac{1}{2})}_{K-m}(G_0) + \frac{(K-m)m}{2} - a(n) \sqrt{\frac{(K-m)m}{4}} ,$$
where $a(n)$ is any positive sequence s.t $a(n)=\omega(1)$.
\end{lemma}
\begin{proof}
 Fix an arbitrary $m$-vertices subgraph $S_1$ of $\mathcal{PC}$. Then $N\triangleq \binom{n-K}{K-m}$ is the number of different $K-m$ vertices subgraphs $S_2$ of $G_0$. Optimizing over the choice of $S_2$ yields
\begin{align*}
\Psi^{Bern(\frac{1}{2})}_K(G)(m) &\geq \max_{S_2} |E(S_1 \cup S_2)|\\
&= \binom{m}{2} + \max_{S_2} \{ |E(S_1,S_2)| + |E(S_2)| \}. \numberthis \label{41}
 \end{align*} 
 Where here $E(H)$ is the sum of all edge random variables over the subgraph $H$. In particular, note that $E(S_1)= \binom{m}{2}$, $E(S_2) \stackrel{d}{=} Bin\left(\binom{K-m}{2},\frac{1}{2} \right)$ and $E(H,L)$ is the sum of all edge random variables between $H,L$ for $H\cap L = \emptyset$. Next, let's index the subsets $S_2$ by $S^i, i=1,...,N$ and set
$$ Y_i \triangleq |E(S_1, S^i)|, X_i \triangleq |E(S^i)| .$$
Note that
\begin{enumerate}
    \item $\forall i\in [N], X_i \sim Bin\left(\binom{K-m}{2},\frac{1}{2} \right) .$
    \item $\forall i\in [N], Y_i \sim Bin\left((K-m)m,\frac{1}{2} \right) .$
    \item The sequence $(X_i)_i$ is independent from the sequence $(Y_j)_j$.
    \item $\max_{i \in [N]} X_i = \Psi^{Bern(\frac{1}{2})}_{K-m}(G_0).$
\end{enumerate}
We claim that for any sequence $a(n)=\omega(1)$ the following holds w.h.p as $n\to \infty$.
\begin{align*}
    \max_{i \in [N]} \{X_i + Y_i\} &\geq \max_{i\in [N]} X_i + \frac{(K-m)m}{2} - a(n)\sqrt{\frac{(K-m)m}{4}}. \numberthis \label{42}  
 \end{align*}
Indeed, let $i^* \in \arg \max_{i \in [N]} X_i $. It suffices to prove that w.h.p
$$ Y_{i^*} \geq  \frac{(K-m)m}{2} - a(n)\sqrt{\frac{(K-m)m}{4}}.$$
Since the sequences $X,Y$ are mutually independent, we see that $i^*$ is independent of the sequence $Y$, thus $Y_{i^*} \stackrel{d}{=} Bin((K-m)m, \frac{1}{2})$. The claim then follows since $a(n) =\omega(1)$. Combining (\ref{41}) with (\ref{42}) yields for any $a(n) =\omega(1)$
\begin{align*}
\Psi^{\Bern(\frac{1}{2})}_K(G)(m) &\geq \binom{m}{2} + \Psi^{\Bern(\frac{1}{2})}_{K-m}(G_0) + \frac{(K-m)m}{2} - a(n)\sqrt{\frac{(K-m)m}{4}},
\end{align*}
which concludes the proof.
\end{proof}
The proof of Theorems \ref{Theorem 3.13}, \ref{Theorem 3.14} is mainly based on the following assumption.
$$  \Psi^{\Bern(\frac{1}{2})}_K(G) \geq V(n,K) - o(K\log(n)) , \text{ w.h.p as } n \to \infty ,$$
This assumption holds true in the case $K = n^{\alpha}, \alpha \in (0, \frac{1}{2})$ by Theorem \ref{Theorem 3.1}. In the regime $\alpha \in [\frac{1}{2}, \frac{2}{3})$, Theorem \ref{Theorem 3.2} proves a slightly weaker bound, which is why Theorem \ref{Theorem 3.13} is based on the underlying assumption that Conjecture \ref{conjecture 3.3} holds.

\begin{proof}[proof of Theorem \ref{Theorem 3.13}]
We will assume in the remaining of this section that $K\in 2\mathbb{Z}_{\geq 1}$. The result of \ref{Theorem 3.13} when $K\in 2\mathbb{Z}_{\geq 0} +1$ can be readily deduced from the former case by noting  $|\Psi^{\Bern(\frac{1}{2})}_b(G) - \Psi^{\Bern(\frac{1}{2})}_{b+1}(G)|\leq b$.

We first establish the presence of OGP for the Bernoulli Planted Clique Model in the regime $\alpha \in (0, \frac{1}{2})$. This proof is identical in spirit to the proof of Theorem 2 in \cite{GamarnikZadik}. Let $\epsilon = \frac{1}{2}$ in Lemma \ref{Lemma 15.4}, there exists $C_0(\epsilon); C_1>0$ such that
$$ \max_{z \in I} \Gamma_K(z) + C_1 K\log(K) \leq \Gamma_K\left( \left \lfloor C_0 \frac{K^2}{n} \right \rfloor\right)  \leq \Gamma_K\left(\frac{K}{2}\right), $$
where $ I\triangleq \mathbb{Z} \cap \left[D_1 \sqrt{K\log(K)}, D_2 \sqrt{K\log(K)}\right]$. Note in particular that since $K = o(\sqrt{n})$, the previous inequality is equivalent to (for large enough $n$)
\begin{align*}
     \max_{z \in I} \Gamma_K(z) + C_1 K\log(K) \leq \Gamma_K(0) \leq \Gamma_K\left(\frac{K}{2}\right). \numberthis \label{67}
\end{align*}

Using Proposition \ref{Proposition 15.2}, we see that in order to establish the claim of the Theorem, it suffices to prove that w.h.p as $n \to \infty$
$$  \min \left\{ \Psi^{\Bern(\frac{1}{2})}_K(G)(0), \Psi^{\Bern(\frac{1}{2})}_{K}(G)\left(\frac{K}{2}\right) \right\} \geq \Gamma_K\left(0\right) - o(K\log(K)).$$
Indeed, the latter combined with (\ref{67}) and Proposition \ref{Proposition 15.2} would yield
$$\min\left\{ \Psi^{\Bern(\frac{1}{2})}_K(G)(0), \Psi^{\Bern(\frac{1}{2})}_{K}(G)\left(\frac{K}{2}\right) \right\}  \geq \max_{z \in I} \Psi^{\Bern(\frac{1}{2})}_K(G)(z) + \frac{C_1}{2} K\log(K),  $$
which readily complete the proof of Theorem \ref{Theorem 3.13}. We first establish
\begin{align*}
\Psi^{\Bern(\frac{1}{2})}_K(G)(0) &\geq \Gamma_K\left(0\right) - o (K \log(K)). \numberthis\label{68}
\end{align*}
We have by Theorem \ref{Theorem 3.1}
\begin{align*}
    \Psi^{\Bern(\frac{1}{2})}_{K}(G)(0) &= \frac{1}{2} \binom{K}{2}  + \frac{1}{2} V(n-K,K) + O\left(\sqrt{K\log(n)^3}\right)\\
    &\geq \frac{1}{2} \binom{K}{2} + \sqrt{\frac{1}{2} \binom{K}{2} \log \left(\binom{n-K}{K}\right)} - o(K\log(K)) \numberthis \label{69}.
\end{align*}
Using Lemma  \ref{Lemma 4.9}, we have
\begin{align*}
    \Gamma_{K}(0)  &= h^{-1} \left( \log(2) -   \frac{ \log(   \binom{n-K}{K}  ) }{\binom{K}{2}}  \right)  \binom{K}{2}\\
    &= \binom{K}{2} \left[       \frac{1}{2} + \frac{1}{\sqrt{2}} \sqrt{   \frac{ \log(   \binom{n-K}{K}  ) }{\binom{K}{2}}} + O \left(\frac{ \log(   \binom{n-K}{K}  ) }{\binom{K}{2}}  \right)^{\frac{3}{2}}    \right]. \numberthis \label{70}
\end{align*}
By part 3 of Lemma \ref{Lemma 4.5}  we have $ \frac{ \log\left(   \binom{n-K}{K}  \right) }{\binom{K}{2}}  \sim \frac{K \log\left(\frac{n-K}{K}\right)}{\frac{K^2}{2}}= O\left(\frac{\log(n)}{K}\right)$. Henceforth $$\binom{K}{2} O  \left(\frac{ \log(   \binom{n-K}{K}  ) }{\binom{K}{2}}  \right)^{\frac{3}{2}}  = O\left(K^2 \frac{\log(n)^{\frac{3}{2}}}{K^{\frac{3}{2}}}\right) = O\left(\frac{\log(n)^{\frac{3}{2}}}{\sqrt{K}}\right) = o(K\log(K)).$$
From (\ref{69}) and (\ref{70}), we see that in order to establish (\ref{68}), it suffices to prove
\begin{align*}
    & \frac{1}{2} \binom{K}{2} + \sqrt{\frac{1}{2} \binom{K}{2} \log\left(\binom{n-K}{K}\right)} \geq \binom{K}{2} \left[       \frac{1}{2} + \frac{1}{\sqrt{2}} \sqrt{   \frac{ \log(   \binom{n-K}{K}  ) }{\binom{K}{2}}}\right] - o(K\log(K)),
\end{align*}
which trivially holds. Next we establish
$$\Psi^{Bern(\frac{1}{2})}_K(G)\left(\frac{K}{2}\right) \geq \Gamma_K(0) - o (K \log(K)).$$
We have by Lemma \ref{Lemma 15.5} (picking $a(n)= 4\sqrt{\log(n)}$):
\begin{align*}
    \Psi^{Bern(\frac{1}{2})}_K(G)\left(\frac{K}{2}\right) &\geq \binom{\frac{K}{2}}{2} + \Psi^{Bern(\frac{1}{2})}_{\frac{K}{2}}(G_0)(0) + \frac{K^2}{8} -K\sqrt{\log(n)} \text{ w.h.p as }n\to \infty.
\end{align*}
Where $G_0 \triangleq G \setminus \mathcal{PC}$.
We then have using Theorem \ref{Theorem 3.1}  w.h.p as $n\to \infty$
\begin{align*}
    \Psi^{Bern(\frac{1}{2})}_K(G)\left(\frac{K}{2}\right) &\geq \binom{\frac{K}{2}}{2}  + \frac{1}{2}\binom{\frac{K}{2}}{2}  + \frac{1}{2} V\left(n-K,\frac{K}{2}\right) -o(K\log(K)) + \frac{K^2}{8} - K \sqrt{\log(n)}\\
    &\sim \left[ \frac{1}{8} + \frac{1}{16} +\frac{1}{8}\right]K^2\\
    &= \frac{5}{16}K^2.
\end{align*}
Note that by (\ref{70}), it holds
$$ \Gamma_K(0) \sim \frac{1}{2} \binom{K}{2} \sim \frac{K^2}{4} . $$
Therefore
$$ \Psi^{Bern(\frac{1}{2})}_K(G)\left(\frac{K}{2}\right) - \Gamma_K(0) \gtrsim \frac{K^2}{4},$$
which yields the desired inequality $\Psi^{Bern(\frac{1}{2})}_K(G)\left(\frac{K}{2}\right) \geq \Gamma_K(0) - o (K \log(K))$. This completes the proof of  Theorem \ref{Theorem 3.13}. 
\end{proof}

\begin{proof}[Proof of Theorem \ref{Theorem 3.14}]
We now deal with the case $\alpha \in \left[ \frac{1}{2}, \frac{2}{3}\right)$ as stated in Theorem \ref{Theorem 3.14}. Let $ m\triangleq \lfloor C_0 \frac{K^2}{n} \rfloor$. Similarly to the case $\alpha \in \left(0,\frac{1}{2}\right)$, it suffices to prove that w.h.p as $n\to \infty$
$$  \min \left\{ \Psi^{Bern(\frac{1}{2})}_K(G)(m), \Psi^{Bern(\frac{1}{2})}_{K}(G)\left(\frac{K}{2}\right) \right\} \geq \Gamma_K(m) - o(K\log(K)).$$
We first prove that:
\begin{align*}
    \Psi^{Bern(\frac{1}{2})}_K(G)(m) &\geq \Gamma_K(m) - o(K\log(K)).  \numberthis \label{71}
\end{align*}
We have by Lemma \ref{Lemma 15.5} (picking $a(n) = 2\sqrt{\log(n)}$):
\begin{align*}
    \Psi^{Bern(\frac{1}{2})}_K(G)(m) &\geq \binom{m}{2} + \Psi^{Bern(\frac{1}{2})}_{K-m}(G_0) +   \frac{(K-m)m}{2} - \sqrt{ (K-m)m \log(n)}.
\end{align*}
Note then that $\sqrt{ (K-m)m \log(n)} \sim \sqrt{K C_0\frac{K^2}{n} \log(n)} = o(K\log(K))$. Combining the latter with Conjecture \ref{conjecture 3.3}, it follows
\begin{align*}
        \Psi^{Bern(\frac{1}{2})}_K(G)(m)&\geq  \binom{m}{2} + \frac{1}{2} \binom{K-m}{2} + \frac{1}{2} \sqrt{2 \binom{K-m}{2} \log\left(\binom{n-K}{K-m}\right)} + o(K\log(K)) +\frac{(K-m)m}{2}. 
\end{align*}
Using Lemma \ref{Lemma 4.9}
\begin{align*}
    \Gamma_K(m) &= \binom{m}{2} + h^{-1} \left( \log(2) - \frac{\log(\binom{K}{m}\binom{n-K}{K-m})}{\binom{K}{2}-\binom{m}{2}} \right) \left(\binom{K}{2} - \binom{m}{2} \right)\\
    &= \binom{m}{2} +\left(\binom{K}{2} - \binom{m}{2}\right) \left[\frac{1}{2} + \frac{1}{\sqrt{2}}  \sqrt{  \frac{\log(\binom{K}{m} \binom{n-K}{K-m})}{\binom{K}{2} - \binom{m}{2}} }  + O\left( \left[\frac{\log(\binom{K}{m} \binom{n-K}{K-m})}{\binom{K}{2} - \binom{m}{2}}  \right]^{3/2}\right)\right]\\
&= \frac{1}{2} \binom{m}{2} + \frac{1}{2} \binom{K}{2} + \frac{1}{\sqrt{2}} \sqrt{\left(\binom{K}{2} - \binom{m}{2}\right) \log\left(\binom{K}{m} \binom{n-K}{K-m}\right)  }  + O\left(  \sqrt{K}\log(n)^{3/2}\right) \numberthis \label{72}\\
&= \frac{1}{2} \binom{m}{2} + \frac{1}{2} \binom{K}{2} + \frac{1}{\sqrt{2}} \sqrt{\left (\binom{K}{2} - \binom{m}{2}\right) \log\left(\binom{K}{m} \binom{n-K}{K-m}\right)  }   + o( K \log(K) ) \numberthis \label{73}.
\end{align*}
Where we used part 3 of Lemma \ref{Lemma 4.5} to simplify the error term in (\ref{72}) since $\log\left( \binom{K}{m} \binom{n-K}{K-m} \right)  \sim m \log\left( \frac{K}{m}\right)+ K \log\left( \frac{n-K}{K-m} \right) \sim K \log\left( \frac{n}{K} \right)$. Note that
$$ \binom{m}{2} + \frac{1}{2} \binom{K-m}{2}  +\frac{(K-m)m}{2}  - \left(\frac{1}{2} \binom{m}{2} + \frac{1}{2} \binom{K}{2}\right) =0.$$
Therefore, in order to prove (\ref{71}), it suffices to prove
$$\Delta \triangleq \frac{1}{2} \sqrt{2 \binom{K-m}{2} \log\left(\binom{n-K}{K-m}\right)}  -\frac{1}{\sqrt{2}} \sqrt{\left (\binom{K}{2} - \binom{m}{2}\right) \log\left(\binom{K}{m} \binom{n-K}{K-m}\right)  }   \geq -o(K\log(K)) ,$$
note  that $\Delta = \frac{A}{B}$ where
\begin{alignat*}{2}
&A &&\triangleq  \frac{1}{2}  \binom{K-m}{2} \log\left(\binom{n-K}{K-m}\right)  -\frac{1}{2} \left (\binom{K}{2} - \binom{m}{2}\right) \log\left(\binom{K}{m} \binom{n-K}{K-m}\right)\\
 &B &&\triangleq \frac{1}{2} \sqrt{2 \binom{K-m}{2} \log\left(\binom{n-K}{K-m}\right)}  +\frac{1}{\sqrt{2}} \sqrt{\left (\binom{K}{2} - \binom{m}{2}\right) \log\left(\binom{K}{m} \binom{n-K}{K-m}\right)  }.
 \end{alignat*}
Using Lemma \ref{Lemma 4.5} and $m = o(K)$ we have
$$  B = \Theta(\sqrt{K^2 K \log(n)}) = \Theta(K^{\frac{3}{2}} \sqrt{\log(K)}),$$
and
\begin{align*}
    A  &= \frac{1}{2}\left[ \binom{K-m}{2} -\binom{K}{2} + \binom{m}{2} \right]\log\left(\binom{n-K}{K-m}\right) - \frac{1}{2} \left( \binom{K}{2} - \binom{m}{2}\right) \log\left(\binom{K}{m}\right)\\
    &= \frac{1}{2}\left[ m(m-K)\right]\log\left(\binom{n-K}{K-m}\right) - \frac{1}{2} \left( \binom{K}{2} - \binom{m}{2}\right) \log\left(\binom{K}{m}\right).
\end{align*}
Using part 3 of  Lemma \ref{Lemma 4.5} we have $m(m-K)\log\left(\binom{n-K}{K-m}\right) \sim  m K (K-m) \log \left( \frac{n-K}{K-m} \right) = \Theta \left( m K^2 \log(K) \right)$. Similarly we have $\left( \binom{K}{2} - \binom{m}{2} \right) \log \left( \binom{K}{m} \right) = \Theta \left( m K^2 \log(K) \right)$. Therefore
\begin{align*}
    A = O\left( m K^2 \log(K)  \right) = O\left( \frac{K^4}{n} \log(K) \right).
\end{align*}
We thus have:
\begin{align*}
\Delta &= \frac{A}{B}\\
&=   \frac{ O \left( K^4 \log(K) \right) }{ \Theta\left( nK^{\frac{3}{2}} \sqrt{\log(K)} \right)}\\
&= O\left( \frac{K^{\frac{5}{2}}}{n} \sqrt{\log(n)}\right)\\
&= o(K \log(K)),
\end{align*}
where the last line follows from the fact that $ K^{\frac{3}{2}} \leq n$. We have thus proven that $\Delta \geq -o(K\log(K))$ which ends the proof of (\ref{71}).
It remains to prove
$$\Psi^{\Bern(\frac{1}{2})}_K(G)\left(\frac{K}{2}\right) \geq \Gamma_K\left(m \right) - o(K\log(K)) . $$
We have by Lemma \ref{Lemma 15.5} (picking $a(n)= 4\sqrt{\log(n)}$)
\begin{align*}
    \Psi^{\Bern(\frac{1}{2})}_K(G)\left(\frac{K}{2}\right) &\geq \binom{\frac{K}{2}}{2} + \Psi^{\Bern(\frac{1}{2})}_{\frac{K}{2}}(G_0)(0) + \frac{K^2}{8} -K\sqrt{\log(n)} \text{ w.h.p as }n\to \infty.
\end{align*}
Using the assumption in Conjecture \ref{conjecture 3.3} we have w.h.p as $n\to \infty$:
$$ \Psi^{\Bern(\frac{1}{2})}_{\frac{K}{2}}(G_0)(0) \geq \frac{1}{2} \binom{\frac{K}{2}}{2}  V\left(n-K, \frac{K}{2}\right) - o(K \log(K)).$$
Henceforth
\begin{align*}
    \Psi^{\Bern(\frac{1}{2})}_K(G)\left(\frac{K}{2}\right) &\geq \binom{\frac{K}{2}}{2} +\frac{1}{2}\binom{\frac{K}{2}}{2}+ V\left(n-K, \frac{K}{2}\right) - o(K \log(K)) + \frac{K^2}{8} -K\sqrt{\log(n)}\\
    &= \binom{\frac{K}{2}}{2} +\frac{1}{2}\binom{\frac{K}{2}}{2}+ \sqrt{2 \binom{K}{2} \log\left(\binom{n-K}{\frac{K}{2}} \right)} + \frac{K^2}{8} - o(K \log(K))\\
    &\sim \frac{5}{16}K^2.
\end{align*}
On the other hand, we have from (\ref{73})
\begin{align*}
    \Gamma_K(m) &= \frac{1}{2} \binom{m}{2} + \frac{1}{2} \binom{K}{2} + \frac{1}{\sqrt{2}} \sqrt{\left (\binom{K}{2} - \binom{m}{2}\right) \log\left(\binom{K}{m} \binom{n-K}{K-m}\right)  }   + o( K \log(K) )\\
    &\sim \frac{K^2}{4}.
\end{align*}
Therefore
$$ \Psi^{\Bern(\frac{1}{2})}_K(G)\left(\frac{K}{2}\right)  -  \Gamma_K(m) \gtrsim \frac{K^2}{4},$$
which implies
$$\Psi^{\Bern(\frac{1}{2})}_K(G)\left(\frac{K}{2}\right) \geq \Gamma_K\left(m \right) - o(K\log(K)) , $$
and concludes the proof of Theorem \ref{Theorem 3.14}.
\end{proof}

%% file: Lindeberg.tex
In this section we will denote by  $X_{ij}$  the edge weights sampled from the Gaussian model $\mathcal{N} \left( \frac{1}{2}, \frac{1}{4} \right)$ and $Y_{ij}$ the edge weights sampled from some distribution $\mathcal{A}$  satisfying the assumptions of  Theorem \ref{Theorem 3.11}. In order to obtain bounds on the general disorder $\Psi^{\mathcal{A}}_K(G)$, we use an interpolation scheme between the distribution $\mathcal{A}$ and the Gaussian disorder case. The idea of interpolation between distributions to establish universal results isn't new and has been used in  prior works. However, as we show further below, the application of such method in our specific setting is  problematic. 
In particular we will show that a second vanilla application of the Lindeberg interpolation method provides a nontrivial asymptotics only in the regime $\alpha \in (4/5,1)$.\par

 We will denote vectors by $\mathbf{Z}$, and their entries by $Z_{k}\triangleq \mathbf{Z}_k$. Consider the following smooth max function for some inverse temperature parameter $\beta > 0$

$$f_{\beta} : \mathbb{R}^{\binom{n}{2}} \to \mathbb{R} : \mathbf{Z} \longmapsto  f(\mathbf{Z}) = \frac{1}{\beta} \log \left( \sum_{S \subset V(G) , |S|=K} \exp(\beta Z_S) \right),$$
where $Z_S$ is defined as in (\ref{eq:Z_S})
\begin{align}\label{eq:bold_Z_S}
Z_S\triangleq\sum_{\substack{1\le i<j\le n\\ i,j \in S}} Z_{ij}.
\end{align} 
The function $f_{\beta}$  is infinitely differentiable and its partial derivatives with respect to the variable $Z_{ij}$ (for fixed  $i,j$) are given by
\begin{align*}
  \partial_{ij} f_{\beta} (\mathbf{Z})  \triangleq  \frac{\partial f_{\beta}}{\partial Z_{ij}} &=\frac{U^{\mathbf{Z}}(i,j)}{U^{\mathbf{Z}}(i,j)+V^{\mathbf{Z}}(i,j)},\\
  \partial_{ij}^2 f_{\beta}(\mathbf{Z})  \triangleq\frac{\partial^2 f_{\beta}}{\partial Z_{ij}^2} &=\beta\frac{U^{\mathbf{Z}}(i,j)V^{\mathbf{Z}}(i,j)}{(U^{\mathbf{Z}}(i,j)+V^{\mathbf{Z}}(i,j))^2},\\
  \partial_{ij}^3 f_{\beta}(\mathbf{Z})  \triangleq\frac{\partial^3 f_{\beta}}{\partial Z_{ij}^3} &= \beta^2 \frac{U^{\mathbf{Z}}(i,j)V^{\mathbf{Z}}(i,j)(V^{\mathbf{Z}}(i,j)-U^{\mathbf{Z}}(i,j))}{(U^{\mathbf{Z}}(i,j)+V^{\mathbf{Z}}(i,j))^3},  
\end{align*}
where 
\begin{align*}
    U^{\mathbf{Z}}(i,j) &\triangleq \sum_{S \subset V(G) , |S|=K, (i,j) \in S} e^{\beta Z_S} ,\numberthis \label{U}\\
    V^{\mathbf{Z}}(i,j) &\triangleq \sum_{S \subset V(G) , |S|=K, (i,j) \not \in S } e^{\beta Z_S}, \numberthis \label{V}
\end{align*}
and  the notation $(i,j) \in S$, ( $(i,j) \not \in S$ resp.) is used to indicate that the set of vertices $S$ contains both vertices $i,j$ (doesn't contain at least one of the vertices $i,j$, resp.). We will also use the notation $U^{\mathbf{Z}}(e)$ where $e$ denotes an edge $(i,j)$. Note, in particular, that since for all $i,j$ the sum $U^{\mathbf{Z}}(i,j) + V^{\mathbf{Z}}(i,j)$ is independent of the the choice of $i,j$. We can  define
$$ P^{\mathbf{Z}} \triangleq U^{\mathbf{Z}}(i,j) + V^{\mathbf{Z}}(i,j) , \forall (i,j).$$
Since $U^{\mathbf{Z}}(i,j), V^{\mathbf{Z}}(i,j)\geq 0$ we  obtain the following upper bound on the third partial derivatives of $f_{\beta}$
\begin{align*} \left|\partial_{ij}^3 f_{\beta} (\mathbf{Z})\right| = \beta^2 \frac{U^{\mathbf{Z}}(i,j)V^{\mathbf{Z}}(i,j) }{(U^{\mathbf{Z}}(i,j)+V^{\mathbf{Z}}(i,j))^2} \frac{|V^{\mathbf{Z}}(i,j)-U^{\mathbf{Z}}(i,j)|}{U^{\mathbf{Z}}(i,j)+V^{\mathbf{Z}}(i,j)} \leq \beta^2 \frac{U^{\mathbf{Z}}(i,j)V^{\mathbf{Z}}(i,j)}{(U^{\mathbf{Z}}(i,j)+V^{\mathbf{Z}}(i,j))^2} \leq \frac{\beta^2}{4}. \numberthis \label{derivativebound}
\end{align*}
Using the above bound, we derive in the next section the vanilla application of the Lindeberg method and show its limitation.
\subsection{Direct Application of Lindeberg's Method.}
The Lindeberg method is an interpolation scheme that yields bounds on the difference between the expected values of two probability distributions evaluated at some sufficiently smooth function. We recall here a generalized statement of the Lindeberg method from \cite{Lindeberg}. 
\begin{theorem}\label{Theorem 16.1}
Suppose $\mathbf{X}$ and $\mathbf{Y}$ are random vectors in $\mathbb{R}^n$ with $\mathbf{Y}$ having independent components. For $1 \leq i \leq n$, let
\begin{align*}
    A_i  &\triangleq \E| \E [ X_i | X_1, ..., X_{i-1} ]  - \E[Y_i]|.\\
    B_i  &\triangleq \E| \E [ X_i^2 | X_1, ..., X_{i-1} ]  - \E[Y_i^2]|.
\end{align*}
Let $M_3$ be an upper bound on $\max_i ( \E|X_i|^3 + \E|Y_i|^3) $. Suppose $f: \mathbb{R}^n \to \mathbb{R}$ is a thrice continuously differentiable function, and for $r=1,2,3$ let $L_r(f)$ be a finite constant such that $|\partial_i^r f(x)|\leq L_r(f)$ for each $i$ and $x$, where $\partial_i^r$ denotes the $r$-fold derivative in the ith coordinate. Then
$$  |\E f(\mathbf{X}) - \E f(\mathbf{Y})| \leq \sum_{i=1}^{n} ( A_i L_1 (f) + \frac{1}{2} B_i L_2(f)) + \frac{1}{6} n L_3(f) M_3.$$
\end{theorem}
In our setting, $X_i, Y_i$ will be the  random variables associated with edge weights. In particular, $(\mathbf{X})_i$ (resp. $(\mathbf{Y})_i$) are i.i.d. Note that $\mathbf{X},\mathbf{Y}$  have equal first and second moments in our setting, therefore $\forall i, A_i = B_i = 0$. The inequality in Theorem \ref{Theorem 16.1} simplifies to
$$  |\E f(\mathbf{X}) - \E f(\mathbf{Y})| \leq   \frac{1}{6} \binom{n}{2} L_3(f_{\beta}) M_3 \leq \frac{n^2\beta^2}{48} M_3,$$
where we used (\ref{derivativebound}) in the last inequality and  $M_3 \triangleq \E_{Z\sim \Nc\left(\frac{1}{2}, \frac{1}{4}\right)}[|Z|^3] + \E_{Z\sim \A}[|Z|^3]$. Note that we have the following elementary inequality for any $\mathbf{Z} \in  \mathbb{R}^{\binom{n}{2}}$
$$  \max_{S \subset V(G) , |S|=K} Z_S \leq  f_{\beta}(\mathbf{Z}) \leq \max_{S \subset V(G) , |S|=K} Z_S+ \frac{ \log \left( \binom{n}{K}\right)}{\beta}  .  $$
Hence
$$ \E[ \max_{S \subset V(G) , |S|=K} X_S ]\leq    \E[ f_{\beta}(\mathbf{X})] \leq  \E[ \max_{S \subset V(G) , |S|=K} X_S ]+ \frac{ \log \left( \binom{n}{K}\right)}{\beta}   . $$
$$ \E[ \max_{S \subset V(G) , |S|=K} Y_S] \leq   \E[ f_{\beta}(\mathbf{Y})] \leq  \E[\max_{S \subset V(G) , |S|=K} Y_S ]+ \frac{ \log \left( \binom{n}{K}\right)}{\beta}  .  $$
Combining
\begin{align*}
    |\E[ \Psi^{\Nc\left(\frac{1}{2}, \frac{1}{4}\right)}_K ]-\E[ \Psi^{\mathcal{A}}_K ]| &\leq |  E[ f_{\beta}(\mathbf{X})]  -  E[ f_{\beta}(\mathbf{Y})| + \frac{ \log \left( \binom{n}{K}\right)}{\beta} \numberthis \label{maxineq}\\
    &\leq \frac{n^2\beta^2}{48} M_3 +\frac{ \log \left( \binom{n}{K}\right)}{\beta}  \\
    &= \Theta\left(  \beta^2 n^2 + \frac{K\log(n)}{\beta} \right) .
\end{align*}
Since only the right side  depends on $\beta$, it can minimized  with respect to $\beta >0$. This minimum is reached for $\beta^* = \left(\frac{K \log(n)}{2 n^2} \right)^{\frac{1}{3}}$ and yields
\begin{align*}
    |\E[ \Psi^{\Nc\left(\frac{1}{2}, \frac{1}{4}\right)}_K ]-\E[ \Psi^{\mathcal{A}} _K ]|  = O(n^{ \frac{2}{3}} K^{\frac{2}{3}} \log(n)^{\frac{2}{3}}) . \numberthis \label{weaklindeberg}
\end{align*}

Since we are able to obtain asymptotics up to order  $o \left(  K^{\frac{3}{2}} \sqrt{\log(n/K)}\right)$  in Corollaries \ref{Corollary 3.5} and \ref{Corollary 3.9}, the bound in (\ref{weaklindeberg}) is interesting only if it is smaller than $K^{\frac{3}{2}} \sqrt{\log(n/K)}$. However, this is not the case unless $K^{\frac{3}{2}} \ge K^{\frac{2}{3}}n^{\frac{2}{3}}$, or equivalently if $\alpha \ge \frac{4}{5}$.

Our next goal is deriving tighter bounds from the Lindeberg's principle.


\subsection{Rederiving Lindeberg's Principle .}
In the remaining of this section we fix a distribution $\mathcal{A}$ satisfying the assumptions in Theorem \ref{Theorem 3.11}. Let $e_1,...,e_N, N\triangleq \binom{n}{2}$ be an enumeration of all pairs $(i,j), i\neq j$.
For $1\leq l \leq N$ let 
$$\mathbf{W}^{l} \triangleq (X_{e_1},...,X_{e_l},Y_{e_{l+1}},...,X_{e_N})= (w^{l}_s)_{1\leq s \leq N} \in \mathbb{R}^{N}$$
be the vector obtained by using the Gaussian random variables $X_{e_i}$ up to and including the edge $e_l$, then the $\A$ random variables $Y_{e_{i}}, i\geq l+1$. Similarly, let
\begin{alignat*}{2}
    &\tilde{\mathbf{W}}^{l} &&\triangleq (X_{e_1}...X_{e_{l-1}},0,Y_{e_{l+1}},...,Y_{e_N}) \in \mathbb{R}^{N},\\
    &\tilde{\mathbf{W}}^{l}(w) &&\triangleq (X_{e_1}...X_{e_{l-1}},w,Y_{e_{l+1}},...,Y_{e_N}) \in \mathbb{R}^{N}.
\end{alignat*}
Note that
$$ \E[f_{\beta}(\mathbf{X})] - \E[f_{\beta}(\mathbf{Y})] = \sum_{1 \leq l \leq N} (\E[f_{\beta}(\mathbf{W}^{l})] - \E[f_{\beta}(\mathbf{W}^{l-1})]),$$
where we use the convention $\mathbf{W}^{0} \triangleq \mathbf{Y}$. Using third-order Taylor expansion, we have
\begin{align*}
f_{\beta}(\mathbf{W}^{l}) - f_{\beta}(\tilde{\mathbf{W}}^{l}) - X_{e_l} \partial_{e_l} f_{\beta}(\tilde{\mathbf{W}}^{l}) - \frac{(X_{e_l})^2}{2} \partial_{e_l}^2 f_{\beta}(\tilde{\mathbf{W}}^l)     & = \frac{1}{6} \int_{0}^{X_{e_l}} \partial_{e_l}^3 f_{\beta}(...,X_{e_{l-1}},u,Y_{e_{l+1}},...) (X_{e_l} -u)^2 du \\
&= \frac{1}{6} \int_{0}^{X_{e_l}} \partial_{e_l}^3 f_{\beta}(\tilde{\mathbf{W}}^{l}(u)) (X_{e_l} -u)^2 du.
\end{align*}
Similarly
\begin{align*}
f_{\beta}(\mathbf{W}^{l-1}) - f_{\beta}(\tilde{\mathbf{W}}^{l}) - Y_{e_{l}} \partial_{e_l} f_{\beta}(\tilde{\mathbf{W}}^{l}) - \frac{(Y_{e_{l}})^2}{2} \partial_{e_l}^2 f_{\beta}(\tilde{\mathbf{W}}^{l})      &= \frac{1}{6} \int_{0}^{Y_{e_{l}}} \partial_{e_l}^3 f_{\beta}(...,X_{e_{l-1}},u,Y_{e_{l+1}},...) (Y_{e_{l}}-u)^2 du\\
&= \frac{1}{6} \int_{0}^{Y_{e_{l}}} \partial_{e_l}^3 f_{\beta}(\tilde{\mathbf{W}}^{l}(u)) (Y_{e_{l}}-u)^2 du.
\end{align*}
Subtracting and taking the expected value of the above two equations yields
\begin{align*}
 |\E[f_{\beta}(\mathbf{W}^l)]- \E[f_{\beta}(\mathbf{W}^{l-1})]| &= \frac{1}{6} \left|  \E\left[ \int_{Y_{e_l}}^{X_{e_l}} \partial^3_{e_l} f_{\beta}(\tilde{\mathbf{W}}^{l}(u)) \left[ (X_{e_l}-u)^2 \mathbf{1}_{u\in (0,X_{e_l})} - (Y_{e_l}-u)^2\mathbf{1}_{u\in (Y_{e_l},0)} \right]du\right] \right| ,
 \end{align*}
where the interval notation $(0,X_{e_l})$ ($(Y_{e_l} ,0)$ resp.) should be understood as $(X_{e_l},0)$ ($(0,Y_{e_l} )$ resp.) if $X_{e_l}<0$ ($Y_{e_l}>0$ resp.). Note that $\left[ (X_{e_l}-u)^2 \mathbf{1}_{u\in (0,X_{e_l})} - (Y_{e_l}-u)^2\mathbf{1}_{u\in (Y_{e_l},0)} \right] \leq |X_{e_l}|^2 + |Y_{e_l}|^2$, therefore
\begin{align*}
 |\E[f_{\beta}(\mathbf{W}^l)]- \E[f_{\beta}(\mathbf{W}^{l-1})]|&\leq \frac{1}{6} \left|  \E\left[ \int_{Y_{e_l}}^{X_{e_l}} \partial^3_{e_l} f_{\beta}(\tilde{\mathbf{W}}^{l}(u)) \left[ |X_{e_l}|^2 + |Y_{e_l}|^2 \right]du\right] \right|\\
 &= \frac{1}{6} \E\left[|X_{e_l} - Y_{e_l}|\left( |X_{e_l}|^2 + |Y_{e_l}|^2\right)   |\partial^{3}_{e_l}f_{\beta} (\tilde{\mathbf{W}}^{l}(w))|\right],
\end{align*}
where $w$ is an unknown random variable in the interval $\left( \min \left( X_{e_l}, Y_{e_l}) \right) , \max\left((X_{e_l},Y_{e_l}\right ) \right)$. Since both $\Nc\left(\frac{1}{2},\frac{1}{4}\right)$ and $\A$ are sub-Gaussian distributions, we can find  $A>0$ such that $ \frac{1}{6}|X_{e_l} - Y_{e_l}|\left( |X_{e_l}|^2 + |Y_{e_l}|^2\right)  > A$ with probability at most  $e^{-\Theta(A^2)}$. Combining this  with (\ref{derivativebound}) yields
\begin{align*}
 \frac{1}{6} \E\left[|X_{e_l} - Y_{e_l}|\left( |X_{e_l}|^2 + |Y_{e_l}|^2\right)   |\partial^{3}_{e_l}f_{\beta} (\tilde{\mathbf{W}}^{l}(w))|\right]&\leq A \E[|\partial^{3}_{e_l}f_{\beta} (\tilde{\mathbf{W}}^{l}(w))|] + \frac{\beta^2}{4} \E\left[|X_{e_l} - Y_{e_l}|^2\left( |X_{e_l}|^2 + |Y_{e_l}|^4\right) \right]^{\frac{1}{2}} e^{-\Theta(A^2)} \\
 &= A \beta^2 \E\left[\frac{U^{\tilde{\mathbf{W}}^{l}(w)}(e_l)V^{\tilde{\mathbf{W}}^{l}(w)}(e_l) }{(U^{\tilde{\mathbf{W}}^{l}(w)}(e_l)+V^{\tilde{\mathbf{W}}^{l}(w)}(e_l))^2}  \right]+ O\left( \beta^2 e^{-\Theta(A^2)} \right)\\
 &\leq A \beta^2 \E\left[\frac{U^{\tilde{\mathbf{W}}^{l}(w)}(e_l)}{U^{\tilde{\mathbf{W}}^{l}(w)}(e_l)+V^{\tilde{\mathbf{W}}^{l}(w)}(e_l)} \right] +O\left(  \beta^2 e^{-\Theta(A^2)} \right) \\
 &= A \beta^2 \E\left[\frac{U^{\tilde{\mathbf{W}}^{l}(w)}(e_l)}{P^{\tilde{\mathbf{W}}^{l}}(e_l)} \right] + O\left( \beta^2 e^{-\Theta(A^2)} \right), 
\end{align*}
where we used Cauchy Schwarz inequality in the first line and the trivial inequality $V^{\tilde{\mathbf{W}}^{l}(w)}(e_l)\leq U^{\tilde{\mathbf{W}}^{l}(w)}(e_l)+V^{\tilde{\mathbf{W}}^{l}(w)}(e_l) $ in the third line. Note that  $w$ has unknown distribution within the interval $(\min (X_{e_l}, Y_{e_l}),\max (X_{e_l}, Y_{e_l}))$. We can nonetheless upper bound the term inside the expectation above by

\begin{align*}
  \frac{U^{\tilde{\mathbf{W}}^{l}(w)}(e_l)}{U^{\tilde{\mathbf{W}}^{l}(w)}(e_l)+V^{\tilde{\mathbf{W}}^{l}(w)}(e_l)} &\leq \frac{U^{\tilde{\mathbf{W}}^{l}(X_{e_l})}(e_l)}{U^{\tilde{\mathbf{W}}^{l}(X_{e_l})}(e_l)+V^{\tilde{\mathbf{W}}^{l}(X_{e_l})}(e_l)}  + \frac{U^{\tilde{\mathbf{W}}^{l}(Y_{e_l})}(e_l)}{U^{\tilde{\mathbf{W}}^{l}(Y_{e_l})}(e_l)+V^{\tilde{\mathbf{W}}^{l}(Y_{e_l})}(e_l)} \\
  &= \frac{U^{{\mathbf{W}}^{l}}(e_l)}{U^{{\mathbf{W}}^{l}}(e_l)+V^{{\mathbf{W}}^{l}}(e_l)}  + \frac{U^{{\mathbf{W}}^{l-1}}(e_l)}{U^{{\mathbf{W}}^{l-1}}(e_l)+V^{{\mathbf{W}}^{l-1}}(e_l)} . 
\end{align*}
Indeed, notice that the left hand side can be written in the form $ \frac{e^{\beta w}A}{e^{\beta w}A + B}$ where $A,B$ don't depend on $w$. Since the function $\mathbb{R} \to \mathbb{R}: x \to \frac{e^{\beta x} A}{e^{\beta x}A+B}$ is monotonic when $A,B\geq0$, then 
$$ \frac{e^{\beta w} A}{e^{\beta w}A+B} \leq \max \left\{ \frac{e^{\beta X_{e_l}} A}{e^{\beta X_{e_l}}A+B} , \frac{e^{\beta Y_{e_l}} A}{e^{\beta Y_{e_l}}A+B} \right\} \leq  \frac{e^{\beta X_{e_l}} A}{e^{\beta X_{e_l}}A+B} + \frac{e^{\beta Y_{e_l}} A}{e^{\beta Y_{e_l}}A+B} . $$
Therefore
\begin{align*}
    |\E[f_{\beta}(\mathbf{W}^{l})]- \E[f_{\beta}(\mathbf{W}^{l-1})]|  \leq  A \beta^2 \E\left[\frac{U^{{\mathbf{W}}^{l}}(e_l)}{P^{{\mathbf{W}}^{l}}}  + \frac{U^{{\mathbf{W}}^{l-1}}(e_l)}{P^{{\mathbf{W}}^{l-1}}}  \right] + O\left( \beta^2 e^{-\Theta(A^2)}  \right) , \numberthis \label{eq1} 
\end{align*}
summing the above yields the following main upper bound
\begin{align*}
|\E[f_{\beta}(\mathbf{X})] - \E[f_{\beta}(\mathbf{Y})]| \leq  A\beta^2 \sum_{l=1}^{N} \E\left[\frac{U^{{\mathbf{W}}^{l}}(e_l)}{P^{{\mathbf{W}}^{l}}}  + \frac{U^{{\mathbf{W}}^{l-1}}(e_l)}{P^{{\mathbf{W}}^{l-1}}}  \right] + O\left( \beta^2e^{-\Theta(A^2)}\binom{n}{2} \right). \numberthis \label{lindebergsum}    
\end{align*}

\subsection{Lindeberg's Method Agregated Out. Proof of Theorem \ref{Theorem 3.11}.}
The inequality (\ref{lindebergsum}) was obtained using the interpolation \textit{order} induced by the enumeration $e_1,...,e_N$. We can view this as starting from the \textit{State} $\mathbf{Y}$ then at step $1 \leq l \leq N$ switching the edge weight $Y_{e_l}$ with $X_{e_l}$. It is unclear how to directly derive upper bounds on (\ref{lindebergsum}) that are tighter than (\ref{weaklindeberg}), as the behavior of the terms $U^{{\mathbf{W}}^{l}}(e_l)/P^{{\mathbf{W}}^{l}}$ depends on the ordering of edges used and induces a lack of symmetry in the summation. To restore the symmetry  we choose the enumeration of edges to be uniformly random. Specifically, let $\sigma \in \mathfrak{S}_N$ be chosen uniformly at random, where we recall that $\mathfrak{S}_N$ is the set of all permutations of $N$ elements. Consider the interpolation induced by the enumeration  $e_{\sigma(1)},...,e_{\sigma(N)}$. We define  the \textit{State} vectors $\mathbf{W}^{l,\sigma}, 1\leq l \leq N$ as
$$\mathbf{W}^{l,\sigma} \triangleq (X_{e_{\sigma(1)}},...,X_{e_{\sigma(l)}},Y_{e_{\sigma(l+1)}},...,X_{e_{\sigma(N)}})= (w^{l,\sigma}_s)_{1\leq s \leq N} \in \mathbb{R}^{N}.$$
Note that using our previous notation, we have $\mathbf{W}^{l}=\mathbf{W}^{l,\text{id}_{\mathfrak{S}_N} }$. Therefore summing over all possible permutations the bound (\ref{lindebergsum}),  yields
\begin{align*}
\hspace{-.25cm}
N! |\E[f_{\beta}(\mathbf{X})] - \E[f_{\beta}(\mathbf{Y})]| &\leq  A\beta^2 \sum_{\sigma \in \mathfrak{S}_N}\sum_{i=1}^{N} \E\left[\frac{U^{{\mathbf{W}}^{l,\sigma}}(e_{\sigma(l)})}{P^{{\mathbf{W}}^{l,\sigma}}}  + \frac{U^{{\mathbf{W}}^{l-1,\sigma}}(e_{\sigma(l)})}{P^{{\mathbf{W}}^{l-1,\sigma}}}  \right] + O\left( N!\beta^2 n^2e^{-\Theta(A^2)} \right). \numberthis \label{81} 
\end{align*}
Surprisingly, the right hand side above can be computed explicitly using a double counting argument that we detail next. Given a binary function $g :\{ e_1,...,e_N\} \mapsto \{1, 0 \}$, we consider the \textit{State} $\textbf{S}_g$ induced by $g$ to be the set of edge weights formed by $\{X_{e_l} |  g(e_l)=1, l\in [N]\} \cup \{Y_{e_l} | g(e_l)=0, l\in[N]\}$. Let $\mathcal{E}$ be the set of all states, and $\mathcal{E}_p , 0\leq p \leq N$ be the set of states induced by functions $g$ that achieve value $1$ exactly $p$ times (and achieve value $0$ exactly $q=N-p$ times). Note that every $\mathbf{W}^{l,\sigma}$ is associated with a state in $\mathcal{E}$.
Given a state $\mathbf{S}_g\in \mathcal{E}$ and edge $e \in \{e_1,...,e_N\}$ we can define  the quantities $U^{\textbf{S}_g}(e), V^{\textbf{S}_g}(e)$ similarly to (\ref{U}), (\ref{V}). Therefore, we can view the upper bound in (\ref{81}) as a sum over pairs of states and edges. Let $0 \leq p \leq N$ and consider now a fixed state $\textbf{S}_g\in \mathcal{E}_p$ and a fixed edge $e \in \{e_1,...,e_N\}$. We will count the number of times the term $U^{\mathbf{S}_g}(e)/P^{\mathbf{S}_g}$ appears in the summation in the right-hand side of (\ref{81}). We consider two cases.

\textbf{Case 1 : $g(e)=1$ }:
Consider first the terms of the form $\frac{U^{{\mathbf{W}}^{l,\sigma}}(e_{\sigma(l)})}{P^{{\mathbf{W}}^{l,\sigma}}} $ in (\ref{81}). The latter equals $\frac{U^{\mathbf{S}_g}(e)}{P^{\mathbf{S}_g}}$ if  the interpolation induced by the enumeration $e_{\sigma(1)},...,e_{\sigma(N)}$ switches the edge weight of $e$ to reach the state $\mathbf{S}_g$.

Notice  that we transit through $\mathbf{S}_g$ exactly $p!q!$ times over all possible interpolation orders induced by permutations $\sigma$, where $q=m-p$. Once we are at state $\mathbf{S}_g$, there are $p$ edges achieving value $1$ at $g$, and each of them is equally likely to have been the switch used to get to $\mathbf{S}_g$. Hence we transit through $\mathbf{S}_g$ using the edge switch $e$ exactly $\frac{p!q!}{p}$ times.
Note that the terms $\frac{U^{{\mathbf{W}}^{l-1,\sigma}}(e_{\sigma(l)})}{P^{{\mathbf{W}}^{l-1,\sigma}}}$ are never equal to $\frac{U^{\mathbf{S}_g}(e)}{P^{\mathbf{S}_g}}$ since we assumed $g(e)=1$, which does not hold on states induced by $\mathbf{W}^{l-1,\sigma}$.

\textbf{Case 2 : $g(e)=0$} :
 Consider  the terms $\frac{U^{{\mathbf{W}}^{l-1,\sigma}}(e_{\sigma(l)})}{P^{{\mathbf{W}}^{l-1,\sigma}}} $ in (\ref{81}). The latter equals $\frac{U^{\mathbf{S}_g}(e)}{P^{\mathbf{S}_g}}$ if  the interpolation induced by the enumerations $e_{\sigma(1)},...,e_{\sigma(N)}$  switches the edge weight of $e$ to obtain the next state after $\mathbf{S}_g$. As before, we transit through $\mathbf{S}_g$ exactly $p!q!$ times over all possible interpolation orders induced by permutations $\sigma$, where $q=N-p$. There are $q=N-p$ edges achieving value $0$ at $g$, and each of them is "equally likely" to be the edge  switched  to obtain the next state. Hence, $e$ was the switch edge used when transitioning through $\mathbf{S}_g$ to the next state exactly $\frac{p!q!}{q}$ times. Similarly to the previous case, the terms $\frac{U^{{\mathbf{W}}^{l,\sigma}}(e_{\sigma(l)})}{P^{{\mathbf{W}}^{l,\sigma}}}$ are never equal to $\frac{U^{\mathbf{S}_g}(e)}{P^{\mathbf{S}_g}}$ since we assumed $g(e)=0$, which does not hold on states induced by $\mathbf{W}^{l,\sigma}$.\\

Combining the above observations, we have
\begin{align*}
\sum_{\sigma \in \mathfrak{S}_N}\sum_{i=1}^{N} \left[\frac{U^{{\mathbf{W}}^{l}}(e_l)}{P^{{\mathbf{W}}^{l}}}  + \frac{U^{{\mathbf{W}}^{l-1}}(e_l)}{P^{{\mathbf{W}}^{l-1}}}  \right]  
&= \sum_{p=1}^{N} \sum_{\substack{ \mathbf{S_g} \in \mathcal{E}_p\\  e\in\{e_1,...,e_N\}\\  g(e)=1}} \frac{p! q!}{p} \frac{U^{\mathbf{S}_g}(e)}{P^{\mathbf{S}_g}} + \sum_{p=0}^{N-1} \sum_{\substack{ \mathbf{S_g} \in \mathcal{E}_p\\ e\in\{e_1,...,e_N\}\\  g(e)=0}} 
 \frac{p!q!}{q}\frac{U^{\mathbf{S}_g}(e)}{P^{\mathbf{S}_g}}\\
 &= B_0 + B_N + B_{1}^{N-1},
\end{align*}
where
\begin{alignat*}{1}
B_0 &\triangleq  \sum_{\substack{ \mathbf{S_g} \in \mathcal{E}_0\\ e\in\{e_1,...,e_N\}\\  g(e)=0}} 
 (N-1)! \frac{U^{\mathbf{S}_g}(e)}{P^{\mathbf{S}_g}},\\
B_N & \triangleq    \sum_{\substack{ \mathbf{S_g} \in \mathcal{E}_N\\  e\in\{e_1,...,e_N\}\\  g(e)=1}} (N-1)! \frac{U^{\mathbf{S}_g}(e)}{P^{\mathbf{S}_g}},\\
B_{1}^{N-1} &\triangleq \sum_{p=1}^{N-1} \left[ \sum_{\substack{ \mathbf{S_g} \in \mathcal{E}_p\\  e\in\{e_1,...,e_N\}\\  g(e)=1}} \frac{p! q!}{p} \frac{U^{\mathbf{S}_g}(e)}{P^{\mathbf{S}_g}} +  \sum_{\substack{ \mathbf{S_g} \in \mathcal{E}_p\\ e\in\{e_1,...,e_N\}\\  g(e)=0}} 
 \frac{p!q!}{q}\frac{U^{\mathbf{S}_g}(e)}{P^{\mathbf{S}_g}}\right] .
\end{alignat*}
In the remaining computations we will upper bound $B_0,B_N, B_{1}^{N-1}$. We first show the following result.
\begin{lemma}\label{Lemma 15.2}
Suppose $\mathbf{S}_g \in \mathcal{E}$, then
$$ \sum_{e \in \{e_1,...,e_N\}} U^{\mathbf{S}_g} (e) = \binom{K}{2} P^{\mathbf{S}_g}$$
\end{lemma}
\begin{proof}
Let $\mathbf{Z} \in \mathbb{R}^N$ be such that $\mathbf{Z}_{e_l} = Z_{e_l}= X_{e_l} \mathbf{1}_{g(e_l) = 1} +  Y_{e_l} \mathbf{1}_{g(e_l) =0}$ and let $G^{\mathbf{S}_g}$ be the graph induced by the edge weights $Z_{e_l}$. We have
\begin{align*}
    \sum_{e \in \{e_1,...,e_N\}} U^{\mathbf{S}_g}(e)&= \sum_{e \in \{e_1,...,e_N\}} \left [ \sum_{S\subset V(G^{\mathbf{S}_g}), |S|=K , e \in S } e^{\beta Z_{S}} \right] \\
    &=  \sum_{S\subset V(G^{\mathbf{S}_g}) , |S|=K  } \left[ \sum_{e \in S, e \in \{e_1,...,e_N\} } e^{\beta Z_{S}} \right].
\end{align*}
Note that for each fixed subset $S\subset V(G^{\mathbf{S}_g})$, there are exactly $\binom{K}{2}$ edges $e$ such that $e \in \{e_1,...,e_N\}, e\in S$, therefore
\begin{align*}
    \sum_{e \in \{e_1,...,e_N\}} U^{\mathbf{S}_g}(e)&=  \sum_{ S \subset V(G^{\mathbf{S}_g}) , |S|=K  } \left[\binom{K}{2} e^{\beta Z_{S}} \right] \\
    &= \binom{K}{2} \sum_{ S \subset V(G^{\mathbf{S}_g}) , |S|=K   } e^{\beta Z_{S}} \\
    &= \binom{K}{2} P^{\mathbf{S}_g}.
\end{align*}
\end{proof}
Note that $\mathcal{E}_0$ consists of the unique state $\mathbf{S}_{\mathbf{0}} \triangleq \mathbf{S}_{\{e_1,...,e_N\}\mapsto \{0\}^{N}} $. Similarly  $\mathcal{E}_N$ consists of the unique state $\mathbf{S}_{\mathbf{1}} \triangleq \mathbf{S}_{\{e_1,...,e_N\}\mapsto \{1\}^{N}}$. We thus have using Lemma \ref{Lemma 15.2}
\begin{align*}
    B_0 &=  \sum_{\substack{  e\in\{e_1,...,e_N\}\\  g(e)=0}} 
 (N-1)! \frac{U^{\mathbf{S}_{\mathbf{0}}}(e)}{P^{\mathbf{S}_{\mathbf{0}}}}\\
 &= \frac{(N-1)!}{{P^{\mathbf{S}_{\mathbf{0}}}}} \sum_{e \in \{e_1, ...,e_N\}} 
  U^{\mathbf{S}_{\mathbf{0}}}(e)\\
  &= (N-1)! \binom{K}{2}.
\end{align*}
Similarly we obtain $B_N = (N-1)! \binom{K}{2}$. We next upper bound $B_{1}^{N-1}$. We have for $1 \leq p \leq N-1$ using Cauchy Schwartz inequality
\begin{align*}
     \sum_{\substack{ \mathbf{S_g} \in \mathcal{E}_p\\  e\in\{e_1,...,e_N\}\\  g(e)=1}} \frac{p! q!}{p} \frac{U^{\mathbf{S}_g}(e)}{P^{\mathbf{S}_g}} +  \sum_{\substack{ \mathbf{S_g} \in \mathcal{E}_p\\ e\in\{e_1,...,e_N\}\\  g(e)=0}} 
 \frac{p!q!}{q}\frac{U^{\mathbf{S}_g}(e)}{P^{\mathbf{S}_g}} &= \sum_{ \mathbf{S}_g \in \mathcal{E}_p} \frac{p!q!}{P^{\mathbf{S}_g}} \left[   \frac{1}{p} \left(  \sum_{\substack{ \  e\in\{e_1,...,e_N\}\\  g(e)=1}} U^{\mathbf{S}_g}(e) \right) + 
 \frac{1}{q} \left( \sum_{\substack{ e\in\{e_1,...,e_N\}\\  g(e)=0}} U^{\mathbf{S}_g}(e) \right) \right] \\
 &\leq \sum_{ \mathbf{S}_g \in \mathcal{E}_p} \frac{p!q!}{P^{\mathbf{S}_g}} \left(\frac{1}{p^2}+\frac{1}{q^2} \right)^{1\over 2}\\& \left[ \left(  \sum_{\substack{ \  e\in\{e_1,...,e_N\}\\  g(e)=1}} U^{\mathbf{S}_g}(e) \right)^2 + 
 \left( \sum_{\substack{ e\in\{e_1,...,e_N\}\\  g(e)=0}} U^{\mathbf{S}_g}(e) \right)^2 \right] ^{1\over 2}\\
 &\leq \sum_{ \mathbf{S}_g \in \mathcal{E}_p} \frac{p!q!}{P^{\mathbf{S}_g}} \left(\frac{1}{p^2}+\frac{1}{q^2} \right)^{1\over 2} \left[ \sum_{\substack{ \  e\in\{e_1,...,e_N\}\\  g(e)=1}} U^{\mathbf{S}_g}(e)  + 
  \sum_{\substack{ e\in\{e_1,...,e_N\}\\  g(e)=0}} U^{\mathbf{S}_g}(e)  \right] \\
  &=  \sum_{ \mathbf{S}_g \in \mathcal{E}_p} \frac{p!q!}{P^{\mathbf{S}_g}} \left(\frac{1}{p^2}+\frac{1}{q^2} \right)^{1\over 2} \binom{K}{2} P^{\mathbf{S}_g}\\
  &=  \binom{K}{2}\sum_{ 1\leq p \leq N-1} p!q! \frac{(p^2+q^2)^{1\over 2}}{pq} |\mathcal{E}_p|   \\
  &=  \binom{K}{2}\sum_{ 1\leq p \leq N-1} p!q! \frac{(p^2+q^2)^{1\over 2}}{pq} \binom{N}{p}\\
  &\leq \binom{K}{2} N!\sum_{ 1\leq p \leq N-1}  \left(\frac{1}{p}+ \frac{1}{q}\right) \\
  &\lesssim N! \frac{K^2}{2} 2\log(N)\\
  &\sim 2N! K^2 \log(n) .
\end{align*}
Combining the bounds of $B_, B_N, B_{1}^{N-1}$ we obtain
\begin{align*}
    \sum_{\sigma \in \mathfrak{S}_N}\sum_{i=1}^{N} \left[\frac{U^{{\mathbf{W}}^{l}}(e_l)}{P^{{\mathbf{W}}^{l}}}  + \frac{U^{{\mathbf{W}}^{l-1}}(e_l)}{P^{{\mathbf{W}}^{l-1}}}  \right]  &\lesssim 2N! K^2 \log(n) + 2 (N-1)! \binom{K}{2}\\
    &\sim 2N! K^2 \log(n).
\end{align*}
Using (\ref{81}) we have
$$ |\E[f_{\beta}(\mathbf{X})] - \E[f_{\beta}(\mathbf{Y})]| \lesssim  A\beta^2 2 K^2 \log(n) + O\left(\beta^2 n^2 e^{-\Theta(A^2)}  \right).$$
From (\ref{maxineq})
\begin{align*}
 \E[ \Psi^{\Nc\left(\frac{1}{2}, \frac{1}{4}\right)}_K ]-\E[ \Psi^{\mathcal{A}}_K ]| &\leq |\E[f_{\beta}(\mathbf{X})] - \E[f_{\beta}(\mathbf{Y})]| + \frac{ \log \left( \binom{n}{K}\right)}{\beta}\\
 &\lesssim  A\beta^2 2 K^2 \log(n) +O\left(\beta^2 n^2 e^{-\Theta(A^2)}  \right) + \frac{\log\left( \binom{n}{K}\right)}{\beta} \\
 &= \Theta\left( A\beta^2  K^2 \log(n) +   \beta^2n^2 e^{-\Theta(A^2)}  + \frac{K\log\left( n\right)}{\beta} \right).
\end{align*}
This holds for any large enough $A$ (the magnitude of $A$ is independent of $K,n$ by construction). Let $M>0$ be the constant that appears in $\beta^2 n^2 e^{\Theta(A^2)}$ in the sense that this expression is at most $\beta^2 n^2 e^{-MA^2} $ for all $A>0$. Pick $A = \frac{\sqrt{2}}{\sqrt{M}}\sqrt{\log(n)}$ so that
$ n^2 e^{-\Theta(A^2)} \leq n^2 e^{- 2\log(n)}=1$.
We then have
$$   \E[ \Psi^{\Nc\left(\frac{1}{2}, \frac{1}{4}\right)}_K ]-\E[ \Psi^{\mathcal{A}}_K ]| \leq  \Theta\left( \beta^2  K^2 \log(n)^{\frac{3}{2}}  +  \frac{K\log\left( n\right)}{\beta} \right).$$
The right hand side can be minimized with respect to $\beta >0$. The minimizer is  given by $\beta^* \sim  \left( \frac{1}{2 K\sqrt{\log(n)}}  \right)^{\frac{1}{3}} $ and yields
$$  \E[ \Psi^{\Nc\left(\frac{1}{2}, \frac{1}{4}\right)}_K ]-\E[ \Psi^{\mathcal{A}}_K ]| = O(K^{\frac{4}{3}} \log(n)^{\frac{7}{6}}).$$
This concludes the proof of the first part of Theorem \ref{Theorem 3.11}.

\subsection{Proof of Second Part of Theorem \ref{Theorem 3.11}.}
Let $\A$ be a bounded distribution  satisfying the assumptions of Theorem \ref{Theorem 3.11}. We claim the following result.
\begin{lemma}\label{Lemma 16.2}
    There exists universal constants $\beta, \theta >0$ s.t for $t\geq \beta K$ it holds
    $$ \p\left(| \Psi^{\A}_K(G) - \E[      \Psi^{\A}_K(G)   ]| \geq t \right) \leq \exp\left(-\frac{t^2}{\theta K^2}\right) .$$
\end{lemma}
We skip the proof of Lemma \ref{Lemma 16.2}  for now and show how it leads to asymptotics for $\Psi^{\A}_K(G)$. We have using Borell-Tis inequality (Theorem \ref{Theorem 4.16}) and the definition of $\Psi$ for $t>0$
\begin{align}
    \p\left(| \Psi^{\Nc\left(\frac{1}{2}, \frac{1}{4}\right)}_K(G) - \E[      \Psi^{\Nc\left(\frac{1}{2}, \frac{1}{4}\right)}_K(G)   ]| \geq t \right) &= \p\left(| \Psi^{\Nc}_K(G) - \E[      \Psi^{\Nc}_K(G)   ]| \geq 2t \right)\\
    &\leq 4\exp\left(- \frac{4t^2}{2\binom{K}{2}} \right)\\
    &= 4\exp\left( -\frac{4t^2}{K(K-1)} \right).
\end{align}
For $t = K\sqrt{\log(n)}$ we have
$$\p\left(| \Psi^{\Nc\left(\frac{1}{2}, \frac{1}{4}\right)}_K(G) - \E[     \Psi^{\Nc\left(\frac{1}{2}, \frac{1}{4}\right)}_K(G)  ]| \geq t \right) \leq 2 \exp\left(-\frac{4K^2}{K(K-1)} \log(n)\right) = o(1).$$
Therefore, the following holds w.h.p as $n\to +\infty$
$$| \Psi^{\Nc\left(\frac{1}{2}, \frac{1}{4}\right)}_K(G) - \E[     \Psi^{\Nc\left(\frac{1}{2}, \frac{1}{4}\right)}_K(G) ]| < K\sqrt{\log(n)}.$$
Using Lemma \ref{Lemma 16.2} and similar computations, we can establish the above inequality for the distribution $\A$ as well, i.e, it holds w.h.p as $n \to +\infty$
$$| \Psi^{\A}_K(G) - \E[      \Psi^{\A}_K(G) ]| < K\sqrt{\log(n)}.$$
Combining the above with the result of Theorem \ref{Theorem 3.11} we have w.h.p as $n \to +\infty$
\begin{align}
    | \Psi^{\Nc\left(\frac{1}{2}, \frac{1}{4}\right)}_K(G) - \Psi^{\A}_K(G)| &\leq | \Psi^{\Nc\left(\frac{1}{2}, \frac{1}{4}\right)}_K(G) - \E[ \Psi^{\Nc\left( \frac{1}{2}, \frac{1}{4}\right)}_K(G)]| + | \E[ \Psi^{\Nc\left(\frac{1}{2}, \frac{1}{4}\right)}_K(G)] - \E[\Psi^{\A}_K(G)]| \\
    &+| \E[\Psi^{\A}_K(G)] - \Psi^{\A}_K(G)|\\
    &\leq K \sqrt{\log(n)} + O(K^{\frac{4}{3}} \log(n)^{\frac{7}{6}}) + K \sqrt{\log(n)}\\
    &= O(K^{\frac{4}{3}} \log(n)^{\frac{7}{6}}).
\end{align}
Therefore, it holds w.h.p as $n \to +\infty$
$$\Psi^{\A}_K(G)  =  \Psi^{\Nc\left(\frac{1}{2}, \frac{1}{4}\right)}_K(G)  +  O(K^{\frac{4}{3}} \log(n)^{\frac{7}{6}}).$$
Which concludes the proof of the second part of Theorem \ref{Theorem 3.11}.
It remains then to prove Lemma \ref{Lemma 16.2}.
\begin{proof}[Proof of Lemma \ref{Lemma 16.2}]
Since concentration inequalities around expected values are (asymptotically) invariant to shifting and/or multiplying the distribution $\A$ by constants (up to changes on the constants inside the exponential tail decay $e^{-\Theta(t^2)}$), we may then assume that the bounded distribution $\A$ is restricted to taking values in $[1, 2]$. In light of the proof of Theorem \ref{Theorem 4.8}, we first  establish that $h(G)\triangleq \Psi^{\A}_K(G)$ is $f$-certifiable with $f(s)=s$. If $h(G)\geq s$, then if $s > \binom{K}{2}$ let $I$ be the set of all edge weights in (one of) the densest subgraphs in $G$, and if $s\leq \binom{K}{2}$ let   $I$ be any $s$-subset of the edge weights in (one of) the densest subgraphs in $G$. Note that fixing the values of edges in $I$ and changing the rest of edges would still yield a a graph $G'$ s.t $h(G')\geq s$, therefore $h$ is $f$-certifiable and similarly, $3 \binom{K}{2}-h$ is $f$-certifiable by symmetry. We can then follow the roadmap of Theorem \ref{Theorem 4.8} to establish the desired concentration inequality.
\end{proof}